\documentclass[twoside,a4paper,reqno,11pt]{amsart}
\usepackage{amsfonts, amsbsy, amsmath, amsthm, amssymb, latexsym, verbatim, enumerate}
\usepackage{mathrsfs}
\usepackage[top=30mm,right=30mm,bottom=30mm,left=30mm]{geometry}

\usepackage{bm}
\usepackage[pdftex]{color,graphicx}

\makeatletter
\newcommand{\imod}[1]{\allowbreak\mkern4mu({\operator@font mod}\,\,#1)}
\makeatother

\newtheorem{propo}{Proposition}[section]

\newtheorem{defi}[propo]{Definition}

\newtheorem{lemma}[propo]{Lemma}
\newtheorem{corol}[propo]{Corollary}

\newtheorem{theor}[propo]{Theorem}

\newtheorem{examp}[propo]{Example}
\newtheorem{remar}[propo]{Remark}

\numberwithin{equation}{section}

\newcommand{\Ker}{\operatorname{Ker}}
\newcommand{\Aut}{{\mathrm {Aut}}}
\newcommand{\Span}{{\mathrm {Span}}}

\newcommand{\Irr}{{\mathrm {Irr}}}

\renewcommand{\Im}{{\mathrm {Im}}}
\newcommand{\Ind}{{\mathrm {Ind}}}
\newcommand{\diag}{{\mathrm {diag}}}

\newcommand{\Hom}{{\mathrm {Hom}}}

\newcommand{\Gal}{{\it Gal}}

\newcommand{\CC}{{\mathbb C}}
\newcommand{\RR}{{\mathbb R}}
\newcommand{\QQ}{{\mathbb Q}}
\newcommand{\ZZ}{{\mathbb Z}}

\newcommand{\SSS}{{\sf S}}

\newcommand{\AAA}{{\sf A}}
\newcommand{\FF}{{\mathbb F}}
\newcommand{\FQ}{\mathbb{F}_{q}}

\newcommand{\AC}{\mathcal{A}}

\newcommand{\EC}{\mathcal{E}}
\newcommand{\GC}{\mathcal{G}}
\newcommand{\GCD}{\mathcal{G}^*}
\newcommand{\GD}{G^*}
\newcommand{\FD}{F^*}

\newcommand{\LC}{\mathcal{L}}
\newcommand{\LBF}{\mathbf{L}}
\newcommand{\LFR}{\mathfrak{L}}

\newcommand{\TC}{\mathcal{T}}
\newcommand{\UC}{\mathcal{U}}
\newcommand{\CL}{\mathcal{C}}
\newcommand{\OC}{\mathcal{O}}
\newcommand{\XC}{\mathcal{X}}

\newcommand{\PC}{\mathcal{P}}

\newcommand{\ZB}{\mathbf{Z}}
\newcommand{\CB}{\mathbf{C}}
\newcommand{\NB}{\mathbf{N}}
\newcommand{\varep}{\varepsilon}
\newcommand{\eps}{\epsilon}
\newcommand{\lam}{\lambda}

\newcommand{\al}{\alpha}
\newcommand{\gam}{\gamma}
\newcommand{\om}{\varpi}

\newcommand{\Om}{\Omega}

\newcommand{\GL}{\mathrm {GL}}
\newcommand{\SL}{\mathrm {SL}}
\newcommand{\GU}{\mathrm {GU}}
\newcommand{\SU}{\mathrm {SU}}

\newcommand{\DC}{D^{\circ}}
\newcommand{\Stab}{{\mathrm {Stab}}}

\newcommand{\hs}{\hat{s}}
\newcommand{\tbe}{\tilde{\beta}}
\newcommand{\reg}{{\mathbf{reg}}}
\newcommand{\cl}{\mathfrak{l}}
\newcommand{\bnu}{{\boldsymbol{\nu}}}
\newcommand{\bmu}{{\boldsymbol{\mu}}}

\newcommand{\blam}{{\boldsymbol{\lambda}}}
\newcommand{\bone}{{\boldsymbol{1}}}

\newcommand{\tw}[1]{{}^#1\!}

\newcommand{\SR}{\tw* R}
\newcommand{\sgn}{{\mathbf {sgn}}}

\begin{document}

\title{Character Levels and Character Bounds}

\author{Robert M. Guralnick}
\address{Department of Mathematics, University of Southern California,
Los Angeles, CA 90089-2532, USA}
\email{guralnic@math.usc.edu}

\author{Michael Larsen}
\address{Department of Mathematics\\
    Indiana University \\
    Bloomington, IN 47405\\
    U.S.A.}
\email{larsen@math.indiana.edu}

\author{Pham Huu Tiep}
\address{Department of Mathematics\\
    University of Arizona\\
    Tucson, AZ 85721\\
    U. S. A.} 
\email{tiep@math.arizona.edu}

\thanks{The first author was partially supported by the NSF
grant DMS-1600056.}
\thanks{The second author was partially supported by the NSF 
grant DMS-1401419.}
\thanks{The third author was partially supported by the NSF grant DMS-1201374, the Simons Foundation
Fellowship 305247, and a Clay Senior Scholarship.}

\begin{abstract} 
We develop the concept of character level for the complex irreducible characters of finite, general or special,
linear and unitary groups. We give characterizations of the level of a character in terms of its Lusztig's label and in terms of 
its degree. Then we prove explicit upper bounds for character values at  elements with not-too-large
centralizers and derive upper bounds on the covering number and mixing time of random walks corresponding
to these conjugacy classes.  We also characterize the level of the character in terms of certain dual pairs 
and prove explicit exponential character bounds
for the character values, provided that the level is not too large.
\end{abstract}

\maketitle

\tableofcontents

\section{Introduction}

It is well known that the complex irreducible characters of the symmetric group $\SSS_n$ are indexed by partitions 
$\lambda = (\lambda_1,\ldots,\lambda_r)$
of $n$.  If $\lambda_2,\ldots,\lambda_r$ are fixed and $\lambda_1$ (and therefore $n$) goes to infinity, by the hook length formula, the degree of the
character $\chi^\lambda$ is given by a polynomial in $n$ of degree $d := \lambda_2+\cdots+\lambda_r$.  For instance, $d=0$ corresponds to the partition $(n)$
and therefore to the trivial character; $d=1$ corresponds to $(n-1,1)$ and therefore to the standard representation of $\SSS_n$; $d=2$ corresponds to
$(n-2,2)$ and $(n-2,1,1)$, each of degree quadratic in $n$, and each appearing in the tensor square of the standard representation.  It makes sense, therefore, to sort the irreducible characters of $\SSS_n$, at least those for which $n$ is large compared to $d$, by their level which 
is defined to be $d$.  If we consider not $\SSS_n$ but ${\sf A}_n$, then all irreducible characters of $\SSS_n$ of low level compared to $n$ restrict to irreducible characters of ${\sf A}_n$, and all irreducible characters of ${\sf A}_n$ of low degree arise in this way.  Replacing ${\sf A}_n$ by a Schur cover would not change this picture, since the minimal degree of a projective representation of ${\sf A}_n$ which does not lift to a linear representation of ${\sf A}_n$ grows exponentially in $n$.

Something similar happens for finite simple, or quasisimple, groups of Lie type.
For instance, a glance at the list of character degrees of the finite special linear group $G = \SL_n(q)$, or of the
finite special unitary group $G = \SU_n(q)$, reveals that $G$ has $1$ irreducible character of degree $1$ (``zero level''), 
then roughly $q$ characters of degree roughly $q^{n-1}$ (``first level''), then roughly $q^2$ characters of degree roughly $q^{2n-4}$ and 
$q^{2n-3}$ (``second level''), and so on. 
Furthermore, the characters of degree roughly $q^{n-1}$ are known as {\it Weil representations} -- these are characters of 
finite analogues of the representations constructed by A. Weil \cite{W} for classical groups over local fields. Weil characters
for finite classical groups were constructed in \cite{BRW,Hw,Is2,Se,Wa} and have many remarkable properties, see e.g. \cite{Gr}.  
For many families of finite classical groups, they prove to be the irreducible characters of smallest degree larger than one
(even in the context of Brauer characters), see \cite{BK,GMST,GT1,GT2,HM,TZ1}. A thorough understanding of Weil characters and characters 
of the next levels plays a major role in recent solutions of several long-standing problems, see e.g. \cite{LBST1}.  In many applications, it happens
 that characters of high enough level display a generic behavior and can be handled by 
uniform arguments, whereas the ones of low level can only be treated individually and after their explicit identification and construction.

\smallskip
But how should one define the {\it level} of irreducible characters $\chi \in \Irr(G)$ for a finite group $G$ of Lie type 
defined over $\FF_q$?  There appear to be a few possible approaches, one using the Deligne-Lusztig theory \cite{L}, particularly the
notions of unipotent support \cite{Geck,GM,L} and wave front sets \cite{K}, and another  utilizing restrictions to ``nice'' subgroups,
cf. \cite{GMST,GT2,GH,T2,TZ2}. At present, our attempts in exploring the first approach yield only some partial and asymptotic results.
We also note the recent paper \cite{BLST} in which strong asymptotic bounds on character values have been established for elements 
with suitable centralizer in $G$.
The second approach can lead to more explicit results, but so far only for characters of first (and possibly second) level. 

\smallskip
The goal of this paper is to develop the concept of character level for complex irreducible characters of finite, general or special,
linear and unitary groups, that is for $\GL_n(q)$, $\GU_n(q)$, $\SL_n(q)$, and $\SU_n(q)$, where $q$ is a prime power. 
We will use the notation $\GL^\eps$ to denote $\GL$ when $\eps = +$ and $\GU$ when $\eps = -$, and similarly for $\SL^\eps$. 
Let $V = \FF_Q^n$ be the natural module of $G \in \{\GL^\eps_n(q),\SL^\eps_n(q)\}$, where 
$Q = q$ when $\eps = +$, and $Q = q^2$ when $\eps = -$. It is known that the class function
\begin{equation}\label{weil-def}
  \tau:g \mapsto \eps^n(\eps q)^{\dim_{\FF_Q}\Ker(g-1_V)}
\end{equation}  
is actually a (reducible) character of $G$ (see e.g. \cite{Ge}). 
The {\it true level $\cl^*(\chi)$} of an irreducible character $\chi \in \Irr(G)$ is then defined to be 
the smallest non-negative integer $j$ such that $\chi$ is an irreducible constituent of $\tau^j$; and the {\it level $\cl(\chi)$} is 
the smallest non-negative integer $j$ such that $\lam\chi$ is an irreducible constituent of $\tau^j$ for some character $\lam$ of 
degree $1$ of $G$.

Our first main results, Theorems \ref{main1-gl} and \ref{main1-gu}, completely determine irreducible characters of $G$ of any 
given level $j$.  Using this characterization, we then establish the following upper and lower bounds for the degree of characters 
in terms of their level. These bounds show in particular
that if $\chi(1)$ is not too large, then $\cl(\chi)$ is $\left\lceil \dfrac{\log_q \chi(1)}{n} \right\rceil$, agreeing with intuitive understanding
of the character level.

\begin{theor}\label{main-degree}
Let $n \geq 2$, $\eps = \pm$, and $G = \GL^\eps_n(q)$. Set $\kappa_+ = 1$ and $\kappa_- = 1/2$. Let $\chi \in \Irr(G)$ have 
level $j = \cl(\chi)$. Then the following statements hold.
\begin{enumerate}[\rm(i)]
\item $\kappa_\eps q^{j(n-j)} \leq \chi(1) \leq q^{nj}$.
\item If $j \geq n/2$, then $\chi(1) > (9/16)(q-1)q^{n^2/4-1}$ if $\eps = +$ and $\chi(1) \geq (q-1)q^{n^2/4-1}$ if $\eps = -$.
In particular, $\chi(1) > q^{n^2/4-2}$ if $\cl(\chi) \geq n/2$.
\item If $n \geq 7$ and  $\lceil (1/n)\log_q \chi(1) \rceil < \sqrt{n-1}-1$ then 
$$\cl(\chi) = \left\lceil \frac{\log_q \chi(1)}{n} \right\rceil.$$
\end{enumerate}
\end{theor}

\begin{theor}\label{slu-degree}
Let $n \geq 2$, $\eps = \pm$, and $S = \SL^\eps_n(q)$. Set $\sigma_+ = 1/(q-1)$ and $\sigma_- = 1/2(q+1)$. Let $\varphi \in \Irr(S)$ have 
level $j = \cl(\varphi)$. Then the following statements hold.
\begin{enumerate}[\rm(i)]
\item $\sigma_\eps q^{j(n-j)} \leq \varphi(1) \leq q^{nj}$.
\item If $j \geq n/2$, then $\varphi(1) > q^{n^2/4-2}/(q-\eps) \geq (2/3)q^{n^2/4-3}$.
\item If $n \geq 7$ and  $\lceil (1/n)\log_q \varphi(1) \rceil < \sqrt{n-1}-1$ then 
$$\cl(\varphi) = \left\lceil \frac{\log_q \varphi(1)}{n} \right\rceil.$$
\end{enumerate}
\end{theor}

Next we prove exponential bounds for character values at the group elements with not too large
centralizers. The first bound shows that character values $\chi(g)$ of ``almost regular'' elements in 
$G$ are arbitrarily small exponentially in comparison to $\chi(1)$, as long as $n$ is sufficiently large.

\begin{theor}\label{main-bound1}
There is an explicit function $h=h(C,m):\RR_{\geq 1} \times \ZZ_{\geq 0} \to \RR_{\geq 1}$
such that, for any $C \in \RR_{\geq 1}$, $m \in \ZZ_{\geq 0}$, the following statement holds.
For any prime power $q$, any $\eps = \pm$, any $G= \GL^\eps_n(q)$ or $\SL^\eps_n(q)$ with $n \geq h(C,m)$, 
any $\chi \in \Irr(G)$, and any $g \in G$ with $|\CB_{\GL^\eps_n(q)}(g)| \leq q^{Cn}$,
$$|\chi(g)| \leq \chi(1)^{1/2^m}.$$ 
\end{theor}

The second bound shows that if the centralizer of $g \in G$ is not 
too large, then the character values $|\chi(g)|$ can be bounded away from $\chi(1)$ exponentially (and explicitly).

\begin{theor}\label{main-bound2}
Let $q$ be any prime power and let $G = \GL^\eps_n(q)$ or $\SL^\eps_n(q)$ with $\eps = \pm$. 
Suppose that $g \in G$ satisfies $|\CB_{\GL^\eps_n(q)}(g)| \leq q^{n^2/12}$. Then
$$|\chi(g)| \leq \chi(1)^{8/9}$$
for all $\chi \in \Irr(G)$.
\end{theor}

In our next results, we give another characterization of characters of level $j$ in terms of certain dual pairs \cite{Hw}, 
which allows us to obtain strong (and explicit) exponential character bounds for {\it all} elements in the group, provided that $j$ is not large.
For any $g \in \GL_n(q) = \GL(A)$, let $\delta(g)$ denote the largest dimension of eigenspaces of $g$ on 
$A \otimes_{\FQ}\overline\FF_q$.

\begin{theor}\label{gl-dual}
Let $q$ be any prime power, $G = \GL(A) \cong \GL_n(q)$ with $A = \FQ^n$, $S = \GL(B) \cong \GL_j(q)$ with 
$B = \FQ^j$ and $1 \leq j \leq n$. Let $V := A \otimes_{\FQ} B$ and consider the (reducible) Weil character 
$\tau$ of $\GL(V) \cong \GL_{nj}(q)$ as in \eqref{weil-def}. Via the natural action of $G \times S$ on 
$V$, we can view $\tau$ as a character of $G \times S$, and decompose 
$$\tau|_{G \times S} = \sum_{\al \in \Irr(S)}D_\al \otimes \al,$$ 
where $D_\al$ is either $0$ or a character of $G$.
Then the following statements hold.
\begin{enumerate}[\rm(i)]
\item If $\cl^*(\al) < 2j-n$ and $D_\al \neq 0$, then all irreducible constituents of $D_\al$ have true level smaller than $j$.
\item If $\cl^*(\al) \geq 2j-n$, then there is a unique irreducible character $\DC_\al \in \Irr(G)$ of true level $j$ such that
$D_\al - \DC_\al$ is either zero, or a sum of irreducible characters of $G$ of true level less than $j$.
\item The map $\al \mapsto \DC_\al$ yields a canonical bijection between 
$\{ \al \in \Irr(\GL_j(q)) \mid \cl^*(\al) \geq 2j-n \}$ 
and $\{ \theta \in \Irr(\GL_n(q)) \mid \cl^*(\theta) = j \}$.
\item If $\cl(\chi) \leq \sqrt{(8n-17)/12}-1/2$ for $\chi \in \Irr(G)$, then $|\chi(g)| < 1.76\chi(1)^{1-1/n}$ for all
$g \in G \smallsetminus \ZB(G)$. Moreover, if $\cl(\chi) \leq (\sqrt{12n-59}-1)/6$ for $\chi \in \Irr(G)$, then 
$$|\chi(g)| < 1.76\chi(1)^{\max(1-1/2\cl(\chi),\delta(g)/n)}$$ 
for all $g \in G$.
\item The same statements as in {\rm (iv)} hold if we replace $G$ by $\SL_n(q)$.
\end{enumerate}
\end{theor}

\begin{theor}\label{gu-dual}
Let $q$ be any prime power, $G = \GU(A) \cong \GU_n(q)$ with $A = \FF_{q^2}^n$, $S = \GU(B) \cong \GU_j(q)$ with 
$B = \FF_{q^2}^j$ and $1 \leq j \leq n$. Let $V := A \otimes_{\FF_{q^2}} B$ and consider the (reducible) Weil character 
$\tau$ of $\GU(V) \cong \GU_{nj}(q)$ as in \eqref{weil-def}. Via the natural action of $G \times S$ on 
$V$, we can view $\tau$ as a character of $G \times S$, and decompose 
$$\tau|_{G \times S} = \sum_{\al \in \Irr(S)}D_\al \otimes \al,$$ 
where $D_\al$ is either $0$ or a character of $G$.
Then the following statements hold.
\begin{enumerate}[\rm(i)]
\item There is a bijection $\Theta$ between $\{ \theta \in \Irr(\GU_n(q)) \mid \cl^*(\theta) =j \}$ 
and $\{ \al \in \Irr(\GU_j(q)) \mid \cl^*(\al) \geq 2j-n \}$.
\item If $1 \leq j \leq \sqrt{n-3/4}-1/2$ 
and $\al \in \Irr(S)$, then there is a unique irreducible character $\DC_\al \in \Irr(G)$ of true level $j$ such that
$D_\al - \DC_\al$ is either zero, or a sum of irreducible characters of $G$ of true level less than $j$.
\item If $\cl(\chi) \leq \sqrt{n-3/4}-1/2$ for $\chi \in \Irr(G)$, then $|\chi(g)| < 2.43\chi(1)^{1-1/n}$ for all
$g \in G \smallsetminus \ZB(G)$. Moreover, if $\cl(\chi) \leq \sqrt{n/2-1}$ for $\chi \in \Irr(G)$, then 
$$|\chi(g)| < 2.43\chi(1)^{\max(1-1/2\cl(\chi),\delta(g)/n)}$$ 
for all $g \in G$.
\item The same statements as in {\rm (iii)} hold if we replace $G$ by $\SU_n(q)$.
\end{enumerate}
\end{theor}

Note that the exponent $1-1/n$ in the character bounds in Theorems \ref{gl-dual} and \ref{gu-dual} is optimal, see Example \ref{glu-ex}(i).  
Similarly, if $g \in \GL_n(q)$, respectively $g \in \GU_n(q)$ is close to be a scalar matrix, i.e. $\delta(g)$ is very close to $n$, then 
the exponent $\max(1-1/2\cl(\chi),\delta(g)/n)$ in Theorems \ref{gl-dual} and \ref{gu-dual} is again optimal, see Example \ref{glu-ex}(ii).
Furthermore, the bijections in Theorems \ref{gl-dual}(iii) and \ref{gu-dual}(i) are explicitly described (in terms of character labels) 
in Corollaries \ref{gl-dual-bij} and \ref{gu-dual-bij}. See also Corollary \ref{slu-dual} for a bijection between irreducible characters of level
$j < n/2$ of $\SL_n(q)$, respectively $\SU_n(q)$, and irreducible characters of $\GL_j(q)$, respectively $\GU_j(q)$. These bijections
are canonical, see Remark \ref{canonical}, and may be helpful in certain situations (for instance, when one would like to 
control the action of outer automorphisms and Galois automorphisms on irreducible characters). 

\smallskip
We expect our results on character levels and character bounds to be useful in various applications.
As an immediate application, we derive upper bounds on the covering number and mixing time of random walks corresponding to certain 
conjugacy classes in $\SL^\eps_n(q)$, see Corollary \ref{slu-mix}. Our next result Corollary \ref{slu-dist} 
is concerned with a conjecture of Avni and Shalev on almost uniform distribution of the commutator map on simple groups of
Lie type.
%
%
%
As another application, Theorem \ref{weil} gives a decomposition
of the restriction of the Weil representation to a dual pair $\GL^\eps_m(q) \times \GL^\eps_n(q)$ -- when $q$ is sufficiently large
this is one of the main results of \cite{S2}.  We also observe a parity phenomenon for the characters of finite unitary groups,
see Corollary \ref{parity}.

Other finite classical groups will be considered in a sequel to this paper.

\smallskip
For any finite group $G$, $\Irr(G)$ denotes the set of complex irreducible characters of $G$,
$\reg_G$ denotes the regular character of $G$ (i.e. $\reg_G(g)$ equals $|G|$ if $g = 1$ and $0$ if 
$1 \neq g \in G$), and $1_G$ denotes the principal character of $G$. For a subgroup $H$ of a finite group $G$,
a class function $\alpha$ of $H$, and class functions $\beta,\gam$ of $G$, $\gam|_H$ denotes the restriction of
$\gam$ to $H$, $\Ind^G_H(\alpha)$ denotes the induced class function on $G$, and $[\beta,\gam]_G$ 
denotes the usual scalar product of class functions. We will say that $\chi \in \Irr(G)$ is an irreducible constituent of 
a class function $\al$ on $G$ if $[\al,\chi]_G \neq 0$. Other notation is standard.

\section{Preliminaries}

Recall that the complex irreducible characters of the symmetric group $\SSS_n$ are labeled by partitions 
$\lam \vdash n$: $\chi = \chi^\lam$. In particular, $\chi^{(n)} = 1_{\SSS_n}$. As usual, we write 
$\lam = (\lam_1, \lam_2, \ldots ,\lam_r) \vdash n$ if $\lam_1 \geq \lam_2 \geq \ldots \geq \lam_r \geq 0$ and $\sum^r_{i=1}\lam_i = n$.

\begin{lemma}\label{sym}
Consider $\gam = (\gam_1, \gam_2, \ldots ,\gam_r) \vdash n+m$ and the Young subgroup 
$Y = \SSS_n \times \SSS_m$ of $S := \SSS_{n+m}$ for some $m, n \geq 1$.
\begin{enumerate}[\rm(i)]
\item $\chi^\gam$ can occur in $\Ind^S_Y(\chi^\al \otimes \chi^{(m)})$ for some $\al \vdash n$ precisely when $\gam_1 \geq m$.

\item If $\gam_1 = m$, then $\chi^\gam$ occurs in $\Ind^S_Y(\chi^\al \otimes \chi^{(m)})$ exactly when $\al = (\gam_2, \gam_3, \ldots, \gam_r)$, 
in which case it occurs with multiplicity one.
\end{enumerate}
\end{lemma}

\begin{proof}
According to Young's rule \cite[2.8.2]{JK}, $\chi^\gam$ can occur in $\Ind^S_Y(\chi^\al \otimes \chi^{(m)})$ for $\al = (\al_1, \al_2, \ldots ,\al_r) \vdash n$
precisely when 
\begin{equation}\label{sym-1}
  \al_r \leq \gam_r \leq \al_{r-1} \leq \gam_{r-1} \leq \al_{r-2} \leq \ldots \leq \al_2 \leq \gam_2 \leq  \al_1 \leq \gam_1.
\end{equation}   
In particular,
$$n = \sum^r_{i=1}\al_i \geq \sum^{r-1}_{i=1}\al_i \geq \sum^r_{i=2}\gam_i = (n+m)-\gam_1,$$
and so $\gam_1 \geq m$. Moreover, if $\gam_1 = m$, then we get $\al_r = 0$ and $\al = (\gam_2, \gam_3, \ldots , \gam_r)$.

Conversely, suppose that $\gam_1 \geq m$. Then we can find $r$ integers $t_1, t_2 , \ldots, t_r \geq 0$ such that $\sum^r_{i=1}t_i = m$ and
$$t_1 \leq \gam_1-\gam_2,~~t_2 \leq \gam_2-\gam_3,~\ldots ,t_{r-1} \leq \gam_{r-1}-\gam_r,~t_r \leq \gam_r.$$ 
Setting $\al_i = \gam_i-t_i$, we see that $\al = (\al_1, \al_2, \ldots ,\al_r) \vdash n$ and $\al$ satisfies \eqref{sym-1}, and so 
$\chi^\gam$ occurs in $\Ind^S_Y(\chi^\al \otimes \chi^{(m)})$. The multiplicity-one claim also follows from Young's rule.
\end{proof}

The following well-known observation is essentially due to Brauer:

\begin{lemma}\label{value}
Let $\Theta$ be a generalized character of a finite group $G$ which takes exactly $N$ different values $a_0 = \Theta(1)$, 
$a_1, \ldots ,a_{N-1}$ on $G$. Suppose also that $\Theta(g) \neq \Theta(1)$ for all $1 \neq g \in G$. Then every irreducible
character $\chi$ of $G$ occurs as an irreducible constituent of $\Theta^k$ for some $0 \leq k \leq N-1$.
\end{lemma}

\begin{proof}
Consider any $\chi \in \Irr(G)$. By assumption,
$$[\chi,\prod^{N-1}_{i=1}(\Theta-a_i \cdot 1_G)]_G = \frac{\chi(1)}{|G|}\prod^{N-1}_{i=1}(a_0-a_i) \neq 0,$$
whence $[\chi,\Theta^k]_G \neq 0$ for some $0 \leq k \leq N-1$.
\end{proof} 

\begin{lemma}\label{separation}
Let $G$ be a finite group and let $\XC_0, \XC_1, \ldots, \XC_n$ be $n+1$ disjoint (possibly empty) subsets of 
$\Irr(G)$.  Let $\al_0, \al_1, \ldots,\al_n$ be (not necessarily irreducible) complex characters of $G$ and 
$\beta_0,\beta_1, \ldots ,\beta_n$ be generalized characters of $G$ such that
\begin{enumerate}[\rm(a)]
\item $\Span_{\ZZ}(\al_0,\al_1,\ldots,\al_j) =  \Span_{\ZZ}(\beta_0,\beta_1,\ldots,\beta_j)$;
\item Each $\chi \in \XC_j$ occurs in $\sum^j_{i=0}\al_i$; and 
\item All irreducible constituents of $\beta_j$ belong to $\cup^j_{i=0}\XC_i$
\end{enumerate}
for all $0 \leq j \leq n$. Then for all $0 \leq j \leq n$, $\XC_j$ is precisely the set of irreducible characters 
of $G$ that occur in $\al_j$ but not in $\sum^{j-1}_{i=0}\al_i$.
\end{lemma}

\begin{proof}
We proceed by induction on $0 \leq j \leq m$. For $j=0$, any $\chi \in \XC_0$ occurs in $\al_0$ by (b). Conversely,
any irreducible constituent of $\al_0$ belongs to $\XC_0$ by (a) and (c).

For the induction step, consider any $\chi \in \XC_j$. By (b), $\chi$ occurs in $\sum^j_{i=0}\al_i$. If, moreover,
$\chi$ occurs in $\sum^{j-1}_{i=0}\al_i$, then $\chi$ belongs to $\cup^{j-1}_{i=0}\XC_i$ by (a) and (c), a contradiction.
Conversely, assume that $\chi \in \Irr(G)$ occurs in $\al_j$ but not in $\sum^{j-1}_{i=0}\al_i$. By (a), $\chi$ occurs in 
$\beta_i$ for some $0 \leq i \leq j$ and so $\chi \in \XC_k$ for some $0 \leq k \leq i$ by (c). Hence $k=j$ by (b), as desired.
\end{proof}

\begin{lemma}\label{orbits}
Let $(G,Q)$ be either $(\GL_n(q),q)$ or $(\GU_n(q),q^2)$, and let $V = \FF_Q^n$ denote the natural module for $G$.
Then, for any $1 \leq j \leq n$, the number $N_j$ of $G$-orbits on the set $\Omega_j$ of ordered  $j$-tuples $(v_1, \ldots, v_j)$
with $v_i \in V$ is at most $8q^{j^2/4}$ in the first case, and at most $2q^{j^2}$ in the second case.  
\end{lemma}

\begin{proof}
(i) Consider $U = \FF_Q^j$ with a fixed basis $(e_1, \ldots ,e_j)$. Then there is a natural bijection between $\Omega_j$
and $\Hom(U,V)$: any $\om = (v_1, \ldots,v_j)$ corresponds to $f=f_\om \in \Hom(U,V)$ with $f(e_i) = v_i$. Suppose that 
$\om' = g(\om)$ for some $g \in G$. Then $f_{\om'} = gf_\om$ and $\Ker(f_{\om'}) = \Ker(f_\om)$. Moreover, in the case
$Q=q^2$, the Hermitian forms of $V$ restricted to $f_\om(V)$ and $f_{\om'}(V)$ have the same Gram matrices in
the bases $(f_\om(u_1), \ldots,f_\om(u_k))$ and $(f_{\om'}(u_1), \ldots ,f_{\om'}(u_k))$, if $(u_1, \ldots,u_k)$ is a basis 
of  $U/\Ker(f_\om)$. Conversely, assume that $f_\om$ and $f_{\om'}$ have the same kernel $W$ for some $\om,\om' \in \Omega_j$.
Again we fix a basis $(u_1, \ldots,u_k)$ of $U/W$. If $Q = q^2$, assume in addition that 
the Hermitian forms of $V$ restricted to $f_\om(V)$ and $f_{\om'}(V)$ have the same Gram matrices in
the bases $(f_\om(u_1), \ldots,f_\om(u_k))$ and $(f_{\om'}(u_1), \ldots ,f_{\om'}(u_k))$. By Witt's lemma \cite[p. 81]{A}, 
there is some $g \in G$ such that $g(f_\om(u_i)) = f_{\om'}(u_i)$ for all $1 \leq i \leq k$. Hence $f_{\om'} =gf_\om$ and 
so $\om' = g(\om)$. Note that there are at most $q^{k^2}$  of possibilities for the Gram matrices in
the basis $(f_\om(u_1), \ldots,f_\om(u_k))$.

\smallskip
(ii) Suppose that $Q = q$. Our arguments in (i) show that $N_j$ is just the total number of subspaces $W$ in $U$, i.e.
\begin{equation}\label{gl-mult}
  N_j = \sum_{i=0}^j \binom ji_q, 
\end{equation}   
where $\binom ji_q$ denotes the Gaussian binomial coefficient:
$$\binom ji_q := \frac{\prod^{i-1}_{t=0}(q^j-q^t)}{\prod^{i-1}_{t=0}(q^i-q^t)}.$$
%
%
%
By \cite[Lemma 4.1(i)]{LMT} we have
$$\binom ji_q  = \frac{\prod^{i-1}_{t=0}(q^j-q^t)}{\prod^{i-1}_{t=0}(q^i-q^t)} = q^{i(j-i)}\frac{\prod^{i-1}_{t=0}(1-1/q^{j-t})}{\prod^{i-1}_{t=0}(1-1/q^{i-t})} < 
    (32/9)q^{i(j-i)}.$$
and so $N_j  < (32/9)\sum^j_{i=0}q^{i(j-i)}$.   
If $j=2j_0+1 \geq 1$, then 
$$\frac{32}{9}\sum^j_{i=0}q^{i(j-i)} < \frac{64}{9}q^{j_0(j_0+1)}\sum^{\infty}_{i=0}\frac{1}{q^{i(i+1)}} < 8q^{j^2/4},$$
yielding the claim. If $j=2j_0 \geq 2$, then 
$$\frac{32}{9}\sum^j_{i=0}q^{i(j-i)} < \frac{32}{9}q^{j_0^2}(1+2\sum^{\infty}_{i=1}\frac{1}{q^{i^2}}) < 8q^{j^2/4},$$
and we are done again.

Note that $N_j \geq q^{\lfloor j^2/4 \rfloor}$ as $G$ has at least $q^{\lfloor j^2/4 \rfloor}$ orbits on $j$-tuples $\om$ corresponding to 
$f_\om$ with $\dim \Ker(f_\om) = \lfloor j/2 \rfloor$.

\smallskip
(iii) Suppose that $Q = q^2$. 
Our arguments in (i) show that $N_j$ is at most the sum over $k$ of 
the total number of $(j-k)$-dimensional subspaces $W$ in $U$ weighted by 
a factor of $q^{k^2}$, i.e.
$$N_j \leq \sum_{i=0}^j  q^{(j-i)^2}\binom ji_Q
= \sum^j_{i=0}q^{(j-i)^2}\frac{\prod^{i-1}_{t=0}(q^{2j}-q^{2t})}{\prod^{i-1}_{t=0}(q^{2i}-q^{2t})}.$$ 
For $j-1 \geq i \geq 1$, by \cite[Lemma 4.1(i)]{LMT} we have
$$\frac{\prod^{i-1}_{t=0}(q^{2j}-q^{2t})}{\prod^{i-1}_{t=0}(q^{2i}-q^{2t})} = q^{2i(j-i)}\frac{\prod^{i-1}_{t=0}(1-1/q^{2(j-t)})}{\prod^{i-1}_{t=0}(1-1/q^{2(i-t)})} < 
    \frac{q^{2i(j-i)}}{1-q^{-2}-q^{-4}+q^{-10}} < (1.46)q^{2i(j-i)}$$
It follows that
$$N_j  < q^{j^2}+(1.46)\sum^{j}_{i=1}q^{j^2-i^2} < q^{j^2}(1+(1.46)\sum^{\infty}_{i=1}\frac{1}{q^{i^2}}) < 1.83q^{j^2}.$$   
Note that $N_j \geq q^{j^2}$ if $j \leq n/2$ as $G$ has $q^{j^2}$ orbits on linearly independent $j$-tuples, and in fact
$N_j = \prod^{j}_{i=1}(q^{2i-1}+1)$ if $1 \leq j \leq n/2$, as we show in Corollary \ref{gu-mult} (below).
\end{proof}

%
%

\begin{lemma}\label{z-identity}
For any $m \in \ZZ_{\geq 0}$ we have
$$t^m = \sum^m_{k=0}\binom{m}{k}_q \prod^{k-1}_{i=0}(t-q^i).$$
\end{lemma}

\begin{proof}
We proceed by induction on $m$, with trivial base case $m=0$. For the induction step $m \geq 2$, note that if $1 \leq k \leq m-1$ then 
\begin{equation}\label{binom1}
  \binom{m}{k}_q = \binom{m-1}{k}_q + q^{m-k}\binom{m-1}{k-1}_q.
\end{equation}  
Applying the induction hypothesis to the two indeterminates $t'=t/q$ and $q$, we have
$$\begin{aligned}\sum^m_{k=1}q^{m-k}\binom{m-1}{k-1}_q\prod^{k-1}_{i=0}(t-q^i) & 
      = (t-1)q^{m-1}\sum^m_{k=1}\binom{m-1}{k-1}_q\prod^{k-2}_{i=0}(t'-q^i)\\
      & = (t-1)q^{m-1}\sum^{m-1}_{j=0}\binom{m-1}{j}_q\prod^{j-1}_{i=0}(t'-q^i)\\
      & = (t-1)q^{m-1}t'^{m-1} = (t-1)t^{m-1}.
      \end{aligned}$$
Combined with \eqref{binom1}, it yields that
$$\begin{aligned}\sum^m_{k=0}\binom{m}{k}_q \prod^{k-1}_{i=0}(t-q^i) & = 
    \sum^{m-1}_{k=0}\binom{m-1}{k}_q \prod^{k-1}_{i=0}(t-q^i) + \sum^m_{k=1}q^{m-k}\binom{m-1}{k-1}_q\prod^{k-1}_{i=0}(t-q^i)\\
    & = t^{m-1} + (t-1)t^{m-1} = t^m.
    \end{aligned}$$    
\end{proof}

Let $\GC$ be a connected reductive algebraic group defined over a field $\FF$ of characteristic $p$ and let $F:\GC \to \GC$ be 
a Frobenius endomorphism. Let $\PC$ be a (not necessarily $F$-stable) parabolic subgroup of $\GC$ with a Levi 
subgroup $\LC$ which is $F$-stable. Then the {\it Lusztig induction} $R^\GC_{\LC \subset \PC}$ is defined and sends 
generalized characters $\psi$ of $L := \LC^F$ to generalized characters of $G := \GC^F$, see \cite[\S11]{DM2}. The character
formula for $R^\GC_{\LC \subset \PC}(\psi)$, see \cite[Proposition 12.2]{DM2} utilizes {\it Green functions}
$$Q^\GC_{\LC \subset \PC}:\GC^F_u \times \LC^F_u \to \ZZ$$ 
as defined in \cite[Definition 12.1]{DM2}, with $\GC^F_u$ and $\LC^F_u$ denoting the set of unipotent elements in $G$ and in $L$. 
If $\PC$ is $F$-stable in addition, then $R^\GC_{\LC \subset \PC}$ is just the {\it Harish-Chandra induction} $R^G_L$ (that first inflates 
any character of $L$ to a character of $P := \PC^F$ and then induces to $G$). It is known that the Lusztig induction 
is transitive (see \cite[Proposition 11.5]{DM2}); furthermore, it satisfies the ``Mackey formula''  in some cases,
including the cases where $\PC$ is $F$-stable \cite[Theorem 5.1]{DM2}, or if one of the involved Levi subgroups is a maximal torus
\cite[Proposition 11.3]{DM2}, or if $\GC = \GL_n(\FF)$ \cite[Theorem 2.6]{DM1}. 
Moreover, when the Mackey formula holds, it implies that $R^\GC_{\LC \subset \PC}$ does not depend on the choice of 
$\PC$ (see \cite[p. 88]{DM2}). Hence, when we work with $\GC$ of type $\GL$, we will therefore write $R^G_L$ instead of 
$R^\GC_{\LC \subset \PC}$.

\medskip
We will record some properties of Lusztig induction.

\begin{lemma}\label{product1}
Let $\GC = \GC_1 \times \GC_2$ be a direct product of connected reductive algebraic groups with a Frobenius 
endomorphism $F:\GC \to \GC$ which stabilizes both $\GC_1$ and $\GC_2$. Let $\PC_i$ be a parabolic subgroup of 
$\GC_i$ with an $F$-stable Levi subgroup $\LC_i$ for $i = 1,2$, and let 
$\PC = \PC_1 \times \PC_2$, $\LC = \LC_1 \times \LC_2$.

\begin{enumerate}[\rm(i)]
\item Suppose that $u_i \in (\GC^F_i)_u$ and $v_i \in (\LC^F_i)_u$ for $i = 1,2$. Then 
$$Q^\GC_{\LC \subset \PC}(u_1u_2,v_1v_2) = Q^{\GC_1}_{\LC_1 \subset \PC_1}(u_1,v_1)Q^{\GC_2}_{\LC_2 \subset \PC_2}(u_2,v_2).$$

\item Suppose that $\gam_i$ is a generalized character of $\LC^F_i$ for $i = 1,2$. Then
$$R^\GC_{\LC \subset \PC}(\gam_1 \otimes \gam_2) = R^{\GC_1}_{\LC_1 \subset \PC_1}(\gam_1) \otimes 
    R^{\GC_2}_{\LC_2 \subset \PC_2}(\gam_2).$$
\end{enumerate}
\end{lemma}

\begin{proof}
(i) Let $\LBF:\GC \to \GC$ denote the {\it Lang map} $\LBF(g)=g^{-1}F(g)$, and let $\UC$ be the unipotent radical of 
$\PC$. Then 
$(g,l) \in \GC^F \times \LC^F$ acts on $\LBF^{-1}(\UC)$ via $x \mapsto gxl$, and this turns the $\ell$-adic cohomology group 
$H^j_c(\LBF^{-1}(\UC),\bar{\QQ}_{\ell})$ into a $\GC^F$-module-$\LC^F$ for all $j \geq 0$, where $\ell \neq p$ is a fixed prime. Next, 
\begin{equation}\label{for-gf}
  Q^\GC_{\LC \subset \PC}(u,v) = \frac{1}{|\LC^F|}\LFR((u,v),\LBF^{-1}(\UC)),
\end{equation}  
where the {\it Lefschetz number} $\LFR((u,v),\LBF^{-1}(\UC))$ is the trace of $(u,v)$ acting on 
$$H^*_c(\LBF^{-1}(\UC))=\sum_{j \geq 0}(-1)^j H^j_c(\LBF^{-1}(\UC),\bar{\QQ}_{\ell}),$$
see \cite[10.3,12.1]{DM2}.

In our case, $\UC=\UC_1 \times \UC_2$, where $\UC_i$ is the unipotent radical of $\PC_i$ for $i = 1,2$. It follows that
$\LBF^{-1}(\UC) = \LBF^{-1}(\UC_1) \times \LBF^{-1}(\UC_2)$, and so 
$$\LFR((u_1u_2,v_1v_2),\LBF^{-1}(\UC)) = \LFR((u_1,v_1),\LBF^{-1}(\UC_1))\cdot \LFR((u_2,v_2),\LBF^{-1}(\UC_2))$$
by \cite[Proposition 10.9(ii)]{DM2}. Together with \eqref{for-gf}, this implies the claim.

(ii) By \cite[Proposition 11.2]{DM2},  
\begin{equation}\label{for-rgl}
  (R^\GC_{\LC \subset \PC}\gam)(g) =  \frac{1}{|\LC^F|}\sum_{l \in \LC^F}\LFR((g,l),\LBF^{-1}(\UC))\gam(l^{-1})
\end{equation}  
for any generalized character $\gam$ of $\GC^F$. Applying this formula to $\gam:=\gam_1 \otimes \gam_2$ and using 
\cite[Proposition 10.9(ii)]{DM2} again (also noting that $\gam((xy)^{-1}) = \gam(x^{-1}y^{-1})$ for all $x,y \in \GC^F$), we 
obtain the claim. 
\end{proof}

In view of Lemma \ref{product1} and the discussion prior to it, when we work with $\GC$ a direct product of groups of type $\GL$, 
we can also write $R^G_L$ instead of $R^\GC_{\LC \subset \PC}$.

\begin{corol}\label{product2}
Let $\GC = \GC_1 \times \GC_2$ be a direct product of connected reductive algebraic groups with a Frobenius 
endomorphism $F:\GC \to \GC$ which stabilizes both $\GC_1$ and $\GC_2$. Let $\PC_1$ be a parabolic subgroup of 
$\GC_1$ with an $F$-stable Levi subgroup $\LC_1$, and let 
$\PC = \PC_1 \times \GC_2$, $\LC = \LC_1 \times \GC_2$. Suppose that $\gam$ is a generalized character of $\LC^F_1$ and 
$\delta$ is a generalized character of $\GC_2^F$. Then
$$R^\GC_{\LC \subset \PC}(\gam \otimes \delta) = R^{\GC_1}_{\LC_1 \subset \PC_1}(\gam) \otimes \delta.$$
\end{corol}

\begin{proof}
Note that \eqref{for-rgl} applied to $R^{\GC_2}_{\LC_2 \subset \PC_2}(\delta)$, where $\PC_2 := \LC_2 := \GC_2$,
yields $R^{\GC_2}_{\LC_2 \subset \PC_2}(\delta) = \delta$.
Now the statement follows by applying Lemma \ref{product1}(ii).
\end{proof}

Note that \cite[Lemma 2.7(ii)]{GKNT} is a partial case of Corollary \ref{product2}.

\section{Character level in finite general linear groups}
Let $q$ be a prime power and let $G := G_n := \GL_n(q)$ with $n \geq 2$ and natural module $V = \langle e_1, \ldots, e_n \rangle_{\FQ}$.
Let $\tau_n$ denote the permutation character of the action of $G$ on the set of vectors of $V$, so that
\begin{equation}\label{for-tau}
  \tau_n(g) = q^{\dim_{\FQ}\Ker(g-1_V)}
\end{equation}  
for all $g \in G$.  Applying Lemma \ref{value} to $(G,\Theta) = (\GL_n(q),\tau_n)$, we get

\begin{corol}\label{value-gl}
Each irreducible character of $\GL_n(q)$ occurs as an irreducible constituent of $(\tau_n)^k$ for some $0 \leq k \leq n$.
\hfill $\Box$
\end{corol} 
 
In view of Corollary \ref{value-gl}, we can introduce

\begin{defi}\label{def:gl}
{\em Let $\chi \in \Irr(\GL_n(q))$. 
\begin{enumerate}[\rm(i)]
\item We say that $\chi$ has {\it true level $j$}, and write $\cl^*(\chi) = j$, if $j$ is the smallest non-negative integer such that $\chi$ is an irreducible 
constituent of $(\tau_n)^j$.
\item  We say that $\chi$ has {\it level $j$}, and write $\cl(\chi) = j$, if $j$ is the smallest non-negative integer such that $\chi\lam$ is an irreducible 
constituent of $(\tau_n)^j$ for some character $\lam \in \Irr(\GL_n(q))$ of degree $1$.
\end{enumerate}}
\end{defi} 

Fix some $1 \leq j \leq n-1$ and consider the parabolic subgroup 
$$P = \Stab_G(\langle e_1, \ldots ,e_j \rangle_{\FQ})$$ 
with unipotent radical $U$ and Levi subgroup $L = G_j \times G_{n-j}$, where $G_j$ is identified with 
$$\cap^{n}_{i=j+1}\Stab_P(e_i) \cong \GL(\langle e_1, \ldots ,e_j\rangle_{\FQ})$$ 
and $G_{n-j}$ is identified with
$$\cap^{j}_{i=1}\Stab_P(e_i) \cap \Stab_P(\langle e_{j+1}, \ldots e_n \rangle_{\FQ}) \cong \GL(\langle e_{j+1}, \ldots ,e_n\rangle_{\FQ}).$$ 

\begin{propo}\label{gl-induced}
In the above notation,
$$R^G_L\left(\reg_{G_j} \otimes 1_{G_{n-j}}\right) = \prod^{j-1}_{i=0}(\tau_n-q^i \cdot 1_G).$$
\end{propo}

\begin{proof}
Note that $G$ acts transitively on the set $\tilde\Omega$ of ordered $j$-tuples $(f_1, \ldots, f_j)$ of linearly independent vectors in $V$,
and the corresponding permutation character is $\pi := \prod^{j-1}_{i=0}(\tau_n-q^i \cdot 1_G)$. Since the stabilizer of 
$(e_1, \ldots ,e_j)$ is $H :=UG_{n-j} \lhd P$, we have 
$$\pi = \Ind^G_P(1_H) = \Ind^G_P(\alpha) = R^G_L\left(\reg_{G_j} \otimes 1_{G_{n-j}}\right),$$
where $\alpha$ is trivial at $U$ and equals to 
$$\Ind^{G_j \times G_{n-j}}_{G_{n-j}}(1_{G_{n-j}}) = \reg_{G_j} \otimes 1_{G_{n-j}}$$ 
as $P/U$-character. Hence the statement follows.
\end{proof}

As in \cite{KT2}, it is convenient for us to use 
the Dipper-James classification of complex irreducible characters of $G$, as described in \cite{J1}. 
Namely, every $\chi\in \Irr(G)$ can be written uniquely, up to a permutation of the pairs 
$(s_1,\lam_1),\dots,(s_m,\lam_m)$, in the form 
\begin{equation}\label{gl-1}
  \chi = S(s_1,\lam_1) \circ S(s_2,\lam_2) \circ \ldots \circ S(s_m,\lam_m).
\end{equation}
Here, $s_i \in \bar{\FF}_q^\times$ has degree $d_i$ over $\FF_q$, $\lam_i \vdash k_i$, 
$\sum^m_{i=1}k_id_i = n$, and the $m$ elements $s_i$ have pairwise distinct minimal polynomials over $\FQ$. 
In particular, $S(s_i,\lam_i)$ is an irreducible character of $\GL_{k_id_i}(q)$. 
Furthermore, there is a parabolic subgroup 
$P_\chi$ of $G$ with Levi subgroup $L_\chi = \GL_{k_1d_1}(q) \times \ldots \times \GL_{k_md_m}(q)$. 
The (outer) tensor product 
$$\psi := S(s_1,\lam_1) \otimes S(s_2,\lam_2) \otimes \ldots \otimes S(s_m,\lam_m)$$
is an $L_\chi$-character, and $\chi = R^{G}_{L_\chi}(\psi)$. 

Note that $S(1,\lam)$ with $\lam \vdash n$ is just the unipotent character of $\GL_n(q)$ labeled by $\lam$.
We will need the following well-known fact about the Harish-Chandra induction of unipotent characters of Levi subgroups
of $\GL_n(q)$ (see e.g. \cite[(3.5)]{J2}): 

\begin{lemma}\label{sym-gl}
Let $\al \vdash m$ and $\beta \vdash n$, and consider the Young subgroup $\SSS_m \times \SSS_n$ of $\SSS_{m+n}$ and the
Levi subgroup $G_m \times G_n$ of $G_{m+n} =\GL_{m+n}(q)$. Then
$$\Ind^{\SSS_{m+n}}_{\SSS_m \times \SSS_n}(\chi^\al \otimes \chi^\beta) = \sum_{\lam \vdash\,m+n}a_{\al\beta\lam}\chi^\lam,~~
    R^{G_{m+n}}_{G_m \times G_n}(S(1,\al) \otimes S(1,\beta)) = \sum_{\lam \vdash\,m+n}a_{\al\beta\lam}S(1,\lam),$$
where $a_{\al\beta\lam}$ is the Littlewood-Richardson coefficient.
\hfill $\Box$
\end{lemma}

\begin{propo}\label{gl-const}
For $G = \GL_n(q)$ and $1 \leq j \leq n$, the set of irreducible characters of $\pi:=\prod^{j-1}_{i=0}(\tau_n-q^i \cdot 1_G)$ consists 
precisely of the characters labeled as in \eqref{gl-1}, where $s_1 = 1$ and the first part of the partition $\lam_1$ is at least $n-j$.
\end{propo}

\begin{proof}
Note that the case $j =n$ follows from the proof of Lemma \ref{value} applied to $\pi$. So we will assume that $1 \leq j \leq n-1$.
Any $\varphi\in \Irr(G)$ occurs in $\pi$ precisely when it is a constituent of 
$R^G_L(\reg_{G_j} \otimes 1_{G_{n-j}})$, by Proposition \ref{gl-induced}. This can happen precisely when there is some 
irreducible character 
$$\chi = S(s_1,\lam_1) \circ S(s_2,\lam_2) \circ \ldots \circ S(s_m,\lam_m)$$
of $G_j$ as described in \eqref{gl-1} such that $\varphi$ is an irreducible constituent of $R^G_L(\chi \otimes 1_{G_{n-j}})$.
Note that $1_{G_{n-j}} = S(1,(n-j))$ in the notation of \eqref{gl-1}. 

Adding a factor $S(1,(0))$ to the right of $\chi$ if $s_i \neq 1$ for all $i$, and relabeling the $s_i$'s if needed, 
we may assume that $s_1, \ldots ,s_{m-1} \neq 1$ and $s_m = 1$. Note that
$$\chi = R^{G_j}_{G_{j-k} \times G_k}(\psi \otimes S(1,\lam_m)),$$
where $k := k_m$ and $\psi := S(s_1,\lam_1) \circ S(s_2,\lam_2) \circ \ldots \circ S(s_{m-1},\lam_{m-1}) \in \Irr(G_{j-k})$. It follows
by Corollary \ref{product2} that
$$\begin{aligned}\chi \otimes 1_{G_{n-j}} & = R^{G_j}_{G_{j-k} \times G_k}(\psi \otimes S(1,\lam_m)) \otimes S(1,(n-j)) \\
   & = R^{G_j \times G_{n-j}}_{G_{j-k} \times G_k \times G_{n-j}}(\psi \otimes S(1,\lam_m) \otimes S(1,(n-j))).\end{aligned}$$
Using the transitivity of the Harish-Chandra induction \cite[Proposition 4.7]{DM2} and Corollary \ref{product2}, we get
$$\begin{aligned}R^G_L(\chi \otimes 1_{G_{n-j}}) & = R^{G_n}_{G_j \times G_{n-j}} 
    (R^{G_j \times G_{n-j}}_{G_{j-k} \times G_k \times G_{n-j}}(\psi \otimes S(1,\lam_m) \otimes S(1,(n-j))))\\
    & = R^{G_n}_{G_{j-k} \times G_k \times G_{n-j}}(\psi \otimes S(1,\lam_m) \otimes S(1,(n-j)))\\
    & = R^{G_n}_{G_{j-k} \times G_{n-j+k}} 
    (R^{G_{j-k} \times G_{n-j+k}}_{G_{j-k} \times G_k \times G_{n-j}}(\psi \otimes S(1,\lam_m) \otimes S(1,(n-j))))\\
    & = R^{G_n}_{G_{j-k} \times G_{n-j+k}}(\psi \otimes R^{G_{n-j+k}}_{G_k \times G_{n-j}}(S(1,\lam_m) \otimes S(1,(n-j)))).\end{aligned}$$
By Lemmas \ref{sym} and \ref{sym-gl}, each irreducible constituent of 
\begin{equation}\label{gl-c1}
  R^{G_{n-j+k}}_{G_k \times G_{n-j}}(S(1,\lam_m) \otimes S(1,(n-j)))
\end{equation}  
is $S(1,\gam)$ for some $\gam = (\gam_1, \ldots,\gam_r) \vdash n-j+k$ with $\gam_1 \geq n-j$, and conversely, any such $S(1,\gam)$ 
occurs in $R^{G_{n-j+k}}_{G_k \times G_{n-j}}(S(1,\lam_m) \otimes S(1,(n-j)))$ for some $\lam_m \vdash k$. Moreover, if $\gam_1 = n-j$,
then such an occurrence can happen only when $\lam_m = (\gam_2, \ldots ,\gam_r)$,
in which case $S(1,\gam)$ occurs in \eqref{gl-c1} with multiplicity one. 
Again by transitivity of the Harish-Chandra induction \cite[Proposition 4.7]{DM2} and Corollary \ref{product2} we have for any such
$\gam$ that 
$$R^{G_n}_{G_{j-k} \times G_{n-j+k}}(\psi \otimes S(1,\gam)) = 
    S(s_1,\lam_1) \circ S(s_2,\lam_2) \circ \ldots \circ S(s_{m-1},\lam_{m-1}) \circ S(1,\gam)$$
and so it is irreducible. We have therefore shown that the irreducible constituents of $\pi$ are precisely the characters  
$$\varphi = S(s_1,\lam_1) \circ S(s_2,\lam_2) \circ \ldots \circ S(s_{m-1},\lam_{m-1}) \circ S(1,\gam)$$ 
with $s_i \neq 1$ and $\gam_1 \geq n-j$. Any such $\varphi$ occurs in $\pi$ by taking $\chi$ as in \eqref{gl-1},
with $s_m=1$ and $\lam_m \vdash k$ chosen suitably. 
\end{proof} 

We will now identify the dual group $G^*$ with $G$ and use Lusztig's classification of complex characters of $G$, see \cite{C}, \cite{DM2}.
If $s \in G$ is a semisimple element, then $\EC(G,(s))$ denotes the rational series of irreducible characters of $G$ labeled by the $G$-conjugacy 
class of $s$.
Next, we can decompose $V = V^0 \oplus V^1$ as direct sum of $s$-invariant subspaces,
where $V^0 = \oplus_{\eps \in \FQ^\times}V_\eps$, $s$ acts on $V_\eps$ as $\eps\cdot 1_{V_\eps}$, and no eigenvalue of  
$s^1 := s|_{V^1}$ belongs to $\FQ^\times$. Then 
$$\CB_G(s) = \prod_{\eps \in \FQ^\times}\GL(V_\eps) \times \CB_{\GL(V^1)}(s^1).$$
Correspondingly, any unipotent character $\psi$ of $\CB_G(s)$ can be written in the form
\begin{equation}\label{gl-2}
  \psi = \bigotimes_{\eps \in \FQ^\times}\psi^{\gam_\eps} \otimes \psi_1,
\end{equation}  
where $\psi^{\gam_\eps} = S(1,\gam_\eps)$ is the unipotent character of $\GL(V_\eps)$ labeled by a partition $\gam_\eps$ of $\dim_{\FQ}V_\eps$, 
and $\psi_1$ is a unipotent character of $\CB_{\GL(V^1)}(s^1)$. If $V_\eps = 0$, then we view $\gam_\eps$ as the partition $(0)$ of $0$. 

\begin{theor}\label{main1-gl}
Let $\chi$ be an irreducible character of $\GL_n(q)$ which is labeled by the $G$-conjugacy class of a semisimple element
$s \in G$ and a unipotent character $\psi$ of $\CB_G(s)$ written as in \eqref{gl-2}. Let $0 \leq j \leq n$ be an integer.

\begin{enumerate}[\rm(i)]
\item $\chi$ has true level $j$ precisely when the first  part of the partition $\gam_1$ is $n-j$.
\item $\chi$ has level $j$ precisely when the longest among the first parts of the partitions $\gam_\eps$, $\eps \in \FQ^\times$, is $n-j$.
\end{enumerate}
\end{theor}

\begin{proof}
(i) Define $\al_i = (\tau_n)^i$ and $\beta_i := \prod^{i-1}_{k=0}(\tau_n-q^k \cdot 1_G)$ for $1 \leq i \leq n$, and $\al_0 = \beta_0 := 1_G$. 
Then condition \ref{separation}(a) is fulfilled. Next, for $0 \leq i \leq n-1$, let $\XC_i$ denote the set of irreducible characters $\chi$ 
as in \eqref{gl-1} and with $s_1=1$ and the first part of $\lam_1$ equal to $n-i$. Also, let $\XC_n$ denote the set of irreducible characters $\chi$ 
as in \eqref{gl-1} and with $s_k \neq 1$ for all $k$. Then Proposition \ref{gl-const} shows that the conditions \ref{separation}(b), (c) are fulfilled.
Hence, by Lemma \ref{separation} and Definition \ref{def:gl}, $\XC_j$ is precisely the set of characters of true 
level $j$. Now the statement follows,
since, by the definition of $\XC_j$, $\chi \in \XC_j$ precisely when it has the first part of the partition $\gam_1$ equal to $n-j$ (see e.g. \cite[2.3.5]{BDK}).

\smallskip
(ii) The linear characters of $G$ are labeled by elements $t \in \ZB(G)$ and the principal character $1_G$, so let $\hat{t}$ denote the linear 
character corresponding to $(t,1_G)$. Then the proof of \cite[Proposition 13.30]{DM2} implies that the multiplication by $\hat{t}$ sends the 
rational series $\EC(G,(s))$ to $\EC(G,(st))$ (see also \cite[2.3.5]{BDK}). In fact, if $t = \eps \cdot 1_V$ for some $\eps \in \FQ^\times$, then 
$S(s_i,\lam_i)\hat{t} = S(\eps s_i,\lam_i)$ in the notation of \eqref{gl-1}. In particular, the partition $\gam_{\eps^{-1}}$ defined for $\chi$ plays the role 
of $\gam_1$ for $\hat{t}\chi$. Hence the statement follows from (i) and Definition \ref{def:gl}.  
\end{proof}
 
\begin{examp}\label{gl-ex}
{\em 
\begin{enumerate}[\rm(i)]
\item It is well known (see e.g. \cite{T1}) that $\tau_n$ is the sum of $1_G$ and some irreducible Weil characters of $G = \GL_n(q)$, 
and every Weil character, multiplied by a suitable linear character, occurs in $\tau_n$. Thus, 
$1_G$ has level $0$, and Weil characters are precisely the characters of level $1$.
\item The Steinberg character of $\GL_n(q)$ is the unipotent character corresponding to the partition $(1^n)$, hence it has level $n-1$. 
\end{enumerate}
}
\end{examp} 
 
\section{Character level in finite general unitary groups}
Let $q$ be a prime power and let $G := G_n := \GU_n(q)$ with $n \geq 2$ and natural module $V = \langle e_1, \ldots, e_n \rangle_{\FF_{q^2}}$.
Let $\zeta_n$ denote the {\it reducible Weil character} of $G$ (see e.g. \cite{Ge}, \cite{TZ2}), that is,
\begin{equation}\label{for-zeta}
  \zeta_n(g) = (-1)^n (-q)^{\dim_{\FF_{q^2}}\Ker(g-1_V)}
\end{equation}  
for all $g \in G$.  Applying Lemma \ref{value} to $(G,\Theta) = (\GU_n(q),\zeta_n)$, we get

\begin{corol}\label{value-gu}
Each irreducible character of $\GU_n(q)$ occurs as an irreducible constituent of $(\zeta_n)^k$ for some $0 \leq k \leq n$.
\hfill $\Box$
\end{corol} 
 
In view of Corollary \ref{value-gu}, we can introduce

\begin{defi}\label{def:gu}
{\em Let $\chi \in \Irr(\GU_n(q))$. 
\begin{enumerate}[\rm(i)]
\item We say that $\chi$ has {\it true level $j$}, and write $\cl^*(\chi) = j$, if $j$ is the smallest non-negative integer such that $\chi$ is an irreducible 
constituent of $(\zeta_n)^j$.
\item  We say that $\chi$ has {\it level $j$}, and write $\cl(\chi) = j$, if $j$ is the smallest non-negative integer such that $\chi\lam$ is an irreducible 
constituent of $(\zeta_n)^j$ for some character $\lam$ of $\GU_n(q)$ of degree $1$.
\end{enumerate}}
\end{defi} 

Recall that the unipotent characters of $\GU_n(q)$ are parametrized by partitions $\lam \vdash n$: $\psi = \psi^\lam$.
We will need the following property of the Lusztig induction of unipotent characters of $\GU_n(q)$, see \cite[Proposition (1C)]{FS}:

\begin{lemma}\label{sym-gu}
Let $\al \vdash m$ and $\beta \vdash n$, and consider the
Levi subgroup $G_m \times G_n$ of $G_{m+n} =\GU_{m+n}(q)$. Then
$$R^{G_{m+n}}_{G_m \times G_n}(\psi^\al \otimes \psi^\beta) = \sum_{\lam \vdash\,m+n}(\pm a_{\al\beta\lam})\psi^\lam,$$
where $a_{\al\beta\lam}$ is the Littlewood-Richardson coefficient as in Lemma \ref{sym-gl}.
\hfill $\Box$
\end{lemma}

We will now introduce some notation as in \cite{ThV}. Let $F$ act on $\bar{\FF}_q^\times$ via $F(x) = x^{-q}$ and 
let $\Theta$ be the set of $F$-orbits on $\bar{\FF}_q^\times$. Then there is a natural bijection between $\Theta$ and 
$\Phi$, the set of {\it $F$-irreducible polynomials $f = f(t)$ over $\FF_{q^2}$}, that is, the monic polynomials 
$f \in \FF_{q^2}[t]$ for which there is an $F$-orbit $\OC$ of length $\deg(f)$ such that $f(t) = \prod_{z \in \OC}(f-z)$.
Let $\PC_n$ denote the set of partitions of $n \geq 0$, and let $\PC = \cup^{\infty}_{n=0}\PC_n$. For $\lam \in \PC_n$, we 
define $|\lam| = n$. Fix a 
linear order on $\Phi$. Now, a {\it $\Phi$-partition} 
$$\bnu = (\bnu(f_1),\bnu(f_2), \bnu(f_3),\ldots)$$
is a sequence of partitions in $\PC$ indexed by $\Phi$, of {\it size} 
$||\bnu|| = \sum_{f \in \XC}|\bnu(f)|\deg(f)$.
Denote 
$$\PC_n^\Phi = \{ \Phi\mbox{-partitions }\bnu \mid ||\bnu|| = n\},~~\PC^\Phi = \cup^\infty_{n=1}\PC_n^\Phi.$$ 
Then the conjugacy classes $c_\bmu$ in $\GU_n(q)$ are naturally indexed by $\bmu \in \PC_n^\Phi$.  We let 
$c_\bone$ denote the class of the identity. Also, for a class $c_\bmu$, let $\pi_\bmu$ denote the class function 
that takes value $1$ on $c_\bmu$ and $0$ elsewhere. 

Let $\CL_n$ denote the space of complex-valued class functions on $\GU_n(q)$. Then Ennola defined in \cite{E}
the following product $\al_1 \star \al_2 \in \CL_{n_1+n_2}$ for $\al_1 \in \CL_{n_1}$ and $\al_2 \in \CL_{n_2}$, where
\begin{equation}\label{star1}
  \al_1 \star \al_2(c_\blam) = \sum_{||\bmu_1||=n_1,||\bmu_2|| = n_2}g^{\blam}_{\bmu_1 \bmu_2}\al_1(c_{\bmu_1})\al_2(c_{\bmu_2}),
\end{equation}   
and $g^{\blam}_{\bmu_1 \bmu_2}$ is defined using the Hall polynomials \cite[Chapter II]{M}:
\begin{equation}\label{star2}
  g^{\blam}_{\bmu_1 \bmu_2} = \prod_{f \in \Phi}g^{\blam(f)}_{\bmu_1(f) \bmu_2(f)}((-q)^{\deg(f)}).
\end{equation}  
It turns out, see \cite[Corollary 4.2]{ThV}, that $\star$ coincides with the Lusztig induction:

\begin{propo}\label{gu-star}
Let $\al$ be a class function on $G_m=\GU_m(q)$ and $\beta$ be a class function on $G_n=\GU_n(q)$. Then 
$$R^{G_{m+n}}_{G_m \times G_n}(\al \otimes \beta) = \al \star \beta.$$
\hfill $\Box$
\end{propo}

\begin{propo}\label{gu-induced}
In the above notation, for $1 \leq j \leq n-1$ we have
$$R^{G_n}_{G_j \times G_{n-j}}\left(\reg_{G_j} \otimes 1_{G_{n-j}}\right) = (-1)^{j(n-j)}\prod^{j-1}_{i=0}(\zeta_n-(-1)^{n-i}q^i \cdot 1_G).$$
\end{propo}

\begin{proof}
Since $\reg_{G_j} = |G_j|\pi_\bone$, Proposition \ref{gu-star} and \eqref{star1}, \eqref{star2} imply that 
\begin{equation}\label{gu-2}
  \begin{split}
  R^{G_n}_{G_j \times G_{n-j}}\left(\reg_{G_j} \otimes 1_{G_{n-j}}\right)(g) & = |G_j|\sum_{||\bnu|| = n-j}g^{\blam}_{\bone \bnu}\\
    & =  |G_j|\sum_{||\bnu|| = n-j}\prod_{f \in \Phi}g^{\blam(f)}_{\bone(f) \bnu(f)}((-q)^{\deg(f)}).\end{split}
\end{equation} 
if $g \in G_n$ belongs to the conjugacy class $c_\blam$.   
For any prime $r$ and any integer $n$, a finite abelian $r$-group $M$ is said to have {\it type $\lam = (\lam_1, \ldots ,\lam_k) \vdash n$}
if 
$$M \cong C_{r^{\lam_1}} \times C_{r^{\lam_2}} \times \ldots \times C_{r^{\lam_k}}.$$      
Note that the Hall polynomial $g^\lam_{\mu \nu}(x)$ is characterized by the property that, for every prime $r$,
$g^\lam_{\mu\nu}(r)$ is  the number of subgroups $N \leq M$ such that $N$ has type $\mu$ and $M/N$ has type $\nu$,
where $M$ is a fixed abelian $r$-group of type $\lam$.   

Recall that $\bone(f)$ equals $(1^j)$ if $f = t-1$ and $0$ otherwise. Hence, for any $f \in \Phi$ with $f \neq t-1$, we see that 
$g^\lam_{\bone(f)\nu}(r)$ equals $1$ if $\nu = \lam$ and $0$ otherwise, whence $g^\lam_{\bone(f)\nu} = \delta_{\lam,\nu}$.
Hence, in the summation in \eqref{gu-2} we need to consider only the $\bnu$ with $||\bnu|| = n-j$ and $\bnu(f) = \lam(f)$ for 
all $f \neq t-1$, for which
\begin{equation}\label{gu-3}
  \prod_{f \in \Phi}g^{\blam(f)}_{\bone(f) \bnu(f)}((-q)^{\deg(f)}) =  g^{\blam^1}_{(1^j),\bnu^1}(-q),
\end{equation}  
where $\blam^1 := \blam(t-1)$ and $\bnu^1:= \bnu(t-1)$. Writing $\blam^1 = (\lam_1, \lam_2, \ldots,\lam_k)$ with $k$ nonzero parts,
we see that the unipotent part $u$ of $g$ has Jordan canonical form $\diag(J_{\lam_1},J_{\lam_2}, \ldots ,J_{\lam_k})$ on 
$V_1 = \Ker(s-1_V)$, where $s$ is the semisimple part of $g$ (and $J_m$ is the Jordan block of size $m$ with eigenvalue $1$). 
It follows that  
\begin{equation}\label{gu-4}
  k = \dim_{\FF_{q^2}}\Ker(g - 1_V).
\end{equation}  

First we consider the case $k < j$. Then $g^{\blam^1}_{(1^j),\bnu^1}(r) = 0$ for all $r$ (indeed,
any abelian $r$-group $M$ of type $\blam^1$ has $r$-rank $k$ and so cannot contain any subgroup $N$ of type $(1^j))$, 
whence $g^{\blam^1}_{(1^j),\bnu^1}(-q) = 0$, regardless of $\bnu^1$. Together with \eqref{gu-2} and \eqref{gu-3}, this implies
that $R^{G_n}_{G_j \times G_{n-j}}\left(\reg_{G_j} \otimes 1_{G_{n-j}}\right)(g)=0$. On the other hand, \eqref{gu-4} and the condition
$k < j$ yield $\prod^{j-1}_{i=0}(\zeta_n(g)-(-1)^{n-i}q^i)=0$, and so we are done in this case.

Now we consider the case $k \geq j$. Recall that we need to consider only those $\bnu$ with 
$\bnu(f) = \blam(f)$ for all $f \neq t-1$, whence $|\bnu^1|=|\blam^1|-j$.
If $M$ is an abelian $r$-group of type $\blam^1$, then it has $r$-rank $k$ and so 
$\Omega_1(M)$ is elementary abelian of rank $k$. Any subgroup $N$ of type $(1^j)$ is then an elementary abelian $r$-subgroup of
rank $j$ in $\Omega_1(M)$, and $M/N$ has type $\nu'$ for some $\nu' \vdash (|\blam^1|-j)$. It follows that $h(r)$ is just the number of 
elementary abelian subgroups of rank $j$ in $\Omega_1(M)$, i.e. 
$$h(r) = \frac{\prod^{j-1}_{i=0}(r^k-r^i)}{\prod^{j-1}_{i=0}(r^j-r^i)},$$
if we set 
$$h(x) := \sum_{\bnu^1 \vdash (|\blam^1|-j)}g^{\blam^1}_{(1^j),\bnu^1}(x) \in \CC[x].$$
Since this happens for all primes $r$, we can conclude (with using also \eqref{gu-4}) that 
$$\sum_{\bnu^1 \vdash (|\blam^1|-j)}g^{\blam^1}_{(1^j),\bnu^1}(-q) = h(-q) =  \frac{\prod^{j-1}_{i=0}((-q)^k-(-q)^i)}{\prod^{j-1}_{i=0}((-q)^j-(-q)^i)} 
    = \frac{\prod^{j-1}_{i=0}((-1)^n\zeta_n(g)-(-q)^i)}{(-1)^{j^2}|G_j|}.$$
Together with \eqref{gu-2} and \eqref{gu-3}, this implies that 
$$R^{G_n}_{G_j \times G_{n-j}}\left(\reg_{G_j} \otimes 1_{G_{n-j}}\right)(g) = (-1)^{j^2}\prod^{j-1}_{i=0}((-1)^n\zeta_n(g)-(-q)^i),$$
as stated.
\end{proof}

We will again identify the dual group $G^*$ with $G=\GU_n(q)$ and use Lusztig's classification of complex characters of $G$.
If $s \in G$ is a semisimple element, then we can decompose $V = V^0 \oplus V^1$ as direct sum of $s$-invariant subspaces,
where $V^0 = \oplus_{\eps \in \mu_{q+1}}V_\eps$, $s$ acts on $V_\eps$ as $\eps\cdot 1_{V_\eps}$, and no eigenvalue of  
$s^1 := s|_{V^1}$ belongs to 
$$\mu_{q+1} := \{ x \in \FF_{q^2}^\times \mid x^{q+1} = 1\}.$$ 
Then 
$$\CB_G(s) = \prod_{\eps \in \mu_{q+1}}\GU(V_\eps) \times \CB_{\GU(V^1)}(s^1).$$
Correspondingly, any unipotent character $\psi$ of $\CB_G(s)$ can be written in the form
\begin{equation}\label{gu-5}
  \psi = \bigotimes_{\eps \in \mu_{q+1}}\psi^{\gam_\eps} \otimes \psi_1,
\end{equation}  
where $\psi^{\gam_\eps}$ is the unipotent character of $\GU(V_\eps)$ labeled by a partition $\gam_\eps$ of $\dim_{\FF_{q^2}}V_\eps$, 
and $\psi_1$ is a unipotent character of $\CB_{\GU(V^1)}(s^1)$. If $V_\eps = 0$, then we view $\gam_\eps$ as the partition $(0)$ of $0$. 

Fix an embedding of $\bar{\FF}^\times$ into $\CC^\times$. Then one can identify $\ZB(\CB_G(s))$ with 
$$\Hom(\CB_G(s)/[\CB_G(s),\CB_G(s)],\CC^\times)$$ 
as in \cite[(1.16)]{FS}, and the linear character of $\CB_G(s)$ corresponding to
$s$ will be denoted by $\hat{s}$. Now, the irreducible character $\chi$ of $G$ labeled by $s$ and the unipotent character $\psi$ is 
\begin{equation}\label{gu-6}
  \chi = \pm R^G_{\CB_G(s)}(\hat{s}\psi),
\end{equation} 
see \cite[p. 116]{FS}. 

\begin{propo}\label{gu-const}
For $G = \GU_n(q)$ and $1 \leq j \leq n$, all irreducible constituents of $\pi:=\prod^{j-1}_{i=0}(\zeta_n-(-1)^{n-i}q^i \cdot 1_G)$ are among
the characters given in \eqref{gu-6}, where $\psi$ is as in \eqref{gu-5} and 
the first part of the partition $\gam_1$ is at least $n-j$. Moreover, if the first part of the partition $\gam_1$ is exactly $n-j$,
then the corresponding character is an irreducible constituent of $\pi$.
\end{propo}

\begin{proof}
Note that the case $j =n$ follows from the proof of Lemma \ref{value} applied to $\pi$. So we will assume that $1 \leq j \leq n-1$.
Any $\varphi\in \Irr(G)$ occurs in $\pi$ precisely when it is a constituent of 
$R^G_L(\reg_{G_j} \otimes 1_{G_{n-j}})$ for $L := G_j \times G_{n-j}$, by Proposition \ref{gu-induced}. Thus there is 
an irreducible character $\al \in \Irr(G_j)$ such that $\varphi$ is an irreducible constituent of $R^G_L(\al \otimes 1_{G_{n-j}})$.
Now we can find a semisimple element $s =(s_W,1_U)  \in G$, where $V = W \oplus U$, $G_j = \GU(W)$, 
$G_{n-j} = \GU(U)$, and $\al \in \EC(G_j,(s_W))$. We also consider the decomposition $W = W^0 \oplus W^1$
for the element $s_W$ as prior to \eqref{gu-5}. Then the decomposition $V = V^0 \oplus V^1$ for $s$ satisfies  
$$V^1 = W^1,~~V_1 = W_1 \oplus U,~~V_\eps = W_\eps~\forall \eps \in \mu_{q+1} \smallsetminus \{1\}.$$
Let $\dim_{\FF_{q^2}}W_1 = k$ for some $0 \leq k \leq j$ and $t := s|_{V^1} = (s_W)|_{W^1}$. Now we have 
$$\begin{array}{r}\CB_{G_j}(s_W) = \CB_{\GU(V^1)}(t) \times \prod_{1 \neq \eps \in \mu_{q+1}}\GU(V_\eps) \times \GU(W_1),\\
    \CB_{G_n}(s) = \CB_{\GU(V^1)}(t) \times \prod_{1 \neq \eps \in \mu_{q+1}}\GU(V_\eps) \times \GU(V_1).\end{array}$$
In what follows, for brevity we will denote the restriction of $\hat{s}$, the linear character of $\CB_G(s)$ corresponding to $s$, to any
subgroup of $\CB_G(s)$ by the same symbol $\hat{s}$. In particular, the linear character of $\CB_{G_j}(s_W)$ corresponding to $s_W$
will also be denoted by $\hat{s}$. 

According to \eqref{gu-6} applied to $\al$, we have 
\begin{equation}\label{gu-7}
  \al = \pm R^{G_j}_{\CB_{G_j}(s_W)}(\hat{s}(\beta \otimes \psi^\nu)),
\end{equation}  
for some unipotent character $\beta$ of $\CB_{\GU(V^1)}(t) \times \prod_{1 \neq \eps \in \mu_{q+1}}\GU(V_\eps) \leq G_{j-k}$ and 
unipotent character $\psi^\nu$ of $\GU(W_1) = G_k$ (so $\nu \vdash k$). Note that $\hs$ is trivial on $\GU(V_1)$.
By transitivity of the Lusztig induction and Corollary \ref{product2} (and 
using \cite[Proposition 12.6]{DM2}) we have
$$\pm \al = R^{G_j}_{G_{j-k} \times G_{k}}(R^{G_{j-k} \times G_{k}}_{\CB_{G_j}(s_W)}(\hs\beta \otimes \psi^\nu)) 
         = R^{G_j}_{G_{j-k} \times G_{k}}(\tbe \otimes \psi^\nu),$$
where $\tbe := R^{G_{j-k}}_{\CB_{G_{j-k}}(s_W)}(\hs\beta)$. It follows, using the same properties of the Lusztig induction, that
$$\pm \al \otimes 1_{G_{n-j}} = R^{G_j}_{G_{j-k} \times G_{k}}(\tbe \otimes \psi^\nu) \otimes 1_{G_{n-j}}
    = R^{G_j \times G_{n-j}}_{G_{j-k} \times G_k \times G_{n-j}}(\tbe \otimes \psi^\nu \otimes 1_{G_{n-j}}).$$
Hence,
$$\begin{aligned}\pm R^G_L(\al \otimes 1_{G_{n-j}}) & = R^{G_n}_{G_j \times G_{n-j}} 
    (R^{G_j \times G_{n-j}}_{G_{j-k} \times G_k \times G_{n-j}}(\tbe \otimes \psi^\nu \otimes 1_{G_{n-j}}))\\
    & = R^{G_n}_{G_{j-k} \times G_k \times G_{n-j}}(\tbe \otimes \psi^\nu \otimes 1_{G_{n-j}})\\
    & = R^{G_n}_{G_{j-k} \times G_{n-j+k}} 
    (R^{G_{j-k} \times G_{n-j+k}}_{G_{j-k} \times G_k \times G_{n-j}}(\tbe \otimes \psi^\nu \otimes 1_{G_{n-j}}))\\
    & = R^{G_n}_{G_{j-k} \times G_{n-j+k}}(\tbe \otimes R^{G_{n-j+k}}_{G_k \times G_{n-j}}(\psi^\nu \otimes 1_{G_{n-j}})).\end{aligned}$$
By Lemmas \ref{sym} and \ref{sym-gu}, each irreducible constituent of 
$$R^{G_{n-j+k}}_{G_k \times G_{n-j}}(\psi^\nu \otimes 1_{G_{n-j}})$$
is $\psi^\lam$ for some $\lam = (\lam_1, \lam_2, \ldots,\lam_r) \vdash n-j+k$ with $\lam_1 \geq n-j$.
Again by transitivity of the Lusztig induction \cite[Proposition 4.7]{DM2} and Corollary \ref{product2} we have for any such
$\lam$ that 
$$\begin{aligned}R^{G_n}_{G_{j-k} \times G_{n-j+k}}(\tbe \otimes \psi^\lam) & = 
    R^{G_n}_{G_{j-k} \times G_{n-j+k}}(R^{G_{j-k}}_{\CB_{G_{j-k}}(s_W)}(\hs\beta) \otimes \psi^\lam)\\
    & = R^{G_n}_{G_{j-k} \times G_{n-j+k}}(R^{G_{j-k} \times G_{n-j+k}}_{\CB_{G}(s)}(\hs\beta \otimes \psi^\lam))\\ 
    & = R^{G_n}_{\CB_{G}(s)}(\hs\beta \otimes \psi^\lam)\\ 
    & = R^{G_n}_{\CB_{G}(s)}(\hs(\beta \otimes \psi^\lam)),\end{aligned}$$
and so by \eqref{gu-6} it is an irreducible character of $G$ up to sign. We have therefore shown that the irreducible constituents of $\pi$ 
are among the characters  given in \eqref{gu-6} with the first part of the partition $\gam_1 = \lam$ being at least $n-j$. 

Conversely, suppose that the partition $\gam_1 = \lam$ for the character $\varphi$ has the first part $\lam_1 = n-j$. 
The above arguments and Lemmas \ref{sym} and \ref{sym-gu} show that the multiplicity of $\varphi$ in $R^G_L(\al \otimes 1_{G_{n-j}})$ 
can be nonzero precisely when $\nu = (\lam_2, \lam_3, \ldots,\lam_r)$, in which case it is $\pm 1$. Certainly, 
$s$ (up to conjugacy), $\hs$, and $\beta$ are uniquely determined by $\varphi$. Now \eqref{gu-7} implies that $\al$ is 
uniquely determined by $\varphi$. Thus, there is a unique $\al \in \Irr(G_j)$ such that $\varphi$ is an irreducible constituent 
of $R^G_L(\al \otimes 1_{G_{n-j}})$. Therefore, even though the Lusztig induction $R^G_L$ may send characters to generalized characters 
and cancel out irreducible constituents of different $R^G_L(\al' \otimes 1_{G_{n-j}})$, the established uniqueness of $\al$ allows us 
to conclude that $\varphi$ is an irreducible constituent of $\pi$. Since $\reg_{G_j} = \sum_{\al' \in \Irr(G_j)}\al'(1)\al'$, we also see 
that $[\varphi,\pi]_G = \pm \al(1)$.
\end{proof} 

\begin{theor}\label{main1-gu}
Let $\varphi$ be an irreducible character of $\GU_n(q)$ which is labeled by the $G$-conjugacy class of a semisimple element
$s \in G$ and a unipotent character $\psi$ of $\CB_G(s)$ written as in \eqref{gu-5}. Let $0 \leq j \leq n$ be an integer.

\begin{enumerate}[\rm(i)]
\item $\varphi$ has true level $j$ precisely when the first  part of the partition $\gam_1$ is $n-j$. In this case, there is a unique 
$\al \in \Irr(G_j)$ such that $\varphi$ is a constituent of $R^{G_n}_{G_j \times G_{n-j}}(\al \otimes 1_{G_{n-j}})$; moreover,
$\varphi$ occurs in $(\zeta_n)^j$ with multiplicity $\al(1)$, and $\cl^*(\al) \geq 2j-n$. Furthermore, the map 
$\Theta:\varphi \mapsto \al$ yields a bijection between $\{\chi \in \Irr(\GU_n(q)) \mid \cl^*(\chi) =n\}$ and
$\{ \al \in \Irr(\GU_j(q)) \mid \cl^*(\al) \geq 2j-n\}$. 
\item $\varphi$ has level $j$ precisely when the longest among the first parts of the partitions $\gam_\eps$, $\eps \in \mu_{q+1}$, is $n-j$.
\end{enumerate}
\end{theor}

\begin{proof}
(i) Define $\al_i = (\zeta_n)^i$ and $\beta_i := \prod^{i-1}_{k=0}(\zeta_n-(-1)^{n-k}q^k \cdot 1_G)$ for $1 \leq i \leq n$, and $\al_0 = \beta_0 := 1_G$. 
Then condition \ref{separation}(a) is fulfilled. Next, for $0 \leq i \leq n-1$, let $\XC_i$ denote the set of irreducible characters $\chi$ labeled
by a semisimple element $s \in G$ and a unipotent character $\psi$ of $\CB_G(s)$ 
as in \eqref{gu-5}, where the first part of $\gam_1$ equal to $n-i$, and let $\XC_n$ denote the set of such irreducible characters $\chi$ 
where $V_1 = 0$ (i.e. $s-1_V$ is non-degenerate). Then Proposition \ref{gu-const} shows that the conditions \ref{separation}(b), (c) are fulfilled.
Hence, by Lemma \ref{separation} and Definition \ref{def:gu}, $\XC_j$ is precisely the set of characters of true level $j$. 

Suppose now that $\cl^*(\varphi) = j$. Then the existence and uniqueness of $\al$ have been established in the proof of 
Proposition \ref{gu-const}, where we also showed that $[\varphi,\beta_j]_G = \pm \al(1)$. Note that $\al_j-\beta_j$ is a linear combination of 
$(\zeta_n)^i$, $0 \leq i < j$, and so has no constituent of true level $j$ by Definition \ref{def:gu}. Hence 
$[\varphi,\al_j]_G = [\varphi,\beta_j]_G = \pm \al(1)$. Since $\al_j = (\zeta_n)^j$ is a true character, we must have that
$[\varphi,(\zeta_n)^j]_G = \al(1)$. Also, in the notation of the proof of Proposition \ref{gu-const}, $\lam=(n-j,\lam_2, \ldots ,\lam_r)$ and 
$\nu = (\lam_2, \ldots ,\gam_r)$. In particular, $\lam_2 \leq n-j$, whence $\cl^*(\al) \geq 2j-n$ by the first statement applied to $G_j$,
and the map $\Theta$ is well-defined. It is injective, as $\varphi$ is uniquely determined by $\al$.

Conversely, suppose that $\al \in \Irr(G_j)$ as given in \eqref{gu-7} satisfies $\cl^*(\al) \geq 2j-n$. Then, in the above 
notation $\nu = (\lam_2, \ldots ,\lam_r) \vdash k$ we have $\lam_2 \leq j-(2j-n) = n-j$, hence $\lam := (n-j,\lam_2, \ldots ,\lam_r)$ is 
a partition of $n-j+k$. With $\varphi$ defined as in \eqref{gu-6}, the proof of Proposition \ref{gu-const} shows that $\Theta(\varphi) = \al$
and so $\Theta$ is onto. 

\smallskip
(ii) The linear characters of $G$ are precisely the characters $\hat{t}$ with $t \in \ZB(G)$. Also note that
if $s \in G$ is semisimple then $\hat{t}|_{\CB_G(s)}\hat{s} = \widehat{st}$. Hence, by \cite[Proposition 12.6]{DM2},
we have $R^G_{\CB_G(s)}(\hs\psi)\hat{t} = R^G_{\CB_G(s)}(\widehat{st}\psi)$. In particular, if $t = \eps \cdot 1_V$ for some $\eps \in \mu_{q+1}$, then 
the partition $\gam_{\eps^{-1}}$ defined for $\varphi$ plays the role 
of $\gam_1$ for $\hat{t}\varphi$. Hence the statement follows from (i) and Definition \ref{def:gu}.  
\end{proof}
 
We record the following consequence of the above proof:
 
\begin{corol}\label{gu-dual-bij}
Let $\cl^*(\al) \geq 2j-n$ for $\al \in \Irr(\GU_j(q))$ and express $\al$ as in \eqref{gu-7}. If $\Theta$ denotes the 
bijection in Theorem \ref{main1-gu}(i), then $\Theta(\varphi) = \al$, where $\varphi$ is as defined in \eqref{gu-6}, 
with $\lam = \gam_1$ obtained by adding $n-j$ to the front of $\nu = (\lam_2, \ldots,\lam_r)$: 
$\lam = (n-j,\lam_2, \ldots,\lam_r)$.  
\hfill $\Box$
\end{corol}
 
\begin{examp}\label{gu-ex}
{\em 
\begin{enumerate}[\rm(i)]
\item It is well known (see e.g. \cite{TZ2}) that $\zeta_n$ is a sum of irreducible Weil characters of $G = \GU_n(q)$,
and every Weil character, multiplied by a suitable linear character, occurs in $\zeta_n$. Thus, 
Weil characters of $G$ are precisely the characters of level $1$.
\item The Steinberg character of $\GU_n(q)$ is the unipotent character corresponding to the partition $(1^n)$, hence it has level $n-1$. 
\end{enumerate}
}
\end{examp} 

\section{Further results on character levels}

Recall that if $\GC$ is a connected reductive algebraic group and $F:\GC \to \GC$ is a Frobenius endomorphism, 
then the {\it Alvis-Curtis duality functor} $D_\GC$ sends any 
irreducible character of $G:=\GC^F$ to an irreducible character, up to a sign, see \cite[Chapter 8]{DM2}, and defines an involutive 
unitary transformation on the space $\CL(G)$ of complex-valued class functions on $G$. If $(\GCD,\FD)$ is dual to 
$(\GC,F)$ and $G^* := (\GCD)^{\FD}$, then $\EC(G,(s))$ denotes the rational series corresponding to the $\GD$-conjugacy class
of a semisimple element $s \in \GD$, see \cite[p. 136]{DM2}. Furthermore, let $\varep_\GC := (-1)^{\sigma_\GC}$, where 
$\sigma_\GC$ is the {\it relative rank of $\GC$}, as defined on \cite[p. 197]{C}.

\begin{lemma}\label{dual1}
In the above notation, the following statements hold.
\begin{enumerate}[\rm(i)]
\item The Alvis-Curtis duality functor respects rational series of irreducible characters of $G$.
\item Suppose $f \in \CL(G)$ is such that $f(g) = f(s)$ whenever $s$ is the semisimple part of $g \in G$. Then 
$D_\GC(f\al) = fD_\GC(\al)$ for any $\al \in \CL(G)$. 
\end{enumerate}
\end{lemma}

\begin{proof}
(i) If $\TC$ is any $F$-stable maximal torus, then note by \cite[Definition 8.8]{DM2} that $D_\TC$ acts trivially on $\CL(\TC^F)$. 
Hence, by \cite[Proposition 8.11]{DM2}, for any character $\theta \in \Irr(\TC^F)$ we have that
\begin{equation}\label{ac1}
   D_\GC(R^\GC_\TC(\theta)) = \varep_\GC \varep_\TC R^\GC_\TC(D_\TC(\theta)) =  \varep_\GC \varep_\TC R^\GC_\TC(\theta) = 
      \pm R^\GC_\TC(\theta),
\end{equation}      
whence the statement follows.    
   
(ii) Let $\LC$ be an $F$-stable Levi subgroup of an $F$-stable parabolic subgroup of $\GC$, and let 
$\SR^\GC_\LC$ denote the Harish-Chandra restriction. By \cite[Proposition 12.6]{DM2} we have
$$R^\GC_\LC(\SR^\GC_\LC(f\al)) = R^\GC_\LC ((\SR^\GC_\LC \al) \cdot f|_{\LC^F}) 
    = f \cdot R^\GC_\LC(\SR^\GC_\LC(\al)).$$  
Now the statement follows by applying \cite[Definition 8.8]{DM2}.       
\end{proof}

\begin{lemma}\label{dual2}
Let $\GC = \GC_1 \times \GC_2 \times \ldots \times \GC_m$ be a direct product of algebraic groups $\GC_i = \GL_{n_i}$ and 
let $F:\GC \to \GC$ be a Frobenius endomorphism that stabilizes each factor $\GC_i$. Let 
$$W = W_1 \times W_2 \times \ldots \times W_m \cong \SSS_{n_1} \times \SSS_{n_2} \times \ldots \times \SSS_{n_m}$$
be the Weyl group of $\GC$. 
For any $\lam \in \Irr(W)$, let $\psi^\lam$ denote the corresponding unipotent character of $\GC^F$. Then 
$$D_\GC(\psi^\lam) = \pm \psi^{\lam \cdot \sgn},$$
where $\sgn \in \Irr(W)$ is the character that sends each permutation $\pi \in \SSS_{n_i}$ to the sign of $\pi$.
\end{lemma}

\begin{proof}
For $w \in W$, let $\TC_w$ denote an $F$-stable maximal torus of $\GC$ corresponding to the $W$-conjugacy class of $w$. Then, 
according to \cite[(1.13)]{FS}, 
\begin{equation}\label{ac2}
  \psi^\lam = \frac{a_\lam}{|W|}\sum_{w \in W}\lam(w)R^\GC_{\TC_w}(1_{\TC^F_w})
\end{equation}   
for some $a_\lam = \pm 1$.
Note that $\varep_\GC\varep_{\TC_w} = (-1)^{l(w)}$ by \cite[Exercise 13.29]{DM2}, where $l(w)$ is the length of the element
$w$ in $W$. Since $W_i$ is of type $A$, we have $(-1)^{l(w)} = \sgn(w)$. Hence \eqref{ac1} implies that 
$$D_\GC(R^\GC_{\TC_w}(1_{\TC^F_w})) = \varep_\GC\varep_{\TC_w}R^\GC_{\TC_w}(1_{\TC^F_w}) = \sgn(w)R^\GC_{\TC_w}(1_{\TC^F_w}).$$
Applying $D_\GC$ to \eqref{ac2}, we obtain that $D_\GC(\psi^\lam) = a_\lam a_{\lam \cdot \sgn}\psi^{\lam \cdot \sgn}$.
\end{proof}

Let $\GC = \GL_n$ and let $F:\GC \to \GC$ be a Frobenius endomorphism, so that $G := \GC^F = \GL_n(q)$ or $\GU_n(q)$. As before,
we can identity the dual group $\GD$ with $G$. For any semisimple element $s \in G$, 
$$\LC := \CB_\GC(s) = \GC_1 \times \GC_2 \times \ldots \times \GC_m$$ 
is of the form described in  Lemma \ref{dual2}, and likewise the Weyl group 
$$W_\LC \cong \SSS_{n_1} \times \SSS_{n_2} \times \ldots \times \SSS_{n_m}$$ 
is a direct product of symmetric groups.
Hence, any unipotent character $\psi^\mu$ of $\CB_G(s)$ is labeled by an irreducible character $\mu \in \Irr(W_\LC)$, as described in \eqref{ac2}. 
We will keep the notation $\sgn$ as in Lemma \ref{dual2}.

\begin{propo}\label{dual3}
In the above introduced notation, let $\chi^{s,\mu}$ denote the irreducible character of $G=\GC^F$ labeled by a semisimple element 
$s \in G$ and the unipotent character $\psi^\mu \in \Irr(\CB_G(s))$ corresponding to $\mu \in \Irr(W_\LC)$, as in \cite[(1.18)]{FS}. Then 
$$D_\GC(\chi^{s,\mu}) = \pm \chi^{s,\mu \cdot \sgn}.$$
\end{propo}

\begin{proof}
For $w \in W_\LC$, let $\TC_w$ denote an $F$-stable maximal torus of $\LC = \CB_\GC(s)$ corresponding to the $W_\LC$-conjugacy class of $w$. 
Then, according to \cite[(1.18)]{FS} we have that
\begin{equation}\label{ac3}
  \chi^{s,\mu} = \frac{a_{s,\mu}}{|W_\LC|}\sum_{w \in W_\LC}\mu(w)R^\GC_{\TC_w}(\hs|_{\TC^F_w})
\end{equation}   
for some $a_{s,\mu} = \pm 1$, where the linear character $\hs$ of $\CB_G(s)$ is introduced before \eqref{gu-6}.
As mentioned in the proof of Lemma \ref{dual2}, $\varep_{\TC_w} = \sgn(w)\varep_\LC$. Applying 
\eqref{ac1}, we obtain that 
$$D_\GC(R^\GC_{\TC_w}(\hs|_{\TC^F_w})) = \varep_\GC\varep_{\TC_w}R^\GC_{\TC_w}(\hs|_{\TC^F_w}) = 
    \sgn(w)\varep_\GC\varep_\LC R^\GC_{\TC_w}(\hs|_{\TC^F_w}).$$
Applying $D_\GC$ to \eqref{ac3}, we then arrive at 
$$D_\GC(\chi^{s,\mu}) = a_{s,\mu}a_{s,\mu \cdot \sgn}\varep_\GC\varep_\LC\chi^{s,\mu \cdot \sgn}.$$
\end{proof}
 
Now we can link the levels of  $\chi \in \Irr(G)$ and its Alvis-Curtis dual. Slightly abusing the notation,
we define $\cl(-\chi) := \cl(\chi)$ for any $\chi \in \Irr(G)$.

\begin{propo}\label{dual4}
Let $\GC = \GL_n$ and $G = \GC^F = \GL_n(q)$ or $\GU_n(q)$ for a Frobenius endomorphism $F:\GC \to \GC$. Then for any
$\chi \in \Irr(G)$,
$$\cl(\chi) + \cl(D_\GC(\chi)) \geq n-1.$$ 
\end{propo}

\begin{proof}
The proofs for $\GL_n(q)$ and $\GU_n(q)$ are identical, so we will give the details in the case $G = \GU_n(q)$. Let $\cl(\chi) = j$ and 
apply Theorem \ref{main1-gu} to $\chi$. Then the longest part among all the parts of the partitions $\gam_\eps$, $\eps \in \mu_{q+1}$,
is precisely $n-j$, if we express $\chi$ as in \eqref{gu-5}, \eqref{gu-6}. Say the partition $\gam_{\eps_0}$ has the longest part
$k_{\eps_0} = n-j$. 

According to Proposition \ref{dual3}, 
$D_\GC(\chi) = \pm \chi^{s,\mu \cdot \sgn}$, and so the unipotent character of $\CB_G(s)$ corresponding to $D_\GC(\chi)$ 
has $\eps$-component equal to $\psi^{\gam_\eps} \cdot \sgn = \psi^{\gam'_\eps}$, where $\gam'_\eps$ is the partition associated to 
$\gam_\eps \vdash n_\eps$, with the longest part $k'_\eps$. Now, if $\eps \neq \eps_0$, then 
$$k'_\eps \leq n_\eps \leq n-n_{\eps_0} \leq n-k_{\eps_0} = j.$$
Furthermore, since $\gam_{\eps_0}$ and $\gam'_{\eps_0}$ are associated partitions, we have
$$k'_{\eps_0} \leq n_{\eps_0}+1 -k_{\eps_0} \leq n+1 -k_{\eps_0} = j+1.$$
We have shown that $\max\{k'_\eps \mid \eps \in \mu_{q+1}\} \leq j+1$. Hence $\cl(D_\GC(\chi)) \geq n-j-1$ by Theorem \ref{main1-gu}.
\end{proof}

The example of $\chi = 1_G$ shows that the bound in Proposition \ref{dual4} is sharp, since $\pm D_\GC(1_G)$ is the 
Steinberg character of $G$.

\smallskip
Next, for a fixed $\eps = \pm$ we consider $G_n = \GL^\eps_n(q) = \GC^F$, 
$G_k = \GL^\eps_k(q) = (\GC_k)^F$ with $\GC_k = \GL_k$ for $1 \leq k \leq n$,
and define $\omega_n := \tau_n$ if $\eps = +$ and
$\omega_n:= (-1)^n\zeta_n$ if $\eps = -$. In what follows, the notation $(\TC) \subset \GC_k$ means that we sum over a set of representatives 
of $G_k$-conjugacy classes of $F$-stable maximal tori $\TC$ in $\GC_k$; furthermore, we set 
$T := \TC^F$ and $W(T) := \NB_{G_k}(\TC)/T$ for any such $\TC$.

Now we can prove the following result, which gives a decomposition for 
$(\omega_n)^m$ in terms of Deligne-Lusztig characters. In the case $q$ is sufficiently large, this statement follows 
from the main result of \cite{S2}. 

\begin{theor}\label{weil}
Fix $\eps = \pm$ and let $G_k = \GL^\eps_k(q)$ for $1 \leq k \leq n$. Then, for any $1 \leq m \leq n$,
$$\begin{aligned}(\omega_n)^m & = \sum^m_{k=0}\sum_{(\TC) \subset \GC_k}\frac{\eps^k}{|W(T)|}
    \sum_{\theta \in \Irr(T)}R^{G_m}_{T \times G_{m-k}}(\theta \otimes 1_{G_{m-k}})(1) R^{G_n}_{T \times G_{n-k}}(\theta \otimes 1_{G_{n-k}})\\
    & = \sum^m_{k=0}\sum_{(\TC) \subset \GC_k}\frac{\eps^k\varep_{\GC_m}\varep_{\GC_{m-k}}\varep_\TC}{|W(T)|}\cdot \frac{|G_m|_{p'}}{|T|\cdot |G_{m-k}|_{p'}}
    \sum_{\theta \in \Irr(T)}R^{G_n}_{T \times G_{n-k}}(\theta \otimes 1_{G_{n-k}}).
    \end{aligned}$$
\end{theor}

\begin{proof}
First we note that the second equality in the statement follows from the degree formula for $R^{G_m}_{T \times G_{m-k}}(\theta \otimes 1_{G_{m-k}})$
(see e.g. (2.3) and (2.6) of \cite{S2}). Next, we can replace the summation $\sum_{(\TC) \subset \GC_k}\frac{1}{|W(T)| \cdot |T|}$ by 
$\frac{1}{|G_k|}\sum_{\TC \subset \GC_k}$, where the second summation runs over all $F$-stable maximal tori $\TC$ in $\GC_k$. Since 
$\varep_{\GC_n} = \eps^{n(n-1)/2}$, we also see that 
\begin{equation}\label{rel-rk1}
  \varep_{\GC_m}\varep_{\GC_{m-k}} = \varep_{\GC_k}\eps^{km-k}
\end{equation}  
for $0 \leq k \leq m$. 
Hence, the right-hand side $R$ of the formula in Theorem \ref{weil} is 
$$\begin{aligned}R & = \sum^m_{k=0}\sum_{\TC \subset \GC_k}\frac{\varep_{\GC_k}\varep_\TC\eps^{km}|G_m|_{p'}}{|G_{m-k}|_{p'} \cdot |G_k|}
    \sum_{\theta \in \Irr(T)}R^{G_n}_{T \times G_{n-k}}(\theta \otimes 1_{G_{n-k}})\\
    & =  \sum^m_{k=0}\frac{\eps^{km}|G_m|_{p'}}{|G_{m-k}|_{p'}\cdot |G_k|} 
    \sum_{\TC \subset \GC_k,\,\theta \in \Irr(T)}\varep_{\GC_k}\varep_\TC R^{G_n}_{T \times G_{n-k}}(\theta \otimes 1_{G_{n-k}})\\
    & \stackrel{{\rm Cor.}\ \ref{product2}}{=} \sum^m_{k=0}\frac{\eps^{km}|G_m|_{p'}}{|G_{m-k}|_{p'} \cdot |G_k|_{p'}} R^{G_n}_{G_k \times G_{n-k}}
    \left(\sum_{\TC \subset \GC_k,\,\theta \in \Irr(T)}\frac{\varep_{\GC_k}\varep_\TC}{|G_k|_p} R^{G_k}_T(\theta) \otimes 1_{G_{n-k}}\right)\\
    & \stackrel{{\rm [DM2,\ Cor.\ 12.14]}}{=} \sum^m_{k=0}\frac{\eps^{km}|G_m|_{p'}}{|G_{m-k}|_{p'}\cdot |G_k|_{p'}} R^{G_n}_{G_k \times G_{n-k}}
    \left(\reg_{G_k} \otimes 1_{G_{n-k}}\right)\\
    & \stackrel{{\rm Props.\ \ref{gl-induced},\ \ref{gu-induced}}}{=} 
    \sum^m_{k=0}\frac{\eps^{k(m-k)}|G_m|_{p'}}{|G_{m-k}|_{p'}\cdot |G_k|_{p'}}
    \prod^{k-1}_{i=0}(\omega_n-(\eps q)^i \cdot 1_{G_n})\\
    & = \sum^m_{k=0}\binom{m}{k}_{\!\eps q}
    \prod^{k-1}_{i=0}(\omega_n-(\eps q)^i \cdot 1_{G_n}) \stackrel{{\rm Lemma\ \ref{z-identity}}}{=} (\omega_n)^m.
    \end{aligned}$$
(Here, at the last step, we apply Lemma \ref{z-identity} to $t = \omega_n(g)$ and $z=\eps q$ for every $g \in G_n$.)
\end{proof}

In the next result, which is a $\GU$-analogue of \eqref{gl-mult}, the first equality was known in the case $q$ is sufficiently large,
cf. \cite{S2}.

\begin{corol}\label{gu-mult}
Let $q$ be a power of a prime $p$, $G_k = \GU_k(q)$  for $1 \leq k \leq n$. Then for any $1 \leq m \leq n$, 
$$[(\zeta_n)^m,1_{G_n}]_{G_n} = \sum^m_{k=0}\frac{(-1)^{km}|G_m|_{p'}}{|G_k|_{p'} \cdot |G_{m-k}|_{p'}} = 
 \left\{ \begin{array}{ll} 0, & 2 \nmid m,\\ (q+1)(q^3+1) \ldots (q^{m-1}+1), & 2|m. \end{array} \right.$$
In particular, if $1 \leq j \leq n/2$ and $V = \FF_{q^2}^n$ is the natural $\GU_n(q)$-module, 
then the number of $\GU_n(q)$-orbits on ordered $j$-tuples $(v_1, \ldots,v_j)$ with $v_i \in V$ is $\prod^{j}_{i=1}(q^{2i-1}+1)$.
\end{corol}

\begin{proof}
We apply Theorem \ref{weil}. For any $\theta \in \Irr(T)$, note by \cite[Corollary 12.7]{DM2} that
$$\begin{aligned}
    ~[R^{G_n}_{T \times G_{n-k}}(\theta \otimes 1_{G_{n-k}}),1_{G_n}]_{G_n}  & = 
    [\theta \otimes 1_{G_{n-k}},\SR^{G_n}_{T \times G_{n-k}}(1_{G_n})]_{T \times G_{n-k}}\\
     ~& = [\theta \otimes 1_{G_{n-k}},1_T \otimes 1_{G_{n-k}}]_{T \times G_{n-k}} = \delta_{\theta,1_T},\end{aligned}$$  
where $\delta_{\cdot,\cdot}$ is the Kronecker symbol. Also, by \cite[Theorem 7.5.1]{C} and \cite[(2.6)]{S2} we have 
$$\sum_{(\TC) \subset \GC_k}\frac{\varep_{\GC_k}\varep_\TC}{|W(T)|}\cdot \frac{|G_k|_{p'}}{|T|} = \sum_{(\TC) \subset \GC_k}
    \frac{R^{G_k}_{T}(1_T)(1)}{|W(T)|} = 1.$$    
Now, taking the inner product of the last term in the displayed formula of Theorem \ref{weil} with $1_{G_n}$ and using 
\eqref{rel-rk1}, we obtain
$$\begin{aligned}
    ~[(\zeta_n)^m,1_{G_n}]_{G_n} & = 
    \sum^m_{k=0}\sum_{(\TC) \subset \GC_k}\frac{(-1)^k\varep_{\GC_m}\varep_{\GC_{m-k}}\varep_\TC}{|W(T)|}\cdot \frac{|G_m|_{p'}}{|T|\cdot |G_{m-k}|_{p'}}\\
    ~& = \sum^m_{k=0}\frac{(-1)^k\varep_{\GC_m}\varep_{\GC_{m-k}}\varep_{\GC_k}\cdot |G_m|_{p'}}{|G_k|_{p'}\cdot |G_{m-k}|_{p'}}
     \sum_{(\TC) \subset \GC_k}\frac{\varep_{\GC_k}\varep_\TC}{|W(T)|}\cdot \frac{|G_k|_{p'}}{|T|}\\
     ~& = \sum^m_{k=0}\frac{(-1)^{km}|G_m|_{p'}}{|G_k|_{p'} \cdot |G_{m-k}|_{p'}}.
    \end{aligned}$$
When $2\nmid m$, the terms for $k$ and $m-k$, $0 \leq k < m/2$, in the last summation cancel each other, yielding 
$[(\zeta_n)^m,1_{G_n}]_{G_n}=0$. If $2|m$, the last summation is 
$\sum^m_{k=0}(-1)^k\binom{m}{k}_{-q}$, and so it is $(q+1)(q^3+1) \ldots (q^{m-1}+1)$ by Gauss' formula (see \cite[(1.7b)]{Ku}).   

Finally, the last statement follows by applying the formula we just proved to $m = 2j$. 
\end{proof}

Corollary \ref{gu-mult} implies the following {\it parity phenomenon} for unitary groups:

\begin{corol}\label{parity}
Suppose that $0 \leq i, j \leq i+j \leq n$ and $2\nmid(i+j)$. Then the $\GU_n(q)$-characters $(\zeta_n)^i$ and $(\zeta_n)^j$
have no common irreducible constituents. In particular, if $0 \leq j \leq n/2$, then $(\zeta_n)^j$ contains only 
irreducible characters of true level $j-2t$, $0 \leq t \leq j/2$.
\end{corol}

\begin{proof}
Note that $[(\zeta_n)^i,(\zeta_n)^j]_{G_n} = [(\zeta_n)^{i+j},1_{G_n}]_{G_n}$. Hence the statements follow from
Corollary \ref{gu-mult} and Definition \ref{def:gu}.
\end{proof}

\section{Bounds on character degrees}
For $q=p^f \geq 2$ and $\eps = \pm$, we let $\GL^\eps_n(q)$ denote $\GL_n(q)$ if $\eps = +$ and $\GU_n(q)$ if $\eps = -$. Recall that 
$\psi^\lam$ denotes the unipotent character of $\GL^\eps_n(q)$ corresponding to a partition 
$\lam = (\lam_1,\lam_2, \ldots ,\lam_r) \vdash n$
(with the convention $\lam_1 \geq \lam_2 \geq \ldots \geq \lam_r \geq 0$).
For such $\lam$, define
\begin{equation}\label{ab}
  a(\lam) := \sum^r_{i=1}(i-1)\lam_i,~~b(\lam) := \frac{1}{2}(n^2-\sum^r_{i=1}\lam_i^2),
\end{equation}  
and let $G^\eps_\lam := \GL^\eps_{\lam_1}(q) \times \GL^\eps_{\lam_2}(q) \times \ldots \times \GL^\eps_{\lam_r}(q)$ be a Levi 
subgroup of $G^\eps_n:= \GL^\eps_n(q)$.

\smallskip
First we collect some elementary estimates:

\begin{lemma}\label{trivial}
Let $q \geq 2$ and $a, b \in \ZZ_{\geq 1}$. Then the following inequalities hold.
\begin{enumerate}[\rm(i)]
\item $\prod^{\infty}_{i=2}(1-1/q^i) > 9/16$ and $\prod^{\infty}_{i=1}(1-1/q^i) > (9/16)(1-1/q) \geq 9/32$.
\item If $d \in \ZZ_{\geq 2}$ then $\prod^{\infty}_{i=1,~d \nmid i}(1-1/q^i) > (9/16)(1-1/q)$.
\item $(q^{2a}-1)(q^{2a+1}+1) < q^{4a+1}$ and $(q^{2a-1}+1)(q^{2a}-1) > q^{4a-1}$.
\item $\prod^n_{i=1}(q^i-(-1)^i) > q^{n(n+1)/2}$.
\item If $a > b$, then 
$$\frac{q^a+1}{q^b+1} < q^{a-b} < \frac{q^a-1}{q^b-1},~~\frac{(q^{a+1}-1)(q^a+1)}{(q^{b+1}-1)(q^b+1)} < q^{2a-2b}.$$
\item $q^{ab}/2 \leq (q-1)q^{ab-1} \leq [G^-_{a+b}:(G^-_a \times G^-_b)]_{p'} < q^{ab} < [G^+_{a+b}:(G^+_a \times G^+_b)]_{p'}$. Furthermore, 
$[G^-_{a+b}:(G^-_a \times G^-_b)]_{p'} \geq (5/8)q^{ab}$ if $a \geq 2$, and $[G^-_{a+b}:(G^-_a \times G^-_b)]_{p'} > (q-1)q^{ab-1}$ if $a+b \geq 3$. 
\end{enumerate}
\end{lemma}

\begin{proof}
(i) This is \cite[Lemma 4.1(ii)]{LMT}. (ii) follows from (i). (iii) and (v) are obvious, and (iv) follows from (iii).

It is easy to see that $q^{a-b} < \frac{q^a-1}{q^b-1}$ implies $q^{ab} < [G^+_{a+b}:(G^+_a \times G^+_b)]_{p'}$. Next, if $2|a$, then 
a repeated application of (v) yields
$$X:= [G^-_{a+b}:(G^-_a \times G^-_b)]_{p'} = \frac{(q^{a+b}-(-1)^{a+b}) \ldots (q^{a+2}-1)(q^{a+1}+1)}{(q^b-(-1)^b) \ldots (q^2-1)(q+1)} < q^{ab}.$$
If $2 \nmid a$, then a repeated application of (iii) to the numerator and an application of (iv) to the denominator of $X$ yields
$$X = \frac{(q^{a+b}-(-1)^{a+b}) \ldots (q^{a+2}+1)(q^{a+1}-1)}{(q^b-(-1)^b) \ldots (q^2-1)(q+1)} < 
    \frac{q^{(a+b) + \ldots +(a+2)+(a+1)}}{q^{b+ \ldots + 2+1}} =  q^{ab}.$$
Assume now that $a \geq 2$. If $2|a$,  then a repeated application of (iii) shows that the numerator of $X$ is $> q^N$ for 
$N := \sum^{a+b}_{i=a+1}i$. If $2 \nmid a$, then, singling out the factor $q^{a+1}-1$ and then applying (iii) to the remaining 
product shows that the numerator of $X$ is $\geq q^{N-4}(q^4-1) \geq (15/16)q^N$.  Singling out the factor $q+1$ and then applying (iii) to the remaining 
product shows that the denominator of $X$ is $\leq q^{M-1}(q+1) \leq (3/2)q^M$ for $M := \sum^b_{i=1}i$.  It follows that
$X \geq (15/16)(2/3)q^{N-M} = (5/8)q^{ab}$, as well as 
$$X \geq (q^4-1)q^{N-M-3}/(q+1) > (q-1)q^{ab-1}.$$ 
The same argument applies if $b \geq 2$. Finally, if $a=b=1$ then $X = q-1$.
\end{proof}

\begin{lemma}\label{sum1}
Let $\lam = (k=\lam_1 \geq \lam_2 \geq  \ldots \geq \lam_r \geq 0) \vdash n$ and $2 \leq k \leq n/2$. Then 
$$[\lam] := n^2-2\sum^r_{i=1}\lam_i^2 \geq 0.$$ 
In fact, either $[\lam] \geq 2.4n$, or one of the following statements holds.
\begin{enumerate}[\rm(i)]
\item $n = 3k = 6$ and $\lam = (2,2,2)$.
\item $n = 2k+1$, and either $\lam = (k,k,1)$ or $\lam \in \{(3,2,2),(2,1,1,1)\}$.
\item $n = 2k$, and either $\lam \in \{(k,k), (k,k-1,1)\}$ or $\lam \in \{ (4,2,2), (3,1,1,1)\}$.
\end{enumerate}
\end{lemma}

\begin{proof}
(a) Note that if $a,b \in \ZZ_{\geq 1}$ and $a \geq b$, then $((a+1)^2+(b-1)^2)-(a^2+b^2) \geq 2$. In what follows we will
call any replacement of the pair $(a,b)$ among the $\lam_i$'s by $(a+1,b-1)$ a {\it push-up}, and note that any 
push-up decreases $[\lam]$ by at least $4$.

First suppose that $k \leq n/3$. Write $n = 3c+d$ with $c = \lfloor n/3 \rfloor \geq k \geq 2$ and $0 \leq d \leq 2$.  
As long as $\lam_1 < c$, we can apply a 
push-up to some pair $(\lam_1,\lam_j)$ (where $\lam_j > 0$ but $\lam_{j+1}$, if any, is $0$) to increase $\lam_1$ and 
decrease $[\lam]$. Once $\lam_1 = c$, we can apply the same procedure to $\lam_2$, and so on. This argument shows that
$[\lam]$ will be minimized when $\lam = (c,c,c,d)$. In particular, if $c \geq 3$ then
$$[\lam] \geq (3c+d)^2 -2(3c^2+d^2) = 3c^2+6cd-d^2 \geq 3c^2+5cd \geq 9c+15d \geq 3n.$$ 
Assume that $c = 2$, i.e. $6 \leq n \leq 8$. If $n \geq 7$, then a push-up argument shows that $[\lam] \geq [(2,2,2,d)] > 2.4n$.
If $n = 6$, but $\lam \neq (2,2,2)$, then $[\lam] \geq [(2,2,2)]+4 > 2.4n$. 

\smallskip
(b) Now we may assume $n/3 < k \leq n/2$, and write $n = 2k+l$ with $0 \leq l < k$. Again using push-ups, we see that
$$[\lam] \geq [(k,k,l)] = 4kl-l^2 \geq 3kl+l > 2.4n$$
if $l \geq 2$. Suppose $l= 1$ and $\lam \neq (k,k,1)$. Then $\lam_2 \leq k-1$ and we can push $\lam_2$ up to $k-1$. 
In particular, if $k \geq 4$, then $[\lam] \geq [(k,k-1,2)] = 8k-9 > 2.4n$. If $k = 3$ but $\lam \neq (3,2,2)$, then
$[\lam] \geq [(3,2,2)]+4 > 2.4n$. If $k = 2$ then $\lam = (2,1,1,1)$.

Assume now that $n=2k$ but $\lam \neq (k,k),(k,k-1,1)$. Then $\lam_2 \leq k-2$ and we can push $\lam_2$ up to $k-2$. 
In particular, if $k \geq 5$, then $[\lam] \geq [(k,k-2,2)] = 8k-16 \geq 2.4n$. If $k = 4$ but $\lam \neq (4,2,2)$, then
$[\lam] \geq [(4,2,2)]+4 > 2.4n$. Otherwise $k = 3$ and $\lam = (3,1,1,1)$.

Finally, $[\lam] \geq 0$ in all the listed exceptions to the inequality $[\lam] \geq 2.4n$.
\end{proof}

\begin{lemma}\label{unip1}
For $\eps = \pm$ and $\lam = (k=\lam_1 \geq \lam_2 \geq  \ldots \geq \lam_r > 0) \vdash n$, the following statements hold.
\begin{enumerate}[\rm(i)]
\item $\psi^\lam(1)$ and $[G^\eps_n:G^\eps_\lam]_{p'}$ are both monic polynomials in $q$ with integer coefficients and of degree 
$\deg_q \psi^\lam(1) = \deg_q [G^\eps_n:G^\eps_\lam]_{p'} = b(\lam)$.
\item If $\eps = +$, then $\psi^\lam(1) \geq q^{b(\lam)} \geq q^{k(n-k)}$.
\item If $\eps = -$, then 
$$\psi^\lam(1) \geq \max\left\{\left(\frac{q}{q+1}\right)^{|\lam|-1}\cdot q^{b(\lam)} ,\frac{1}{2}q^{k(n-k)}\right\}.$$
\item $[G^-_n:G^-_\lam]_{p'} < q^{b(\lam)} < [G^+_n:G^+_\lam]_{p'}$. 
\end{enumerate}
\end{lemma}

\begin{proof}
(a) We will induct on the length $r$ of $\lam$, with the induction base $r=1$ being obvious. Also note that the statements about 
$\psi^\lam(1)$, respectively $[G^\eps_n:G^\eps_\lam]_{p'}$, being a monic polynomial in $q$ with integer coefficients, are well known.
Next, 
$$\deg_q [G^\eps_n:G^\eps_\lam]_{p'} = n(n+1)/2 -\sum^r_{i=1}\lam_i(\lam_i+1)/2 = b(\lam).$$
Now (iv) follows easily from Lemma \ref{trivial}(vi) by induction on $r$.
For the induction step to prove the remaining statements, define
$$\mu:= (\lam_2, \lam_3, \ldots ,\lam_r) \vdash (n-k),~~\nu := \mu' = (\nu_1, \nu_2, \ldots ,\nu_k),$$
where $\mu'$ is the conjugate partition of $\mu$ (of length $\leq k$ as $\lam_1=k$).
We will use the quantized hook formula
\begin{equation}\label{q-deg}
  \psi^\lam(1) =q^{a(\lam)}\frac{(q-1)(q^2-\eps^2)\cdots(q^n-\eps^n)}{\prod_{h}(q^{l(h)}-\eps^{l(h)})},
\end{equation}    
where $h$ runs over all the hooks of the Young diagram of $\lam$ and $l(h)$ denotes the length of the hook $h$ 
(see for example \cite[(21)]{Ol} or \cite{Ma}).

Let $h_1, h_2, \ldots ,h_k$ denote the hooks corresponding to the first row of $\lam$. Then 
$$\begin{aligned}\deg_q \psi^\lam(1)-\deg_q \psi^\mu(1) & = a(\lam)-a(\mu)+\sum^{n}_{i=n-k+1}i-\sum^{k}_{i=1}l(h_i)\\ 
     & = \sum^r_{i=1}(i-1)\lam_i-\sum^{r-1}_{i=1}(i-1)\lam_{i+1} + \sum^{n}_{i=n-k+1}i- \sum^{k}_{i=1}(\nu_{k-i+1}+i)\\ 
   & = \sum^r_{i=2}\lam_i - |\nu| - \sum^{k}_{i=1}i + \sum^{n}_{i=n-k+1}i
      = k(n-k) = b(\lam)-b(\mu).\end{aligned}$$
As $\deg_q \psi^\mu(1) = b(\mu)$ by the induction hypothesis, we have $\deg_q \psi^\lam(1) = b(\lam)$, and so 
(i) holds.   

\smallskip
(b) Suppose $\eps = +$. Then by (i) we have
$$\frac{\psi^\lam(1)/q^{b(\lam)}}{\psi^\mu(1)/q^{b(\mu)}}= \frac{\prod^k_{i=1}(1-1/q^{n-i+1})}{\prod^k_{i=1}(1-1/q^{l(h_i)})} \geq 1$$
since $l(h_1) \leq n$, $l(h_2) \leq n-1$, $\ldots$, $l(h_k) \leq n-k+1$. As $\psi^\mu(1) \geq q^{b(\mu)}$ by the induction hypothesis, we get
$\psi^\lam(1) \geq q^{b(\lam)}$. Since $b(\lam) = b(\mu)+k(n-k) \geq k(n-k)$, (ii) holds.

\smallskip
(c) From now on we assume $\eps = -$. Then (i), \eqref{q-deg}, and Lemma \ref{trivial}(iv) imply that
$$\frac{\psi^\lam(1)}{q^{b(\lam)}} > \frac{1}{\prod_{h}(1-(-1/q)^{l(h)})}.$$
In particular, if at least one of the hooks $h$ has even length, then $\frac{\psi^\lam(1)}{q^{b(\lam)}} > (\frac{q}{q+1})^{n-1}$.
The same estimate holds if at least two hooks have odd length $\geq 3$, since $(1+1/q^3)^2 < 1+1/q$. The only $\lam$ that have
no hook of even length and at most one hook of odd length $\geq 3$ are $(1)$ and $(2,1)$, for which (iii) also holds. 

Suppose in addition that $k \leq n-2$. By Lemma \ref{trivial}(iii) we have 
$$\frac{(q^n-\eps)(q^{n-1}-\eps^{n-1}) \ldots (q^{n-k+1}-\eps^{n-k+1})}{q^{n+(n-1) + \ldots + (n-k+1)}} \geq \frac{q^4-1}{q^4} \geq 15/16.$$
Note that $l(h_1) > l(h_2) > \ldots > l(h_k) \geq 1$. Hence
$$\frac{q^{l(h_1) + \ldots + l(h_k)}}{(q^{l(h_1)}-\eps^{l(h_1)}) \ldots (q^{(l(h_k)}-\eps^{l(h_k)})} > \frac{1}{\prod^\infty_{i=0}(1+1/q^{2i+1})}   > 
  \frac{2}{3}\cdot\frac{8}{9}(1-\sum^{\infty}_{i=0}1/q^{2i+5}) \geq \frac{46}{81}.$$
Together with (i) and \eqref{q-deg}, the last two inequalities imply that
$$\frac{\psi^\lam(1)/q^{b(\lam)}}{\psi^\mu(1)/q^{b(\mu)}} \geq \frac{15}{16} \cdot \frac{46}{81} > \frac{1}{2},$$
and so $\psi^\lam(1) \geq q^{b(\lam)-b(\mu)}/2 = q^{k(n-k)}/2$. The same estimate holds if $n-1 \leq k \leq n$.
\end{proof}

\begin{propo}\label{degree1}
Let $\eps = \pm$, $G = \GL^\eps_n(q)$, and let $\chi \in \Irr(G)$ have level $0 \leq j \leq n$. Then the following statements holds.
\begin{enumerate}[\rm(i)]
\item $\chi(1) \leq q^{nj}$.
\item $\chi(1) \geq q^{j(n-j)}$ if $\eps = +$ and $\chi(1) \geq q^{j(n-j)}/2$ if $\eps = -$.
\end{enumerate}
\end{propo}

\begin{proof}
Since $\tau_n^j$ and $\zeta^j_n$ have degree $q^{nj}$, (i) follows from the definition of $\cl(\chi)$. Note that (ii) is obvious if 
$j \in \{0,n\}$, so we will assume $1 \leq j \leq n-1$.
We can now apply Theorems \ref{main1-gl}, \ref{main1-gu}, and Corollary \ref{product2} and see that $\chi = \pm R^G_L(\al \otimes \beta)$, where 
$L = G^\eps_a \times G^\eps_b$, where $0 \leq a \leq a+b = n$, $\al \in \Irr(G^\eps_a)$, and, up to a linear character,
$\beta = \psi^\lam$ for some $\lam \vdash b$ with the longest part $\lam_1 = k = n-j$. In particular,
$$\chi(1) = [G:L]_{p'}\al(1)\beta(1) \geq [G:L]_{p'}\psi^\lam(1).$$
If $a=0$ then we are done by Lemma \ref{unip1}(ii), (iii) (applied to $\psi^\lam$). We will now assume that $a \geq 1$.
If $\eps = +$, then by Lemma \ref{trivial}(vi) and 
Lemma \ref{unip1}(ii) we have 
$$\chi(1) \geq q^{ab+k(b-k)} \geq q^{ak+k(b-k)} = q^{k(n-k)} = q^{j(n-j)}.$$
Suppose that $\eps = -$. Applying Lemma \ref{trivial}(vi) and Lemma \ref{unip1}(iii) we obtain that
$$\chi(1) \geq (1/4) q^{ab+k(b-k)} = (1/4)q^{k(n-k)+ a(b-k)} \geq q^{j(n-j)}/2$$
if $k < b$. If $k=b$, then $\beta(1) = q^{k(b-k)}$, so by Lemma \ref{trivial}(vi) we have 
$$\chi(1) \geq (1/2) q^{ab} = q^{j(n-j)}/2.$$
\end{proof}

Next we aim to bound $\chi(1)$ from below when $\cl(\chi) \geq n/2$. First we begin with unipotent characters.

\begin{lemma}\label{unip2}
Let $\lam = (k=\lam_1\geq \lam_2 \geq \ldots \geq \lam_r \geq 0) \vdash n$ and let $\psi = \psi^\lam \in \Irr(\GL^\eps_n(q))$ 
for $n \geq 2$ and $\eps = \pm$. If $k \leq n/2$, then $\psi(1) \geq q^{n^2/4}$. 
\end{lemma}

\begin{proof}
Note that if $k = 1$ then $\psi$ is just the Steinberg character, of degree $q^{n(n-1)/2}$, and the statement holds in this case. So
we will assume $2 \leq k \leq n/2$. Then Lemma \ref{unip1}(i) implies that
\begin{equation}\label{unip2-1}
  2b(\lam) = 2\deg_q[G^+_n:G^+_\lam]_{p'} = n^2-\sum^r_{i=1}\lam_i^2 = [\lam]/2 + n^2/2.
\end{equation}  
Hence Lemma \ref{unip1}(ii) and Lemma \ref{sum1} immediately imply $\psi(1) \geq q^{n^2/4}$ in the case $\eps = +$. We may now assume $\eps = -$. 
Since $(q+1)/q \leq 3/2 < q^{0.6}$, by Lemma \ref{unip1}(iii) and \eqref{unip2-1} we have 
$$\psi(1) > q^{b(\lam) - 0.6(n-1)} \geq q^{n^2/4}$$
if $[\lam] \geq 2.4(n-1)$. In particular, we are done if $[\lam] \geq 2.4n$.
In the cases of exceptions to the latter inequality, as listed in Lemma \ref{sum1}, one can check using explicit formulae for 
$\psi^\lam(1)$ (see \cite[\S13.8]{C}) and estimates in Lemma \ref{trivial} that $\psi(1) \geq q^{n^2/4}$ as well.
\end{proof}

We will need an extension of Lemma \ref{unip2} in the case of unitary groups:

\begin{lemma}\label{unip3}
Let $G = \GU_n(q)$ with $n \geq 2$, and let $\chi \in \Irr(G)$ belong to the rational series $\EC(G,(s))$, where all eigenvalues of the 
semisimple element $s \in G$ belong to $\mu_{q+1}$. Suppose that $\cl(\chi) \geq n/2$. Then either 
$\chi(1) > q^{n^2/4}$, or $\chi(1) \geq (q-1)q^{n^2/4-1}$ and one of the following cases occurs:
\begin{enumerate}[\rm(i)]
\item $n = 2k$, $\CB_G(s) = \GU_k(q) \times \GU_k(q)$, and $\chi(1)=[G:\CB_G(s)]_{p'}$.
\item $2 \leq n \leq 4$.
\end{enumerate}
\end{lemma}

\begin{proof}
(i) By the assumptions, we can decompose the natural module $V = \FF_{q^2}^n$ into an orthogonal sum $\oplus^m_{i=1}V_i$,
where the $V_i$ are distinct eigenspaces for $s$, say with eigenvalue $\eps_i \in \mu_{q+1}$. Setting $a_i:=\dim_{\FF_{q^2}}V_i$
and $L := \CB_G(s) = \GU(V_1) \times \ldots \times \GU(V_m)$, by \eqref{gu-6} we then have 
$$\chi = \pm R^G_L(\al_1 \otimes \al_2 \otimes \ldots \otimes \al_m),$$
where $\al_i = \nu_i\psi^{\lam_i}$, $\lam_i \vdash a_i$, and $\nu_i$ is a linear character of $\GU(V_i) = \GU_{a_i}(q)$. Let 
$k$ denote the largest among all the parts of all $\lam_i$. Then $\cl(\chi) \geq n/2$ implies that $k \leq n/2$.
By Lemma \ref{unip2} we may assume that $m \geq 2$. One can check by direct computation that 
$\chi(1) \geq (q-1)q^{n^2/4-1}$ when $2 \leq n \leq 4$, so we will assume that $n \geq 5$.

\smallskip
(ii) Here we consider the case $k=1$, i.e. $\lam_i = (1)$ for all $i$ and $m=n$. As $n \geq 5$, by 
Lemma \ref{trivial}(iv) we have 
$$\chi(1) = |\GU_n(q)|_{p'}/(q+1)^n > q^{n(n+1)/2-1.6n} > q^{n^2/4}$$
(where we also used the trivial estimate $q+1 < q^{1.6}$). 

\smallskip
(iii) From now on we may assume that $m,k \geq 2$.
According to Lemma \ref{unip1}(iii), $\al_i(1) \geq (2/3)^{a_i-1}q^{b(\lam_i)}$. Furthermore, by Lemma 
\ref{trivial}(iii), (iv),
\begin{equation}\label{unip3-1}
  |\GU_{a_i}(q)|_{p'} \leq (q+1)\prod^{a_i}_{j=2}q^j \leq (3/2)q^{a_i(a_i+1)/2},~~|\GU_n(q)|_{p'} > q^{n(n+1)/2}.
\end{equation}  
It follows that
$$[G:L]_{p'} > (2/3)^mq^{n(n+1)/2 - \sum^m_{i=1}a_i(a_i+1)/2} = (2/3)^mq^{(n^2-\sum^m_{i=1}a_i^2)/2}.$$
Putting all these estimates together, we obtain
$$\chi(1) = [G:L]_{p'}\prod^m_{i=1}\al_i(1)
    > (2/3)^{m+\sum^m_{i=1}(a_i-1)}q^{(n^2-\sum^m_{i=1}a_i^2)/2+\sum^m_{i=1}b(\lam_i)} 
    = (2/3)^nq^{b(\mu)}$$
where the parts of the partition $\mu$ consist of all parts of all $\lam_i$, put together in decreasing order. Note that 
$b(\mu) = ([\mu]+n^2)/4$, and $q^{0.6} > 3/2$. It follows that $\chi(1) > q^{n^2/4}$ if $[\mu] \geq 2.4n$.   
Thus it remains to consider the exceptions to the latter inequality, listed in Lemma \ref{sum1}.

\smallskip
(iv) Consider the case $\mu = (k,k,1)$. If $m=3$, then \eqref{unip3-1} implies that 
$$\begin{aligned}\chi(1) & = [\GU_{2k+1}(q):(\GU_k(q) \times \GU_k(q) \times \GU_1(q))]_{p'} \\
    & > (2/3)^3q^{n(n+1)/2-k(k+1)-1} > q^{k^2+2k-2} > q^{n^2/4}\end{aligned}$$
if $k \geq 3$. It is easy to check that $\chi(1) > q^{n^2/4}$ also holds when $k=2$.

Suppose $m = 2$. If $\{\lam_1,\lam_2\} = \{(k,k),(1)\}$, then by Lemma \ref{trivial}(vi) we have 
$$\chi(1) = \frac{q^k\prod^{2k}_{i=k+2}(q^i-(-1)^i)}{\prod^k_{i=2}(q^i-(-1)^i)} \cdot \frac{q^{2k+1}+1}{q+1} 
    = \frac{q^k\prod^{2k+1}_{i=k+2}(q^i-(-1)^i)}{\prod^k_{i=1}(q^i-(-1)^i)} \geq \frac{q^{k(k+1)+k}}{2} > q^{n^2/4}.$$
If $\{\lam_1,\lam_2\} = \{(k,1),(k)\}$, then by Lemma \ref{trivial}(vi) we have 
$$\chi(1) = \frac{\prod^{2k+1}_{i=k+2}(q^i-(-1)^i)}{\prod^k_{i=1}(q^i-(-1)^i)} \cdot \frac{q^{k+1}-(-1)^k q}{q+1} 
    \geq \frac{q^{k(k+1)+k}}{4} > q^{n^2/4}$$
when $k \geq 3$.  It is easy to check that $\chi(1) > q^{n^2/4}$ also holds when $k=2$.

\smallskip
(v) Next suppose that $\mu = (k,k)$. Since $m \geq 2$, we have $\lam_1 = \lam_2 = (k)$ and so
$$\chi(1) = [\GU_{2k}(q):(\GU_k(q) \times \GU_k(q))]_{p'} > (q-1)q^{n^2/4-1}$$
by Lemma \ref{trivial}(vi). In fact, if the unipotent character $\psi$ of $\CB_G(s)$ corresponding to $\chi$ is not the 
principal character, then $q|\psi(1)$ and so $\chi(1) > q^{n^2/4}$.

\smallskip
(vi) Here we consider the case $\mu = (k,k-1,1)$. Direct computations show that $\chi(1) \geq (q-1)q^{n^2/4-1}$ if 
$k=2$. So we will assume $k \geq 3$.
If $m=3$, then \eqref{unip3-1} implies that 
$$\begin{aligned}\chi(1) & = [\GU_{2k}(q):(\GU_k(q) \times \GU_{k-1}(q) \times \GU_1(q))]_{p'} \\
    & > (2/3)^3q^{k(2k+1)-k^2-1} > q^{k^2+k-3} \geq q^{n^2/4}\end{aligned}.$$

Suppose $m = 2$. If $\{\lam_1,\lam_2\} = \{(k,k-1),(1)\}$, then by Lemma \ref{trivial}(vi) we have 
$$\begin{aligned}\chi(1) & = \frac{q^{k-1}(q^2-1)\prod^{2k-1}_{i=k+2}(q^i-(-1)^i)}{\prod^{k-1}_{i=1}(q^i-(-1)^i)} \cdot \frac{q^{2k}-1}{q+1} \\
    & = \frac{q^{k-1}(q-1)\prod^{2k}_{i=k+2}(q^i-(-1)^i)}{\prod^{k-1}_{i=1}(q^i-(-1)^i)} \geq q^{k^2+k-3}  \geq q^{n^2/4}.\end{aligned}$$
If $\{\lam_1,\lam_2\} = \{(k,1),(k-1)\}$, then by Lemma \ref{trivial}(vi) we have 
$$\chi(1) = \frac{\prod^{2k}_{i=k+2}(q^i-(-1)^i)}{\prod^{k-1}_{i=1}(q^i-(-1)^i)} \cdot \frac{q^{k+1}-(-1)^k q}{q+1} 
    \geq q^{k^2+k-3} \geq q^{n^2/4}.$$
If $\{\lam_1,\lam_2\} = \{(k-1,1),(k)\}$, then by Lemma \ref{trivial}(vi) we have 
$$\chi(1) = \frac{\prod^{2k}_{i=k+1}(q^i-(-1)^i)}{\prod^{k}_{i=1}(q^i-(-1)^i)} \cdot \frac{q^k+(-1)^k q}{q+1} 
    \geq q^{k^2+k-3} \geq q^{n^2/4}.$$

\smallskip
(vii) Finally, one can check by direct computations that $\chi(1) > q^{n^2/4}$ in the remaining cases, where 
$\mu = (4,2,2)$, $(3,2,2)$, $(3,1,1,1)$, $(2,2,2)$, and $(2,1,1,1)$. 
\end{proof}

\begin{lemma}\label{step}
Let $n = a+b$ with $a \in \ZZ_{\geq 2}$, $b \in \ZZ_{\geq 1}$, $\eps = \pm$, and let 
$$\gam = \pm R^{\GL^\eps_n(q)}_{\GL^\eps_a(q) \times \GL^\eps_b(q)}(\al \otimes \beta),$$
where $\al$ is a character of $\GL^\eps_a(q)$, $\beta$ is a character of $\GL^\eps_b(q)$, and the sign for $\gam$ is chosen so that
$\gam(1) > 0$. Then $\gam(1) > q^{(a+b)^2/4}$ if at least one of the following conditions holds:
\begin{enumerate}[\rm(i)]
\item $\eps = +$, $\al(1) \geq (9/16)(q-1)q^{a^2/4-1}$, and $\beta(1) \geq q^{b^2/4-2}$. Morever, 
$\al(1) \geq q-1$ if $(a,b) = (2,1)$.
\item $\eps = +$, $\al(1) \geq q^{a^2/4-2}$, $b \geq 2$, and $\beta(1) \geq q^{k(b-k)}$ for some $k \in \ZZ$ with $b/2 \leq k \leq \min\{b,(a+b)/2\}$.
\item $\eps = -$, $\al(1) \geq q^{a^2/4}$, and $\beta(1) \geq q^{k(b-k)}/2$ with $k \in \ZZ$ and $b/2 \leq k \leq \min\{b,(a+b)/2\}$.
\item $\eps = -$, $\al(1) \geq q^{a^2/4}$, and $\beta(1) \geq q^{b^2/4-1}$. 
\end{enumerate}
\end{lemma}

\begin{proof}
(i) By Lemma \ref{trivial}(vi), $[\GL^+_n:(\GL^+_a \times \GL^+_b)]_{p'} > q^{ab}$. Also, 
$\al(1) > q^{a^2/4-2}$. Hence, if $ab \geq 8$ we have 
$$\gam(1) > q^{a^2/4+b^2/4+ab-4} \geq q^{(a+b)^2/4}.$$
Consider the case $ab < 8$. If $a=2$ and $b \geq 3$, then since $\al(1) \geq 1=q^{a^2/4-1}$, we have  
$$\gam(1) > q^{a^2/4+b^2/4+ab-3} \geq q^{(a+b)^2/4}.$$
The same argument applies if $a \geq 3$ and $b \geq 2$. If $a=b=2$, then 
$$\gam(1) > q^{ab} = q^{(a+b)^2/4}.$$
We may now assume that $b=1$. If $a \geq 3$, then 
$$[\GL^+_n:(\GL^+_a \times \GL^+_b)]_{p'} = (q^{a+1}-1)/(q-1) \geq \frac{15}{16}q^{a+1}/(q-1),$$
whence 
$$\gam(1) \geq \frac{9(q-1)}{16}q^{a^2/4-1} \cdot \frac{15}{16(q-1)}q^{a+1} > q^{a^2/4+a-1} > q^{(a+1)^2/4}.$$
If $a=2$, then $\gam(1) \geq q^3-1 > q^{9/4}$.

\smallskip
(ii) As in (i), note that $\gam(1)/q^{(a+b)^2/4} > q^A$, where 
\begin{equation}\label{for-a}
  A :=\frac{a^2}{4}-2+k(b-k)+ab-\frac{(a+b)^2}{4} \geq \frac{1}{4}((2k-b)(3b-2k)-8)
\end{equation}  
as $a \geq 2k-b$. In particular, $A > 1$, and so we so are done, if $k=b \geq 3$. If $k=b = 2$ but 
$a \geq 3$, then $A = a-3 \geq 0$. If $k=b=a=2$, then we have 
$\gam(1) \geq q^{ab} = q^{(a+b)^2/4}$.  

So we may assume $k \leq b-1$, and so $3b-2k \geq b+2 \geq 4$.
If $2k-b \geq 2$, then $A \geq 0$. If $2k-b = 0$ then $A = (ab-4)/2 \geq 0$. If $2k-b=1$, then $b \geq 3$ and 
$A = ab-9/2 \geq 3/2$.

\smallskip
(iii) First we consider the case $b = 1$ and use the trivial bound $\beta(1) \geq 1$. If $a=2$, then 
$\gam(1) \geq q(q^2-q+1) > q^{9/4}$. If $a \geq 3$, then 
$$\gam(1) \geq \frac{q^{a+1}-(-1)^{a+1}}{q+1}q^{a^2/4} > q^{a^2/4+a-1} > q^{(a+1)^2/4}.$$ 
If $k=a=b=2$, then $\gam(1) \geq q(q^2+1)(q^2-q+1) > q^4$.
Now we may assume that $b \geq 2$ and $(k,a,b) \neq (2,2,2)$. By Lemma \ref{trivial}(vi), 
$$[\GL^-_n:(\GL^-_a \times \GL^-_b)]_{p'} \geq (5/8)q^{ab} > q^{ab-1}$$
and $\beta(1) \geq q^{k(b-k)-1}$. It follows that $\gam(1)/q^{(a+b)^2/4} > q^A$, where
$A$ is defined in \eqref{for-a}. As shown in (ii), $A \geq 0$, and so we are done. 

\smallskip
(iv) The case $b=1$ follows from the same arguments as in (iii). Suppose $b \geq 2$. As in (iii), we now
have $\gam(1) \geq q^{a^2/4+b^2/4+ab-2} \geq q^{(a+b)^2/4}$.
\end{proof}

\begin{lemma}\label{cent}
Let $n = md$ with $m \in \ZZ_{\geq 1}$ and $d \in \ZZ_{\geq 2}$. 
\begin{enumerate}[\rm(i)]
\item If $d = 2$, then $[\GL_n(q):\GL_m(q^d)]_{p'} > (9/16)(q-1)q^{n^2/4-1}$.
\item If $d \geq 3$, then $[\GL_n(q):\GL_m(q^d)]_{p'} > q^{n^2/4}$ unless $(n,d,q) = (3,3,2)$ in which case 
         $[\GL_n(q):\GL_m(q^d)]_{p'} > (q-1)q^{n^2/4-1}$.
\item $[\GU_n(q):\GL_m(q^d)]_{p'} > q^{n^2/4}$ if $2 | d$.         
\item $[\GU_n(q):\GU_m(q^d)]_{p'} > (1.49)q^{n^2/4}$ if $2 \nmid d$.
\end{enumerate}
\end{lemma}

\begin{proof}
(i) By Lemma \ref{trivial}(i)  we have 
$$[\GL_n(q):\GL_m(q^d)]_{p'} = \prod^{m}_{i=1}(q^{2i-1}-1) > \frac{9}{16} \cdot \frac{q-1}{q}q^{\sum^{m}_{i=1}(2i-1)} = \frac{9}{16}(q-1)q^{n^2/4-1}.$$

(ii) The statement can be checked directly if $n \leq 4$, so we will assume that $n \geq 5$. Now by Lemma \ref{trivial}(ii)  we have 
$$[\GL_n(q):\GL_m(q^d)]_{p'} = \prod^n_{i=1,d \nmid i}(q^i-1) > \frac{9}{16} \cdot \frac{q-1}{q}q^{\sum^n_{i=1,d \nmid i}i} >
    q^{n^2/2-n^2/2d-2} > q^{n^2/4}.$$
    
(iii) Here we have 
$$[\GU_n(q):\GL_m(q^d)]_{p'} \geq [\GU_n(q):\GL_{n/2}(q^2)]_{p'} = \prod^{n/2}_{i=1}(q^{2i-1}+1) > q^{n^2/4}.$$

(iv)  By Lemma \ref{trivial}(iii) we have 
$$|\GU_m(q^d)|_{p'} =  (q^d+1)\prod^m_{i=2}(q^{di}-(-1)^i) \leq (q^d+1)\prod^m_{i=2}q^{di} \leq \frac{9}{8}\prod^m_{i=1}q^{di} = \frac{9}{8}q^{md(m+1)/2}.$$
It then follows from Lemma \ref{trivial}(iv) that   
$$[\GU_n(q):\GU_m(q^d)]_{p'}  > \frac{q^{n(n+1)/2}}{(9/8)q^{md(d+1)/2}} = \frac{8}{9}q^{n^2/2-n^2/2d} \geq \frac{8}{9}q^{n^2/3} > (1.49)q^{n^2/4}.$$
\end{proof}

Recall the notation \eqref{gl-1} for irreducible characters of $\GL_n(q)$. 

\begin{propo}\label{gl-big}
Let $n \geq 2$ and let $\chi \in \Irr(\GL_n(q))$ be of level $\cl(\chi) \geq n/2$. Then either $\chi(1) \geq q^{n^2/4}$, or one of the following statements holds.
\begin{enumerate}[\rm(i)]
\item $\chi = S(s,(n/2))$ with $\deg(s) = 2$, $\chi(1) = \prod^{n/2}_{i=1}(q^{2i-1}-1) > (9/16)(q-1)q^{n^2/4-1}$.  
\item $(n,q) = (3,2)$, $\chi = S(s,(1))$ with $\deg(s) = 3$, and $\chi(1) = 3 > (q-1)q^{n^2/4-1}$.
\end{enumerate}
\end{propo}

\begin{proof}
(a) We represent $\chi$ in the form \eqref{gl-1}, where $d_1 \geq d_2 \geq \ldots \geq d_m \geq 1$. First we consider the case $m = 1$
and $d=d_1 \geq 2$. If $d \geq 3$, or $d = 2$ and $\lam_1 = (n/2)$, then the statement follows from Lemma \ref{cent}(i), (ii). Suppose
$d=2$ but $\lam_1 \neq (n/2)$. Then the unipotent character $\psi$ of $\GL_{n/2}(q^2)$ in \eqref{gl-2} corresponding to $\chi$ has 
degree divisible by $q^2$, whence using Lemma \ref{cent}(i) we have
$$\chi(1) > \frac{9}{16}(q-1)q^{n^2/4-1} \cdot q^2 > q^{n^2/4}.$$
We have also shown that 
\begin{equation}\label{gl-big1}
  \deg S(s_i,\lam_i) \geq \frac{9}{16}(q-1)q^{k_i^2d_i^2/4-1}
\end{equation}  
if $d_i > 1$.

\smallskip
(b) Here we consider the case $d_1=1$, and let $k$ denote the largest among all the parts of all $\lam_i$, $1 \leq i \leq m$. Since 
$\cl(\chi) \geq n/2$, we must have by Theorem \ref{main1-gl} that $k \leq n/2$. By Lemma \ref{unip1}(ii), (iv) we have $\chi(1) \geq q^N$, where 
$$N:= \deg_q[G^+_n:(G^+_{k_1} \times \ldots \times G^+_{k_m})]_{p'} + \sum^m_{i=1}b(\lam_i) = b(\mu),$$
and the partitions of $\mu \vdash n$ consist of all parts of all $\lam_i$, put in decreasing order; in particular, the longest part $\mu_1$ of
$\mu$ is $k$. Hence $b(\mu) \geq n^2/4$, as shown in the proof of Lemma \ref{unip2}.

\smallskip
(c) We may now assume that $m \geq 2$, $d_1 \geq \ldots \geq d_t \geq 2$ for some $1 \leq t \leq m$; and furthermore 
$d_{t+1} = \ldots = d_m = 1$ if $t < m$, in which case we 
let $k$ denote the largest among all the parts of all $\lam_i$, $t+1 \leq i \leq m$.  Also, set $a := \sum^t_{i=1}k_id_i$ and $b := \sum^m_{i=t+1}k_i$. Then 
$\chi = R^{\GL_n}_{\GL_a \times \GL_b}(\al \otimes \beta)$, where 
$$\al := S(s_1,\lam_1) \circ \ldots \circ S(s_t,\lam_t),~~\beta := S(s_{t+1},\lam_{t+1}) \circ \ldots \circ S(s_m,\lam_m).$$ 

Note that $\deg S(s_i,\lam_i) = q-1$ if $(k_i,d_i) = (1,2)$; in particular, $\al(1) \geq q-1$. Applying 
\eqref{gl-big1} and Lemma \ref{step}(i), we get $\al(1) \geq (9/16)(q-1)q^{a^2/4-1}$ if $t=1$ and $\al(1) > q^{a^2/4}$ if $t \geq 2$.
In particular, we are done if $t=m$.

We may now assume that $t < m$ and $\al(1) \geq \max\{(9/16)(q-1)q^{a^2/4-1},q-1\}$. If $k \leq b/2$ (in particular, $b \geq 2$), 
then $\beta(1) \geq q^{b^2/4}$ as shown in (b), whence $\chi(1) > q^{n^2/4}$ by Lemma \ref{step}(i). 
Finally, suppose that 
$b \geq k > b/2$. Since $\cl(\chi) \geq n/2$, we again have that $k \leq n/2$. Also, $\beta(1) \geq q^{k(b-k)}$ by Proposition \ref{degree1}(ii).
It follows, by Lemma \ref{step}(i) for $b=1$ and by Lemma \ref{step}(ii) for $b \geq 2$, that $\chi(1) > q^{n^2/4}$.  
\end{proof}

\begin{propo}\label{gu-big}
Let $n \geq 2$ and let $\chi \in \Irr(\GU_n(q))$ be of level $\cl(\chi) \geq n/2$. Then either $\chi(1) \geq q^{n^2/4}$, or 
$\chi(1) \geq (q-1)q^{n^2/4-1}$ and one of the following cases occurs:
\begin{enumerate}[\rm(i)]
\item $n = 2k$, $\CB_G(s) = \GU_k(q) \times \GU_k(q)$, and $\chi(1)=[G:\CB_G(s)]_{p'}$.
\item $2 \leq n \leq 4$.
\end{enumerate}
\end{propo}

\begin{proof}
Let $\chi \in \EC(G,(s))$ where $s \in G$ is semisimple. We will represent $\chi$ in the form of \eqref{gu-6} and \eqref{gu-5}. Let 
$a := \dim_{\FF_{q^2}}V^0$ and $b := \dim_{\FF_{q^2}}V^1$. 

\smallskip
(a) Suppose that $b > 0$. Then we can decompose $V^1$ into an orthogonal sum of $s$-invariant non-degenerate subspaces, say $V_j$ of 
dimension $b_j$ over $\FF_{q^2}$, $1 \leq j \leq t$, and $b = \sum^t_{j=1}b_j$, in such a way that $\CB_{\GU(V_j)}(s)$ is either 
$\GL_{b_j/d_j}(q^{d_j})$ with $2|d_j$, or $\GU_{b_j/d_j}(q^{d_j})$ with $2 \nmid d_j > 1$. By Lemma \ref{cent}(iii), (iv),
\begin{equation}\label{gu-big1}
  [\GU(V_j):\CB_{\GU(V_j)}(s)]_{p'} > q^{b_j^2/4},
\end{equation}  
and furthermore $b_j \geq 2$. Now, a repeated application of Lemma \ref{step}(iv) using \eqref{gu-big1} shows that 
\begin{equation}\label{gu-big2}
  [\GU(V^1):\CB_{\GU(V^1)}(s)]_{p'} > q^{b^2/4}.
\end{equation} 
In particular, $\chi(1) > q^{n^2/4}$ if $a=0$.

\smallskip
(b) If $b=0$, then we are done by Lemma \ref{unip3}. So we will assume that $a,b > 0$. Let $k$ denote the largest among all the parts of the partitions 
$\gam_\eps$, $\eps \in \mu_{q+1}$. Since $\cl(\chi) \geq n/2$, we have $k \leq n/2$ by Theorem \ref{main1-gu}. By Lemma \ref{product1}(ii) we can also write 
$$\chi = \pm R^{\GU_n(q)}_{\GU_a(q) \times \GU_b(q)}(\al \otimes \beta),$$
where 
$$\al = \pm R^{\GU_a(q)}_{\CB_{\GU_a(q)}(s)}\left(\hs \otimes_{\eps \in \mu_{q+1}}\psi^{\gam_\eps}\right),~~ 
    \beta = \pm R^{\GU_b(q)}_{\CB_{\GU_b(q)}(s)}(\hs \otimes \psi_1).$$
Now $\al(1) \geq q^{k(a-k)}/2$ by Proposition \ref{degree1}, and $\beta(1) > q^{b^2/4}$ by \eqref{gu-big2}; also, $b \geq 2$.     
Applying Lemma \ref{step}(iii), we obtain that $\chi(1) > q^{n^2/4}$ if $k \geq a/2$. If $k < a/2$, then $\al(1) \geq q^{a^2/4-1}$ by Lemma
\ref{unip3}, whence $\chi(1) > q^{n^2/4}$ by Lemma \ref{step}(iv). 
\end{proof}

Now we can prove the main result of this section, Theorem \ref{main-degree}, which we restate below.

\begin{theor}
Let $n \geq 2$, $\eps = \pm$, and $G = \GL^\eps_n(q)$. Set $\kappa_+ = 1$ and $\kappa_- = 1/2$. Let $\chi \in \Irr(G)$ have 
level $j = \cl(\chi)$. Then the following statements hold.
\begin{enumerate}[\rm(i)]
\item $\kappa_\eps q^{j(n-j)} \leq \chi(1) \leq q^{nj}$.
\item If $j \geq n/2$, then $\chi(1) > (9/16)(q-1)q^{n^2/4-1}$ if $\eps = +$ and $\chi(1) \geq (q-1)q^{n^2/4-1}$ if $\eps = -$.
In particular, $\chi(1) > q^{n^2/4-2}$ if $\cl(\chi) \geq n/2$.
\item If $n \geq 7$ and  $\lceil (1/n)\log_q \chi(1) \rceil < \sqrt{n-1}-1$ then 
$$\cl(\chi) = \left\lceil \frac{\log_q \chi(1)}{n} \right\rceil.$$
\end{enumerate}
\end{theor}

\begin{proof}
(i) is Proposition \ref{degree1}, and (ii) follows from Propositions \ref{gl-big} and \ref{gu-big}. For (iii), we have that
$j_0 := \lceil(1/n)\log_q \chi(1) \rceil < \sqrt{n-1}-1$ by hypothesis. In particular,
$$\chi(1) \leq q^{nj_0} < q^{n(\sqrt{n-1}-1)} < q^{n^2/4-2}$$
as $n \geq 7$. It then follows from (ii) that $j < n/2$. Also, $\chi(1) \leq q^{nj}$ by (i), whence $(1/n)\log_q \chi(1) \leq j$ and 
so $j \geq j_0$. Suppose that $j > j_0$. Then $j_0+1 \leq j < n/2$ and $j_0 < \sqrt{n-1}-1$, and so
$$j(n-j) \geq (j_0+1)(n-j_0-1) > nj_0+1.$$
Combined with (i), this implies that $\chi(1) \geq q^{j(n-j)-1} > q^{nj_0}$, a contradiction. 
\end{proof}

\begin{corol}\label{dual5}
Let $G = \GL^\eps_n(q)$ with $n \geq 2$ and $\eps = \pm$. For $\chi \in \Irr(G)$, let $\chi^*$ denote the Alvis-Curtis dual of $\chi$.
Then 
$$\chi(1) \cdot \chi^*(1) > q^{\frac{n^2}{4}-2}.$$ 
\end{corol}

\begin{proof}
The statement follows from Theorem \ref{main-degree}(ii) if $\cl(\chi) \geq n/2$ or $\cl(\chi^*) \geq n/2$. If $\cl(\chi),\cl(\chi^*) < n/2$,
then by Proposition \ref{dual4} we have $\cl(\chi) = \cl(\chi^*) = (n-1)/2$, in which case by Theorem \ref{main-degree}(i) we have
$\chi(1)\cdot\chi^*(1) \geq q^{(n^2-1)/2-2} > q^{n^2/4-2}$.
\end{proof}

\section{Bounds on character values}
\begin{propo}\label{m-bound}
There is an explicit function 
$$f=f(C,m,k):\RR_{\geq 1} \times \ZZ_{\geq -1} \times \ZZ_{\geq 0} \to \RR_{\geq 1}$$ 
such that, for any $C \in \RR_{\geq 1}$, $m \in \ZZ_{\geq -1}$, $k \in \ZZ_{> 0}$, the following property 
$\AC(C,m,k)$ holds:

\begin{tabular}{ll}
  $\AC(C,m,k):$ & 
    \begin{tabular}{l}
       There exists some $\delta = \delta(C,m,k) \in [1/2^{m+1},1/2^m)$\\
       such that, for any prime power $q$, any $\eps = \pm$,\\
       any $G:= \GL^\eps_n(q)$ with $n \geq f(C,m,k)$, \\
       any $g \in G$ with $|\CB_G(g)| \leq q^{Cn}$, and any $\chi \in \Irr(G)$ of level $k$,\\
       $|\chi(g)| \leq \chi(1)^\delta$. 
     \end{tabular}
\end{tabular}
\end{propo}

\begin{proof}
(i) We will prove $\AC(C,m,k)$ by induction on 
$m \geq -1$. Certainly, we can take $f(C,-1,k) = 1$ and $\delta(C,-1,k)=1$. 
For the induction step, assume $m \geq 0$ and consider any $g \in G$ 
with $|\CB_G(g)| \leq q^{Cn}$; in particular, 
$$|\chi(g)| \leq q^{Cn/2},~\forall \chi \in \Irr(G).$$

\smallskip
(ii) Now we consider any $n \geq N_1:= 2^{m+2}C+8$. This implies that 
$n^2/4-2 \geq 2^mCn$. Hence, if $\cl(\chi) = j \geq n/2$, then by Theorem \ref{main-degree}(ii) we have that
$\chi(1) \geq q^{2^mCn}$, and so 
$$|\chi(g)| \leq q^{Cn/2} \leq \chi(1)^{1/2^{m+1}}.$$
Next suppose that $\cl(\chi) = j \geq 2^{m+1}C$ but $j < n/2$. Then $2 \leq j \leq (n-1)/2$, whence 
$jn/2-j^2-1 \geq j/2-1 \geq 0$. It follows by Theorem \ref{main-degree}(i) that $\chi(1) \geq q^{j(n-j)-1} \geq q^{jn/2}$ and so
we again have 
$$|\chi(g)| \leq q^{Cn/2} \leq \chi(1)^{1/2^{m+1}}.$$
Thus $\AC(C,m,j)$ holds for all $n \geq N_1$ and all $j \geq 2^{m+1}C$, by taking $\delta(C,m,j) = 1/2^{m+1}$.

\smallskip
(iii) We will now prove $\AC(C,m,k)$ by {\it backward} induction on $k \geq 0$ assuming $n \geq 2N_1$, $k < 2^{m+1}C$,
and may therefore assume that $\AC(C,m,j)$ holds for 
any $j$ with $k+1 \leq j \leq 2k$, as well as that $\AC(C,m-1,j)$ holds for all $j \geq 0$.

Consider any $\chi \in \Irr(G)$ with $\cl(\chi) = k$. If $k = 0$, then $\chi(1) = 1$ by Theorem \ref{main-degree}(i) and 
so we can take any $\delta \geq 1/2^{m+1}$. So we will assume $k \geq 1$.
Let $V = \FF_q^n$, respectively $V= \FF_{q^2}^n$,
denote the natural module for $G$. By the definition of the level, $\lam\chi$ is an irreducible constituent of $\sigma^k$
for some linear character $\lam$ of $G$, where $\sigma :=\tau_n$, respectively $\zeta_n$. As $\sigma = \bar\sigma$, 
we see that $\sigma^{2k} = \chi\bar\chi+\rho$, where either $\rho = 0$ or $\rho$ is a $G$-character. Writing 
$$\chi\bar\chi = \sum^t_{i=1}a_i\chi_i,$$
where $\chi_i \in \Irr(G)$ and $a_i \in \ZZ_{>0}$, we have 
$$\sum^t_{i=1}a_i \leq \sum^t_{i=1}a_i^2 = [\chi\bar\chi,\chi\bar\chi]_G \leq [\sigma^{2k},\sigma^{2k}]_G = [\sigma^{4k},1_G]_G.$$ 
If $\eps=+$, then $[\sigma^{4k},1_G]_G$ is just the number of $G$-orbits on $\underbrace{V \times V \times \ldots \times V}_{4k}$, and 
$k \leq n/4$ by our assumptions on $n,k$, hence it is at most $8q^{4k^2}$ by Lemma \ref{orbits}. 
If $\eps=-$, then $[\sigma^{4k},1_G]_G$ is the number of $G$-orbits on $\underbrace{V \times V \times \ldots \times V}_{2k}$, 
whence it is at most $2q^{4k^2}$ by Lemma \ref{orbits}.
We have therefore shown that
\begin{equation}\label{m-bound-1}
  |\chi(g)|^2 \leq q^{4k^2+3}\max_{1 \leq i \leq t}|\chi_i(g)|.
\end{equation}  
Since $\chi_i$ is an irreducible constituent of $\sigma^{2k}$, $\cl(\chi_i) \leq 2k$. If $0 \leq \cl(\chi_i) \leq k$, then by taking 
$$n \geq N_2 := \max_{0 \leq j \leq k}f(C,m-1,j),~~\al := \max_{0 \leq j \leq k}\delta(C,m-1,j) \in [1/2^m,1/2^{m-1}),$$  
we have by $\AC(C,m-1,j)$, $0 \leq j \leq k$, and Theorem \ref{main-degree}(i) that 
\begin{equation}\label{m-bound-2}
  |\chi_i(g)| \leq |\chi_i(1)|^\al \leq q^{kn\al}.
\end{equation}  
On the other hand, if $k < \cl(\chi_i) \leq 2k$, then by taking 
$$n \geq N_3 := \max_{k <  j \leq 2k}f(C,m,j),~~\beta:= \max_{k < j \leq 2k}\delta(C,m,j) \in [1/2^{m+1},1/2^m),$$  
we have by $\AC(C,m,j)$, $k < j \leq 2k$, and Theorem \ref{main-degree}(i) that 
\begin{equation}\label{m-bound-3}
  |\chi_i(g)| \leq |\chi_i(1)|^\beta \leq q^{2kn\beta}.
\end{equation}  
It follows from \eqref{m-bound-1}--\eqref{m-bound-3} that 
\begin{equation}\label{m-bound-4}
  |\chi(g)| \leq q^{2k^2+3/2}q^{kn\gam},
\end{equation}  
where $\gam:= \max(\al/2,\beta) \in [1/2^{m+1},1/2^m)$, if $n \geq \max(2N_1,N_2,N_3)$. Now we choose 
$\delta = \delta(C,m,k)$ such that $\gam < \delta < 1/2^m$ and 
$$\begin{aligned}
    n \geq f(C,m,k) & := \max\left(2N_1,N_2,N_3,\frac{3k+3}{\delta-\gam}\right)\\
    & = \max\left(2^{m+3}C+16,\max_{0 \leq j \leq k}f(C,m-1,j),\max_{k <  j \leq 2k}f(C,m,j),\frac{3k+3}{\delta-\gam}\right).\end{aligned}$$ 
Then $(2k^2+3/2+kn\gam) \leq (kn-k^2-1)\delta$. Since $\chi(1) \geq q^{kn-k^2-1}$ by Theorem \ref{main-degree}(i), \eqref{m-bound-4}
now implies that $|\chi(g)| \leq \chi(1)^\delta$, completing the induction step of the proof.
\end{proof}

Now we can prove Theorem \ref{main-bound1} for $\GL^\eps_n(q)$, which we restate below.

\begin{theor}
There is an explicit function $h=h(C,m):\RR_{\geq 1} \times \ZZ_{\geq 0} \to \RR_{\geq 1}$
such that, for any $C \in \RR_{\geq 1}$, $m \in \ZZ_{\geq 0}$, the following statement holds.
For any prime power $q$, any $\eps = \pm$, any $G:= \GL^\eps_n(q)$ with $n \geq h(C,m)$, 
any $g \in G$ with $|\CB_G(g)| \leq q^{Cn}$, and any $\chi \in \Irr(G)$,
$$|\chi(g)| \leq \chi(1)^{1/2^m}.$$ 
\end{theor}

\begin{proof}
Consider any $g \in G$ with $|\CB_G(g)| \leq q^{Cn}$. As shown in p. (ii) of the proof of Proposition 
\ref{m-bound}, $|\chi(g)| \leq \chi(1)^{1/2^m}$ if $n \geq 2^{m+1}C+8$ and $\cl(\chi) \geq 2^mC$. On the 
other hand, by Proposition \ref{m-bound} we have 
$$|\chi(g)| < \chi(1)^{1/2^m}$$
if $\cl(\chi) = k < 2^mC$ and $n \geq f(C,m,k)$. Thus $|\chi(g)| \leq \chi(1)^{1/2^m}$ for all
$$n \geq h(C,m) := \max\left( 2^{m+1}C+8,\max_{1 \leq k < 2^mC}f(C,m,k) \right).$$
\end{proof}

Next we prove Theorem \ref{main-bound2} for $\GL^\eps_n(q)$, which we restate below.

\begin{theor}
Let $q$ be any prime power and let $G = \GL^\eps_n(q)$ with $\eps = \pm$. 
Suppose that $g \in G$ satisfies $|\CB_G(g)| \leq q^{n^2/12}$. Then
$$|\chi(g)| \leq \chi(1)^{8/9}$$
for all $\chi \in \Irr(G)$.
\end{theor}

\begin{proof}
(i) We will work with the assumption that $|\CB_G(g)| \leq q^{n^2 \delta}$ with $\delta = 1/12$.
(In fact, $\delta$ can also be chosen to be $35/418$.) 
Let $\cl(\chi) = k$, and we aim to show that
$$|\chi(g)| \leq \chi(1)^{8/9}.$$ 

\smallskip
(ii) Note that, when $1 \leq n \leq 4$ there is no element $g \in G$ such that $|\CB_G(g)| \leq q^{n^2/12}$, and that the statement is trivial if 
$\chi(1)=1$. So we will assume that $n \geq 5$ and $\chi(1) > 1$.
Furthermore, if $k = 1$, then $\chi$ is a Weil character, see Examples \ref{gl-ex} and \ref{gu-ex}. In this case, it is not
difficult to use character formulae for Weil characters, see e.g. \cite[Lemma 4.1]{TZ2}, to verify that the statement holds for $g$. So we
will assume $k \geq 2$. It then follows by Theorem \ref{main-degree}(i) that $\chi(1) \geq q^{2n-5}$. In particular, if 
$3 \leq n \leq 41$, then $n^2/24 < (8/9)(2n-5)$,
and so
$$|\chi(g)| \leq |\CB_G(g)|^{1/2} \leq q^{n^2\delta/2} = q^{n^2/24} < \chi(1)^{8/9}.$$
So we may 
assume 
\begin{equation}\label{bound2-1}
  n \geq 42,~~k \geq 2.
\end{equation}  
Suppose that $k \geq n/15$. Then Theorem \ref{main-degree} implies that 
$\chi(1) \geq q^{14n^2/225-1}$ whereas 
$|\chi(g)| \leq q^{n^2/24}$ and so $|\chi(g)| < \chi(1)^{8/9}$. So we may assume $k < n/15$, whence $k < n/11-1$
because of \eqref{bound2-1}. In this case,
$k(n-k)-1 > 10kn/11$, and so Theorem \ref{main-degree}(i) yields  
\begin{equation}\label{bound2-2}
  \chi(1) \geq q^{10kn/11}.
\end{equation}
Now, if $k \geq 5n\delta/8$, then $(8/9)(10kn/11) > n^2\delta/2$ and so \eqref{bound2-2} implies that 
$$|\chi(g)| \leq |\CB_G(g)|^{1/2} \leq q^{n^2\delta/2} < \chi(1)^{8/9}.$$
Thus we may assume that  
\begin{equation}\label{bound2-3}
  k < 5n\delta/8.
\end{equation}  

\smallskip
(iii) For some integer $1 \leq m \leq n/4k$, to be chosen later, we decompose
\begin{equation}\label{bound2-chi}
  (\chi\bar\chi)^m = \sum^s_{i=1}a_i\al_i + \sum^t_{j=1}b_j\beta_j,
\end{equation}  
where $a_i, b_j \in \ZZ_{>0}$, $\al_i \in \Irr(G)$, $\beta_j \in \Irr(G)$, and 
$$\al_i(1) < q^{n^2\delta},~~\beta_j(1) \geq q^{n^2\delta}.$$
Arguing as in p. (iii) of the proof of Proposition \ref{m-bound} and using the condition $4km \leq n$, we obtain
$$\sum^s_{i=1}a_i + \sum^t_{j=1}b_j \leq q^{4m^2k^2+3} \leq q^{19m^2k^2/4}$$
as $k \geq 2$.
Using the bounds $|\al_i(g)| \leq \al_i(1)$ and 
$$|\beta_j(g)| \leq |\CB_G(g)|^{1/2} \leq q^{n^2\delta/2} \leq \beta_j(1)/q^{n^2\delta/2},$$ 
we have
\begin{equation}\label{bound2-4}
  \begin{split}
  |\sum^s_{i=1}a_i\al_i(g)|  & \leq \sum^s_{i=1}a_i\al_i(1) \leq q^{19m^2k^2/4+n^2\delta},\\
  |\sum^t_{j=1}b_j\beta_j(g)| & \leq \frac{\sum^t_{j=1}b_j\beta_j(1)}{q^{n^2\delta/2}} \leq \frac{\chi(1)^{2m}}{q^{n^2\delta/2}}.
 \end{split} 
\end{equation}  
Now we set $\al^* := 7/8$ and choose $m \in \ZZ$ such that
\begin{equation}\label{bound2-5}
  mk \leq \frac{n\delta}{4(1-\al^*)} = 2n\delta,~~mk \leq \frac{40n\al^*}{209}=\frac{35n}{209},~~mk \geq \frac{11n\delta}{10\al^*} = \frac{44n\delta}{35}.
\end{equation}  
Note that $2n\delta-44n\delta/35 = 26n\delta/35 > 5n\delta/8 > k$. Furthermore, $2n\delta \leq 35n/209$.
Thus there exists $m \in \ZZ$ satisfying \eqref{bound2-5}. This choice of $m$ guarantees that
\begin{equation}\label{bound2-5a}
  \frac{19mk^2}{8}+\frac{n^2\delta}{2m} \leq \frac{5kn\al^*}{11}+ \frac{5kn\al^*}{11} =  \frac{10kn\al^*}{11}.
\end{equation}  
Now we choose $\al = 8/9$ and note by \eqref{bound2-1} that
$$\frac{10kn\al}{11}- \frac{10kn\al^*}{11} = \frac{5kn}{396} \geq \frac{420}{396} > 1.$$
It follows from \eqref{bound2-5a} that 
$$\frac{19mk^2}{8}+\frac{n^2\delta}{2m} \leq \frac{10kn\al}{11}-1.$$ 
Using \eqref{bound2-4} and \eqref{bound2-2} we then get
\begin{equation}\label{bound2-6}
  |\sum^s_{i=1}a_i\al_i(g)|^{1/2m} \leq q^{(19m^2k^2/4+n^2\delta)/2m} \leq \chi(1)^\al/2.
\end{equation}  
The choice \eqref{bound2-5} of $m$ also ensures that $kn(1-\al^*) \leq n^2\delta/4m$. Again, 
$$kn(1-\al^*)-kn(1-\al) = kn/72 > 1$$ 
by \eqref{bound2-1}, whence $kn(1-\al) \leq n^2\delta/4m-1$. As $\chi(1) \leq q^{kn}$ by Theorem \ref{main-degree}(i), using 
\eqref{bound2-4} we now obtain that  
\begin{equation}\label{bound2-7}
  |\sum^t_{j=1}b_j\beta_j(g)|^{1/2m} \leq \frac{\chi(1)}{q^{n^2\delta/4m}} \leq \chi(1)^\al/2.
\end{equation}  
Combining \eqref{bound2-chi}, \eqref{bound2-6} and \eqref{bound2-7} together, and recalling $m \geq 1$, we arrive at
$$\begin{aligned}|\chi(g)| & \leq \left(|\sum^s_{i=1}a_i\al_i(g)|+|\sum^t_{j=1}b_j\beta_j(g)|\right)^{1/2m} \\
    & \leq |\sum^s_{i=1}a_i\al_i(g)|^{1/2m} +  |\sum^t_{j=1}b_j\beta_j(g)|^{1/2m} \leq \chi(1)^\al,\end{aligned}$$
as stated.
\end{proof}

\section{Special linear and special unitary groups}
In this section we extend the above results to special linear and special unitary groups.

\begin{defi}\label{def:slu}
{\em For $n \in \ZZ_{\geq 1}$ and $\eps = \pm$, consider the subgroup $S := \SL^\eps_n(q)$ of $G := \GL^\eps_n(q)$. 
For any $\varphi \in \Irr(S)$, choose $\chi \in \Irr(G)$ lying above $\varphi$. Then we define the {\it level $\cl(\varphi)$ of $\varphi$} to be 
$\cl(\varphi) := \cl(\chi)$.}
\end{defi}

\begin{lemma}\label{slu-level}
In the notation of Definition \ref{def:slu}, the following statements hold.
\begin{enumerate}[\rm(i)]
\item $\cl(\varphi)$ does not depend on the choice of $\chi \in \Irr(G)$ lying above $\varphi$.
\item $\cl(\varphi)$ is the smallest $j \in \ZZ_{\geq 0}$ such that $\varphi$ is an irreducible constituent of $(\tau_n^j)|_S$ if $\eps = +$,
respectively of $(\zeta_n^j)|_S$ if $\eps = -$.
\end{enumerate}
\end{lemma}

\begin{proof}
(i) Let $J$ denote the inertia subgroup of $\varphi$ in $G$. By the Clifford correspondence, $\chi = \Ind^G_J(\psi)$ for some 
$\psi \in \Irr(J)$ lying above $\varphi$. Since $J/S$ is cyclic, $\varphi$ extends to $J$ by \cite[Corollary (11.22)]{Is2}, 
whence any $\psi' \in \Irr(J)$ lying above 
$\varphi$ is $\psi\nu$ for some $\nu \in \Irr(J/S)$ by Gallagher's theorem \cite[Corollary (6.17)]{Is2}. Now if $\chi' \in \Irr(G)$ is another character also 
lying above $\varphi$, then $\chi' = \Ind^G_J(\psi')$ for some $\psi' = \psi\nu \in \Irr(J/S)$ lying above $\varphi$. As $G/S$ is abelian,
$\nu$ extends to some (linear) $\lam \in \Irr(G/S)$. It follows that 
$$\chi' = \Ind^G_J(\psi(\lam|_S)) = \Ind^G_J(\psi)\lam = \chi\lam,$$
and so $\cl(\chi) = \cl(\chi')$ by Definitions \ref{def:gl}, \ref{def:gu}. 

\smallskip
(ii) Let $\tau$ be defined as in \eqref{weil-def}, so that $\tau = \tau_n$ if $\eps = +$ and $\tau = \zeta_n$ if $\eps = -$, and 
let $\cl(\varphi) = j$. Then by Definitions \ref{def:gl}, \ref{def:slu}, and \ref{def:gu}, $\varphi$ is 
a constituent of $\chi|_S = (\chi\lam)|_S$ for some $\chi \in \Irr(G)$ and $\lam \in \Irr(G/S)$ 
where $\chi\lam$ is a constituent of $\tau^j$. Conversely, suppose that 
$\varphi$ is a constituent of $(\tau^k)|_S$ for some $0 \leq k \leq j$. Then $\varphi$ is a constituent of $\chi'|_S$ for some 
$\chi' \in \Irr(G)$ which is a constituent of $\tau^k$. It follows by Definition \ref{def:slu} and (i) that 
$$j=\cl(\varphi)=\cl(\chi') \leq k,$$
and so $k=j$.    
\end{proof}

\begin{lemma}\label{ext}
In the notation of Definition \ref{def:slu}, suppose that $\varphi$ does not extend to $G=\GL^\eps_n(q)$. Then 
$\chi(1) > q^{n^2/4-2}$. In particular, $\varphi(1) > q^{n^2/4-2}/(q-\eps) \geq (2/3)q^{n^2/4-3}$.
\end{lemma}

\begin{proof}
By hypothesis, $\chi \in \Irr(G)$ is {\it reducible} over $S$. Hence the bound on $\chi(1)$ follows in the case 
$\eps = -$ from \cite[Theorem 3.9]{LBST2}. In the case $\eps = +$, this bound follows from \cite[Proposition 5.10]{KT2} and 
Proposition \ref{gl-big}(i). Now the bound on $\varphi(1)$ also follows, since $\chi(1)/\varphi(1) \leq [G:S] = q-\eps$.
\end{proof}

Now we can prove the bounds for character degrees given in Theorem \ref{slu-degree}, which we restate below.

\begin{theor}
Let $n \geq 2$, $\eps = \pm$, and $S = \SL^\eps_n(q)$. Set $\sigma_+ = 1/(q-1)$ and $\sigma_- = 1/2(q+1)$. Let $\varphi \in \Irr(S)$ have 
level $j = \cl(\varphi)$. Then the following statements hold.
\begin{enumerate}[\rm(i)]
\item $\sigma_\eps q^{j(n-j)} \leq \varphi(1) \leq q^{nj}$.
\item If $j \geq n/2$, then $\varphi(1) > q^{n^2/4-2}/(q-\eps) \geq (2/3)q^{n^2/4-3}$.
\item If $n \geq 7$ and  $\lceil (1/n)\log_q \varphi(1) \rceil < \sqrt{n-1}-1$ then 
$$\cl(\varphi) = \left\lceil \frac{\log_q \varphi(1)}{n} \right\rceil.$$
\end{enumerate}
\end{theor}

\begin{proof}
Keep the notation of Definition \ref{def:slu}. Since 
$$\chi(1)/(q-\eps) \leq \varphi(1) \leq \chi(1),$$
(i) follows from Theorem \ref{main-degree}(i). Next, (ii) follows from Theorem \ref{main-degree}(ii) if 
$\varphi(1)=\chi(1)$, and from Lemma \ref{ext} otherwise. Certainly, (iii) also follows from Theorem 
\ref{main-degree}(iii) if $\varphi(1) = \chi(1)$. Assume that $\varphi(1) < \chi(1)$. Then by Lemma 
\ref{ext} we have
$$q^{n^2/4-3.6} < \frac{q^{n^2/4-2}}{q+1} \leq \frac{\chi(1)}{q+1} \leq \varphi(1),$$
and so 
$$\lceil (1/n)\log_q \varphi(1) \rceil \geq \lceil n/4-3.6/n \rceil > \sqrt{n-1}-1$$
(as $n \geq 7$), a contradiction.
\end{proof}

\begin{corol}\label{slu-bound}
Theorems \ref{main-bound1} and \ref{main-bound2} hold for $\SL^\eps_n(q)$.
\end{corol}

\begin{proof}
Let $S := \SL^\eps_n(q)$. Certainly, the conclusions of Theorem \ref{main-bound1} and \ref{main-bound2} 
hold for any character $\varphi \in \Irr(S)$ that extend to $G:=\GL^\eps_n(q)$. 
If $\varphi$ is not extendible to $G$, then $\varphi(1) > q^{n^2/4-3.6}$ by 
Lemma \ref{ext}. Using the trivial bound $|\varphi(g)| \leq |\CB_S(g)|^{1/2}$, we see that 
the conclusion of Theorem \ref{main-bound1}, respectively Theorem \ref{main-bound2}, holds for 
$\varphi$ if we take $n \geq 2^{m+1}C+4$, respectively $n \geq 5$. Note that for $1 \leq n \leq 4$,
there is no element in $S$ such that $|\CB_G(g)| \leq q^{n^2/12}$.
\end{proof}

In the following application, for a finite group $S$ and a fixed element $g \in S$ we consider the conjugacy class $C = g^S$ and random walks on
the (oriented) Cayley graph $\Gamma(S,C)$ (whose vertices are $x \in S$ and edges are $(x,xh)$ with $x \in S$ and 
$h \in C$). Let $P^t(x)$ denote the probability that a random product of $t$ conjugates of $g$ is equal to $x \in S$, 
and let $U(x) := 1/|S|$ denote the uniform probability distribution on $S$. Also, let
$$||P^t-U||_\infty := |S| \cdot \max_{x \in S}|P^t(x)-U(x)|.$$

\begin{corol}\label{slu-mix}
Let $S= \SL^\eps_n(q)$ with $\eps = \pm$.  Let $g \in S$ be such that $|\CB_{\GL^\eps_n(q)}(g)| \leq q^{n^2/12}$,  and let 
$C = g^S$. 
\begin{enumerate}[\rm(i)]
\item Suppose $n \geq 19$. If $t \geq 19$, then $P^t$ converges to $U$ in the $||\cdot||_{\infty}$-norm when $q\to \infty$; in particular, the 
Cayley graph $\Gamma(S,C)$ has diameter at most $19$.
\item If $n \geq 10$, then the mixing time $T(S,C)$ of the random walk on $\Gamma(S,C)$ is  at most $10$ for $q$ sufficiently large.
\end{enumerate}
\end{corol}

\begin{proof}
We follow the proof of \cite[Theorem 1.11]{BLST}.
Consider the {\it Witten $\zeta$-function}
\begin{equation}\label{witten}
  \zeta^S(s) = \sum_{\chi \in \Irr(S)}\frac{1}{\chi(1)^s}.
\end{equation}  
By \cite[Theorem 1.1]{LS}, $\lim_{q \to \infty}\zeta^S(s) = 1$ as long as $s > 2/n$.

For (i), we have by a well-known result (see \cite[Chapter 1, 10.1]{AH}) and Corollary \ref{slu-bound} that
$$||P^t-U||_\infty \leq \sum_{1_S \neq \chi \in \Irr(S)} \left( \frac{|\chi(g)|}{\chi(1)}\right)^t\chi(1)^2 \leq \zeta^S(t/9-2)-1.$$
Now, as $n \geq 19$, if $t \geq 19$ then $t/9-2 > 2/n$ and so the statement follows.

For (ii), note that $P^t(x)$ is the probability that a random walk on the Cayley graph $\Gamma(S,C)$ reaches $x$ after $t$ steps. Let
$$||P^t-U||_1 := \sum_{x \in S}|P^t(x)-U(x)|.$$
By the Diaconis-Shahshahani bound \cite{DS} and Theorem \ref{main-bound2},
$$(||P^t-U||_1)^2 \leq \sum_{1_S \neq \chi \in \Irr(S)} \left( \frac{|\chi(g)|}{\chi(1)}\right)^{2t}\chi(1)^2 \leq \zeta^S(2t/9-2)-1.$$
As $n \geq 10$, if $t \geq 10$ then $2t/9-2 > 2/n$, and the statement follows.
\end{proof}


\smallskip
We conclude this section with another application. The Ore conjecture, now a theorem \cite{LBST1}, states that 
if $G$ is a finite non-abelian simple group then the commutator map 
$$G \times G \to G,~~(x,y) \mapsto xyx^{-1}y^{-1}$$
is surjective, or, equivalently, 
\begin{equation}\label{measure}
  \mu_G(g):= \sum_{\chi \in \Irr(G)}\frac{\chi(g)}{\chi(1)} > 0
\end{equation}  
for all $g \in G$. A strong qualitative refinement of the Ore conjecture was conjectured by Shalev \cite[Conjecture 1.11]{Sh}, and states that 
if $G$ is a finite simple group of Lie type of bounded rank and $|G| \to \infty$, then the commutator map 
yields an almost uniform distribution on $G$; more precisely,
$$\max_{1 \neq g \in G}|\mu_G(g)-1| \to 0.$$
However, this was disproved by Liebeck and Shalev (unpublished), by considering transvections in $\SL^\eps_3(q)$.
A more recent conjecture of Avni and Shalev \cite[Corollary 1.7]{ST} states that, if 
$G$ is a simple group of Lie type of rank $r$, then 
\begin{equation}\label{av-sh}
  \max_{1 \neq g \in G}\mu_G(g) \leq C(r)
\end{equation}   
for some constant $C(r)$ possibly depending on $r$.  
We can now offer some evidence in support of this conjecture:

\begin{corol}\label{slu-dist}
For any $k \in \ZZ_{\geq 1}$, let $q_k$ be a prime power, $n_k \geq 19$, $\eps_k = \pm$, 
$G_k := \SL^{\eps_k}_{n_k}(q_k)$, and let $g_k \in G_k$ be such that $|\CB_{\GL^{\eps_k}_{n_k}(q_k)}(g_k)| \leq (q_k)^{n_k^2/12}$. 
If $\mu_{G_k}$ is defined as in \eqref{measure} and $\lim_{k \to \infty}|G_k| = \infty$, then
$$\lim_{k \to \infty}\mu_{G_k}(g_k) = 1.$$
\end{corol}

\begin{proof}
By Theorem \ref{main-bound2}, $|\chi(g_k)| \leq \chi(1)^{8/9}$ for all $\chi \in \Irr(G_k)$, whence 
$$|\mu_{G_k}(g_k)-1| \leq \zeta^{G_k}(1/9)-1,$$
where the zeta-function $\zeta^S(s)$ is as defined in \eqref{witten}. Now the result follows by \cite[Theorem 1.1]{LS} (if $n_k$ is bounded)
and \cite[Theorem 1.2]{LS} (if $n_k$ grows unbounded).
\end{proof}

Very recently, it has been shown in \cite{ST}  that the upper bound $C(r)$ in the Avni-Shalev conjecture \eqref{av-sh} must in fact depend on $r$,  
and similarly, 
$$\lim_{n \to \infty}\max_{1 \neq g \in \AAA_n} \mu_{\AAA_n}(g) = \infty.$$

\section{Dual pairs of general linear and unitary groups}
\subsection{The case of $\GL_n(q)$}
In this subsection we prove Theorem \ref{gl-dual}. 

\smallskip
Keep all of the notation of Theorem \ref{gl-dual}.
Any $v \in V = A \otimes_{\FQ} B$ can be written as $\sum^t_{i=1}a_i \otimes b_i$ for some $a_i \in A$, $b_i \in B$. Choosing 
such an expression with smallest possible $t$ for $v$, one then calls $t$ the {\it rank} of $v$; note that this rank cannot exceed $j = \dim B$.
Now let $\Omega$ denote the set of all $v \in V$ of (largest possible) rank $j$. Then $G$ acts transitively on $\Om$. 

Fix a basis $(e_1, \ldots ,e_n)$ of $A$ and a basis $(f_1, \ldots, f_j)$ of $B$, and consider $v_0 := \sum^j_{i=1}e_i \otimes f_i$. It is 
straightforward to check that $(g,s) \in \Gamma:=G \times S$ fixes $v_0$ exactly when
$$g = \begin{pmatrix} \tw t X^{-1} & *\\0 & Y \end{pmatrix},~~s = X,~~X \in \GL_j(q),~Y \in \GL_{n-j}(q)$$
(in the chosen bases). Denoting $R := \Stab_\Gamma(v_0)$, we see that the permutation character $\rho$ of $\Gamma$ acting on $\Om$
is $\Ind^\Gamma_R(1_R)$. Consider the parabolic subgroup
$$P = U \rtimes L := \Stab_G(\langle e_1, \ldots,e_j \rangle_{\FQ})$$
with its radical $U$ and Levi subgroup $L = \{ \diag(X,Y) \} = \GL_j \times \GL_{n-j}$, where 
$$\GL_j := \{ \diag(X,I_{n-j}) \in L\},~~  \GL_{n-j} := \{ \diag(I_j,Y) \in L \}.$$
We can write
$R = \tilde U \rtimes \tilde L$, where
$$\tilde U := \left\{ (u,I_j) \in \Gamma \mid u \in U\right\},~~
    \tilde L = \left\{ \left( \diag( \tw t X^{-1},Y),X\right) \in \Gamma \mid 
            X \in \GL_j(q),~Y \in \GL_{n-j}(q) \right\}.$$
Now consider any $\al \in \Irr(S)$ and $\delta \in \Irr(G)$, and express
\begin{equation}\label{hc-1}
  \SR^G_L(\delta) = \gam \otimes 1_{\GL_{n-j}} + \delta',
\end{equation}  
where $\gam$ is either $0$ or a character of $\GL_j$, $\delta'$ is either $0$ or a character of $L$ with no 
irreducible constituent having $\GL_{n-j}$ in its kernel, and $\SR^G_L$ denotes the Harish-Chandra restriction
(which is adjoint to the Harish-Chandra induction $R^G_L$). 
Also let $\sigma$ denote the transpose-inverse automorphism of 
$S$: $\sigma(X) = \tw t X^{-1}$. Note that $X \in S$ and $\tw t X$ are $S$-conjugate. Hence, 
$$\gam^\sigma(X) = \gam(\sigma(X)) = \gam(\tw t X^{-1}) = \gam(X^{-1}) = \bar\gam,$$
i.e. $\gam^\sigma = \bar\gam$. Hence, the value of $(\gam \otimes 1_{\GL_{n-j}}) \otimes \al$ at 
a typical element $\left(\diag(\tw t X^{-1},Y),X\right) \in \tilde L$ is $\gam^\sigma(X)\al(X) = \bar\gam\al(X)$.

Certainly, the kernel of $1_R$ contains $\tilde U \lhd R$ and also $\GL_{n-j}$. It then follows that
\begin{equation}\label{hc-2}
  [(\delta \otimes \al)|_R,1_R]_R = [((\gam \otimes 1_{\GL_{n-j}}) \otimes \al)|_{\tilde L},1_{\tilde L}]_{\tilde L} = 
    [\bar\gam\al,1_{\GL_j}]_{\GL_j} = [\gam,\al]_{\GL_j}.
\end{equation}    
Using \eqref{hc-1} and the adjoint functor $R^G_L$, we also have
$$[\gam,\al]_{\GL_j} = [\SR^G_L(\delta),\al \otimes 1_{\GL_{n-j}}]_L = 
    [\delta,R^G_L(\al \otimes 1_{\GL_{n-j}})]_G.$$ 
Together with \eqref{hc-2}, this shows
\begin{equation}\label{hc-3}
  [\delta \otimes \al,\rho]_\Gamma =   [(\delta \otimes \al)|_R,1_R]_R = [\delta,R^G_L(\al \otimes 1_{\GL_{n-j}})]_G.
\end{equation}  
Next we will use the proof of Proposition \ref{gl-const} applied to $\chi:=\al$, written in the form
\eqref{gl-1} with $s_m =1 \neq s_1, \ldots,s_{m-1}$.  Also write
$$\lam_m = (\gam_2, \ldots,\gam_r) \vdash k,~~\tilde \lam_m := (n-j,\gam_2, \ldots ,\gam_r).$$ 
In combination with Theorem \ref{main1-gl}, the proof of Proposition \ref{gl-const} shows that all irreducible constituents
of $R^G_L(\al \otimes 1_{\GL_{n-j}})$ are of true level $\leq j$; moreover, if $\theta$ is such a constituent of
true level $j$, then 
\begin{equation}\label{hc-4}
  \theta = \DC_\al := S(s_1,\lam_1) \circ S(s_2,\lam_2) \circ \ldots \circ S(s_{m-1},\lam_{m-1}) \circ S(1,\tilde \lam_m)
\end{equation}
and $\theta$ occurs with multiplicity one. Note that $\tilde\lam_m$ is a partition precisely when $n-j \geq \gam_2$, which is 
equivalent to $\cl^*(\al) = j-\gam_2 \geq 2j-n$. Writing the $\alpha$-isotypic component
of $\rho$ as $E_\al \otimes \al$ and setting $\DC_\al := 0$ when $\cl^*(\al) < 2j-n$, we conclude from \eqref{hc-3} that 
$E_\al - \DC_\al$ is either $0$ or a character of $G$ all irreducible
constituents of which have true level smaller than $j$.

Recall that $G$ acts transitively on $\Omega$, with point stabilizer 
$R \cap G = U \rtimes \GL_{n-j} = \Stab_G(e_1, \ldots ,e_j)$. Hence, the proof of Proposition \ref{gl-induced} 
implies that 
$$\rho|_G = R^G_L(\reg_{\GL_j} \times 1_{\GL_{n-j}}).$$
Now, by Proposition \ref{gl-const} and Theorem \ref{main1-gl}, every irreducible constituent of $(\tau-\rho)|_G = (\tau_n)^j - \rho|_G$ has true
level smaller than $j$. The same is true for all irreducible constituents of $F_\al$, where 
$F_\al \otimes \al$ is the $\alpha$-isotypic component of $\tau|_\Gamma-\rho$. Since 
every irreducible character of $G$ of true level $j$ appears in $(\tau_n)^j$ by Definition \ref{def:gl}, 
such a character must be some $\DC_\al$ for some 
$\al \in \Irr(S)$ with $\cl^*(\al) \geq 2j-n$.
Thus we have completed the proof of statements (i)--(iii) of Theorem \ref{gl-dual}. 
In fact, we have also obtained an explicit formula for the bijection $\al \mapsto \DC_\al$ in (iii):

\begin{corol}\label{gl-dual-bij}
Let $\cl^*(\al) \geq 2j-n$ for $\al \in \Irr(\GL_j(q))$ and express 
$$\al = S(s_1,\lam_1) \circ S(s_2,\lam_2) \circ \ldots \circ S(s_{m-1},\lam_{m-1}) \circ S(1,\lam_m)$$
as in \eqref{gl-1}, with $s_i \neq 1$ and $\lam_m = (\gam_2, \ldots,\gam_r)$. Then in Theorem \ref{gl-dual}(iii) we have
$$\DC_\al = S(s_1,\lam_1) \circ S(s_2,\lam_2) \circ \ldots \circ S(s_{m-1},\lam_{m-1}) \circ S(1,\tilde \lam_m),$$
with $\tilde \lam_m = (n-j,\gam_2, \ldots,\gam_r)$.
\end{corol}

\begin{proof}
This is just \eqref{hc-4}.
\end{proof}

To prove the last two statements of Theorem \ref{gl-dual}, we need some auxiliary statements. For any finite-dimensional vector 
space $U$ over a  field $\FF$ and any element $x \in \GL(U)$, let $d_U(x) := \dim_\FF\Ker(x-1_U)$, and let $\delta_U(x)$ denote 
the largest dimension of $x$-eigenspaces on $U \otimes_{\FF} \overline\FF$.

\begin{lemma}\label{fixed}
Let $\FF$ be a field and let $V = A \otimes_{\FF}B$ with $\dim_{\FF}A = n$ and $\dim_{\FF}B = j$, 
and let $g \in \GL(A)$ with $\delta_A(g) = k$. Then the following statements hold. 
\begin{enumerate}[\rm(i)]   
\item $d_V(g \otimes s) \le kj$ for all $s \in \GL(B)$. 
\item If $k \geq n/2$, then $d_V(g \otimes s) \le k(j-2) + n$ for all but possibly one element $s \in \GL(B)$.
\item Assume $g \notin \ZB(\GL(A))$ and $j \geq 2$. Then 
$d_V(g \otimes s) \leq (n-1)(j-1)+1$ for all but possibly one element $s \in \GL(B)$, and $d_V(g \otimes s) \leq (n-1)j$ for all $s \in \GL(B)$.
\end{enumerate}
\end{lemma}

\begin{proof}  
With no loss we may replace $A$, $B$, $V$ by $A \otimes_\FF \overline\FF$, $B \otimes_\FF \overline\FF$, $V \otimes_\FF \overline\FF$, and 
thus assume that $\FF = \overline\FF$.
Note that $d_V(h) \leq d_U(h) + d_{V/U}(h)$ for any $h \in \GL(V)$ and any $h$-invariant subspace $U \subseteq V$. 
Since $\FF = \overline\FF$, there exists a 
$g \otimes s$-invariant filtration of $V$ with all quotients isomorphic to $A$.  Since the result is obvious for $j=1$, (i) follows.

\smallskip
For  (ii), we may assume that the $1$-eigenspace of $g$ on $A$ has dimension $k$. First suppose that $g$ is not unipotent.
Write $A=A_1 \oplus A_2$ with $A_1 = \Ker((g-1)^n)$ and $A_2 = \Im((g-1)^n)$.  Similarly, write $B=B_1 \oplus B_2$
with $B_1 = \Ker((s-1)^j)$ and $B_2 = \Im((s-1)^j)$. Applying (i) to $g \otimes s$ acting on $A_1 \otimes B_1$, we get  
$d_V(g \otimes s)  \leq k(\dim B_1) + (\dim A_2)(\dim B_2)$, where the right-hand side is clearly maximized when $\dim A_2= n-k$, giving
$$d_V(g \otimes s)  \leq k(\dim B_1) + (n-k)(\dim B_2).$$  
Since $k \geq n-k$, the right-hand side in the latter bound does not decrease as $\dim B_1$ grows.  
So if $B_2 \ne 0$, this gives $d_V(g \otimes s)  \leq k(j-2) + n$. 

We will now prove the same inequality for $s \ne 1$ and $B_2 = 0$, i.e. when $s$ is unipotent. Note that, since $g \otimes s$ has 
no fixed point on $A_2 \otimes B$, the result follows by induction on $n$ unless $A_2=0$, i.e. $g$ is unipotent. 
Let $J_i$ denote the Jordan $i \times i$-block with eigenvalue $1$; also use the symbol $mJ_i$ to denote 
the direct sum of $m$ blocks $J_i$. By the main result of \cite{S1}, $J_a \otimes J_b$ is a direct sum of $\min(a,b)$ Jordan blocks, whence
$d_V(J_a \otimes J_b) = \min(a,b)$. It follows for $h= J_a \oplus rJ_1$ and $t = J_b$ with $r \geq a-2$, $b \geq 2$ that 
\begin{equation}\label{fixed-1}
  d_V(h \otimes t) = \min(a,b) + r \leq r+a + (r+1)(b-2) = d_V(h^\sharp \otimes t^\flat) \leq (r+1)b = d_V(h^\sharp \otimes t^0),
\end{equation}   
where we define
$$h^\sharp := (a-1)J_2 \oplus (r+2-a)J_1,~~t^\flat := J_2 \oplus (b-2)J_1,~~t^0 := bJ_1$$
for the given $h,t$. Write
$$g = J_{a_1} \oplus J_{a_2} \oplus \ldots \oplus J_{a_m} \oplus vJ_1,~~s =  J_{b_1} \oplus J_{b_2} \oplus \ldots \oplus J_{b_l}$$
with $a _1 \geq a_2 \geq \ldots \geq a_m \geq 2$, $b_1 \geq b_2 \geq \ldots \geq b_l \geq 1$. Then the conditions $k \geq n/2$ and 
$s \neq 1$ imply that $v \geq \sum^m_{i=1}(a_i-2)$ and $b_1 \geq 2$. Thus we can write $g= h_1 \oplus \ldots \oplus h_m$ 
with $h_i = J_{a_i} \oplus r_iJ_1$ and $r_i \geq a_i-2$. Applying \eqref{fixed-1}, we see that $d_V(g \otimes s)$ does not decrease 
when we replace $g,s$ by 
$$g^\sharp:=  h_1^\sharp \oplus h_2^\sharp \oplus \ldots \oplus h_m^\sharp,~~
    s^\flat:=  J_{b_1}^\flat \oplus J_{b_2}^0 \oplus \ldots \oplus J_{b_l}^0 = J_2 \oplus (j-2)J_1,$$
which does not change $d_A(g) = k$. Thus 
$$d_V(g \otimes s) \leq d_V(g^\sharp \otimes s^\flat) = \sum^m_{i=1}(r_i+a_i) + \sum^m_{i=1}(r_i+1)(\sum^l_{i=1}b_i-2) = n+k(j-2).$$

\smallskip
For (iii), note that $g \notin \ZB(\GL(A))$ implies that $k \leq n-1$, whence $d_V(g \otimes s) \leq (n-1)j$ by (i). Furthermore, if 
$k < n/2$, then by (i) we have 
$$d_V(g \otimes s) < nj/2 < (n-1)(j-1)+1$$
as $j \geq 2$. If $k \geq n/2$, then the statement follows from (ii).
\end{proof}   

Let $2 \leq j \leq n/2$, and let $\chi \in \Irr(G)$ have $\cl(\chi) = j$. Multiplying $\chi$ by a suitable linear character, we may assume that
$\cl^*(\chi) = j$. By Theorem \ref{gl-dual}(iii), $\chi = \DC_\al$ for some $\al \in \Irr(S)$. Consider any $g \in G \smallsetminus \ZB(G)$.
We will now bound $\DC_\al(1)$ and $|\DC_\al(g)|$ using the formula
\begin{equation}\label{gl-dual1}
  D_\al(g) = \frac{1}{|S|}\sum_{s \in S}\tau(gs)\bar\al(s),
\end{equation}  
see \cite[Lemma 5.5]{LBST1}. According to Definition \ref{def:gl} and Theorem \ref{gl-dual}(ii), we can write 
\begin{equation}\label{gl-dual2}
  \tau|_G = \sum^N_{i=1}a_i\theta_i,~~D'_\al := D_\al-\DC_\al = \sum^{N'}_{i=1}b_i\theta_i,
\end{equation}  
where $\theta_i \in \Irr(G)$ are pairwise distinct, $a_i,b_i \in \ZZ_{\geq 0}$,
$N \geq N'$, $a_i \geq b_i$ if $i \leq N'$, $\cl^*(\theta_i) \leq j$ for all $i$. In fact, if 
$i \leq N'$, then $\cl^*(\theta_i) \leq j-1$, and so $\cl(\theta_i) = \cl^*(\theta_i) \leq j-1 < n/2$ as $j \leq n/2$, whence 
$\theta_i(1) \leq q^{n(j-1)}$ by Theorem \ref{main-degree}.  
Let $k(X) = |\Irr(X)|$ denote the class number of a finite group $X$. By \cite[Proposition 3.5]{FG1}, $k(\GL_n(q)) \leq q^n$.  Note that 
$N'$ cannot exceed the total number of irreducible characters of true level $<j$, hence by Theorem \ref{gl-dual}(iii) we have 
$$N' \leq \sum^{j-1}_{i=0}k(\GL_i(q)) \leq \sum^{j-1}_{i=0}q^i < q^j.$$
Also, $\sum^N_{i=1}a_i^2 = [\tau|_G,\tau|_G]_G \leq 8q^{j^2}$ by Lemma \ref{orbits}. It follows that
$$(\sum^{N'}_{i=1}b_i)^2 \leq N'\sum^{N'}_{i=1}b_i^2 \leq q^j\sum^N_{i=1}a_i^2 \leq 8q^{j^2+j},$$
and so
\begin{equation}\label{gl-dual3}
  D'_\al(1) \leq \sum^{N'}_{i=1}b_iq^{n(j-1)} \leq \sqrt{8q^{j^2+j}}q^{n(j-1)}.
\end{equation}
Next, if $1 \neq s \in S$, then $|\tau(s)| = q^{nd_B(s)} \leq q^{n(j-1)}$, whence \eqref{gl-dual1} implies that
$$D_\al(1) \geq \al(1)(q^{nj}-|S|q^{n(j-1)})/|S|.$$ 
We may assume that $g \notin \ZB(G)$ and so $k := \delta_A(g) \leq n-1$. 
Then by Lemma \ref{fixed}(i), (iii), $|\tau(gs)| = q^{d_V(gs)} \leq q^{(n-1)(j-1)+1}$ for all but possibly one 
element $s \in S$, for which we have $|\tau(gs)| \leq q^{kj} \leq q^{(n-1)j}$.
Hence by \eqref{gl-dual1} we have 
$$|D_\al(g)| \leq \al(1)(q^{kj}+|S|q^{(n-1)(j-1)+1})/|S|.$$
Using \eqref{gl-dual3} and the estimate $|\chi(g)| \leq |D_\al(g)|+D'_\al(1)$, we obtain
\begin{equation}\label{gl-dual4}
  \begin{array}{rll}\chi(1) & \geq & \al(1)(q^{nj}-|S|q^{n(j-1)}-|S| \sqrt{8q^{j^2+j}}q^{n(j-1)})/|S|,\\ \vspace{-3mm} \\
    |\chi(g)| & \leq & \al(1)(q^{kj}+|S|q^{(n-1)(j-1)+1}+|S| \sqrt{8q^{j^2+j}}q^{n(j-1)})/|S|,\end{array} 
\end{equation}     
Now assume that $2 \leq j \leq \sqrt{(8n-17)/12}-1/2$. Then $(n-1)(j-1)+1 \leq n(j-1)$, $|S| < q^{j^2}$, 
$3(j^2+j+1)/2 \leq n-1$, and so 
$$q^{j^2+n(j-1)}\left(1+\sqrt{8q^{j^2+j}}\right) <  1.046q^{3(j^2+j+1)/2+n(j-1)-j} \leq 1.046q^{(n-1)j-1} \leq 0.523q^{(n-1)j}.$$ 
It now follows from \eqref{gl-dual4} that 
$$\chi(1) \geq  \frac{q^{nj}(1-0.523q^{-j})}{|S|/\al(1)} \geq  \frac{0.869q^{nj}}{|S|/\al(1)},~~|\chi(g)| \leq  \frac{1.523q^{(n-1)j}}{|S|/\al(1)}$$
and so $|\chi(g)| < 1.76\chi(1)^{1-1/n}$.

Next assume that $2 \leq j \leq (\sqrt{12n-59}-1)/6$. Then $3j^2+j+3 \leq n-2$, and so 
$$q^{j^2+n(j-1)}\left(1+\sqrt{8q^{j^2+j}}\right) <  1.046q^{(3j^2+j+3)/2+n(j-1)} \leq 1.046q^{n(j-1/2)-1} \leq 0.523q^{n(j-1/2)}.$$ 
It now follows from \eqref{gl-dual4} that 
$$\chi(1) \geq  \frac{q^{nj}(1-0.523q^{-j})}{|S|/\al(1)} \geq  \frac{0.869q^{nj}}{|S|/\al(1)},~~
   |\chi(g)| \leq  \frac{q^{kj}+0.523q^{n(j-1/2)}}{|S|/\al(1)}$$
and so 
\begin{equation}\label{gl-dual5}
  |\chi(g)| < 1.76\chi(1)^{\max(1-\frac{1}{2j},\frac{k}{n})},
\end{equation}  
as stated in Theorem \ref{gl-dual}(iv).

\smallskip
Since the case $j=0$ is obvious, it remains to consider the case $j=1$, whence $\chi$ is a Weil character (see Example \ref{gl-ex}). 
Suppose first that $j =1 \leq \sqrt{(8n-17)/12}-1/2$, and so $n \geq 6$. It is easy to check that
$$\chi(1) \geq (q^n-q)/(q-1),~~|\chi(g)| \leq (q^{n-1}+q)/(q-1)$$
and so again $|\chi(g)| < 1.76\chi(1)^{1-1/n}$; in particular, \eqref{gl-dual5} holds if $k = n-1$. We now consider the case
$k \leq n-2$ and $j =1 \leq (\sqrt{12n-59}-1)/6$, i.e. $n \geq 9$. If $q=2$, then
$$\chi(1) = 2^n-2,~~|\chi(g)| \leq 2^k-2 \leq \chi(1)^{k/n}.$$
If $k \leq (n-1)/2$, then
$$\chi(1) > q^{n-1},~~|\chi(g)| \leq q^k \leq \chi(1)^{1/2}.$$
If $k \geq (n+1)/2$, then
$$\chi(1) \geq (q^n-q)/(q-1),~~|\chi(g)| < (q^k+q^{n-k}+q)/(q-1) < 1.76\chi(1)^{k/n}.$$
If $k = n/2$ and $q \geq 3$, then
$$\chi(1) \geq (q^n-q)/(q-1),~~|\chi(g)| < (2q^k+q)/(q-1) < 1.76\chi(1)^{1/2},$$
completing the proof of \eqref{gl-dual5} for $j=1$.

\smallskip
To prove Theorem \ref{gl-dual}(v), note that if $\psi \in \Irr(\SL_n(q))$ of level $j$ does not extend to $G$, then $j \geq 1$ 
(in particular, $n \geq 6$ as above) and 
any character $\chi \in \Irr(G)$ lying above it has degree $\chi(1) > q^{n^2/4-2}$ by Lemma \ref{ext}. On the other hand,
$\chi(1) \leq q^{nj}$ by Theorem \ref{main-degree}, a contradiction. Hence $\psi$ extends to $G$, and the statement
follows from Theorem \ref{gl-dual}(iv). 

We have completed the proof of Theorem \ref{gl-dual}.
\hfill $\Box$

\subsection{The case of $\GU_n(q)$}
In this subsection we will prove Theorem \ref{gu-dual}; keep all of its notations, as well as 
the notation $G_i = \GU_i(q)$ for $1 \leq i \leq n$. Note that statement (i) of Theorem \ref{gu-dual} is just part of 
Theorem \ref{main1-gu}(i). 

\smallskip
Suppose now that $2 \leq j \leq n/2$ and consider any $\al \in \Irr(G_j)$. Let $D'_\al$ denote the sum of all irreducible 
constituents of true level $< j$ of $D_\al$, counting with their multiplicities, so that $\DC_\al := D_\al-D'_\al$ is either
$0$ or a character, all of whose irreducible constituents have true level $j$. We will again use \eqref{gl-dual1} 
and express $\tau|_G$ and $D'_\al$ as in \eqref{gl-dual2}. In particular, 
$\theta_i \in \Irr(G)$ are pairwise distinct, $a_i,b_i \in \ZZ_{\geq 0}$,
$N \geq N'$, $a_i \geq b_i$ if $i \leq N'$, $\cl^*(\theta_i) \leq j$ for all $i$. In fact, if 
$i \leq N'$, then $\cl^*(\theta_i) \leq j-2$ by Corollary \ref{parity}, and so $\cl(\theta_i) = \cl^*(\theta_i) \leq j-2 < n/2$ as $j \leq n/2$, whence 
$\theta_i(1) \leq q^{n(j-2)}$ by Theorem \ref{main-degree}.  
According to \cite[\S3.3]{FG1}, $k(\GU_n(q)) \leq 8.26q^n$.  Note that 
$N'$ cannot exceed the total number of irreducible characters of true level $\leq j-2$, hence by Theorem \ref{gu-dual}(i) we have 
$$N' \leq \sum^{j-2}_{i=0}k(\GU_i(q)) \leq 8.26\sum^{j-2}_{i=0}q^i < 8.26q^{j-1}.$$
Also, $\sum^N_{i=1}a_i^2 = [\tau|_G,\tau|_G]_G = [(\zeta_n)^{2j},1_G]_G\leq 2q^{j^2}$ by Lemma \ref{orbits}. It follows that
$$(\sum^{N'}_{i=1}b_i)^2 \leq N'\sum^{N'}_{i=1}b_i^2 \leq 8.26q^{j-1}\sum^N_{i=1}a_i^2 \leq 16.52q^{j^2+j-1},$$
and so
\begin{equation}\label{gu-dual1}
  D'_\al(1) \leq \sum^{N'}_{i=1}b_iq^{n(j-2)} \leq \sqrt{16.52q^{j^2+j-1}}q^{n(j-2)}.
\end{equation}
As before, if $1 \neq s \in S$, then $|\tau(s)| = q^{nd_B(s)} \leq q^{n(j-1)}$, whence \eqref{gl-dual1} implies that
$$D_\al(1) \geq \al(1)(q^{nj}-|S|q^{n(j-1)})/|S|.$$ 
We may assume that $g \notin \ZB(G)$ and so $k := \delta_A(g) \leq n-1$. 
Then by Lemma \ref{fixed}(i), (iii), $|\tau(gs)| = q^{d_V(gs)} \leq q^{(n-1)(j-1)+1}$ for all but possibly one 
element $s \in S$, for which we have $|\tau(gs)| \leq q^{kj} \leq q^{(n-1)j}$. Hence by \eqref{gl-dual1} we have 
$$|D_\al(g)| \leq \al(1)(q^{kj}+|S|q^{(n-1)(j-1)+1})/|S|.$$
Using \eqref{gu-dual1} and the estimate $|\DC_\al(g)| \leq |D_\al(g)|+D'_\al(1)$, we obtain
\begin{equation}\label{gu-dual2}
  \begin{array}{rll}\DC_\al(1) & \geq & \al(1)(q^{nj}-|S|q^{n(j-1)}-|S| \sqrt{16.52q^{j^2+j-1}}q^{n(j-2)})/|S|,\\ \vspace{-3mm} \\
    |\DC_\al(g)| & \leq & \al(1)(q^{kj}+|S|q^{(n-1)(j-1)+1}+|S| \sqrt{16.52q^{j^2+j-1}}q^{n(j-2)})/|S|.\end{array} 
\end{equation}     
Note that $|S| \leq 1.5q^{j^2}$ by Lemma \ref{trivial}(iii). 

\smallskip
Now assume that $2 \leq j \leq \sqrt{n-3/4}-1/2$; in particular $n \geq 7$. Then $j^2+j \leq n-1$, $(n-1)(j-1)+1 \leq n(j-1)$, 
and so 
$$1.5q^{j^2+n(j-1)}\left(1+\sqrt{16.52q^{j^2+j-1-2n}}\right) <  1.77q^{j^2+n(j-1)} \leq 1.77q^{(n-1)j-1} \leq 0.885q^{(n-1)j}.$$ 
It now follows from \eqref{gu-dual2} that 
$$\DC_\al(1) \geq  \frac{q^{nj}(1-0.885q^{-j})}{|S|/\al(1)} \geq  \frac{0.778q^{nj}}{|S|/\al(1)},~~|\DC_\al(g)| \leq  \frac{1.885q^{(n-1)j}}{|S|/\al(1)}$$
and so $\DC_\al$ is a {\it true} character of $G$ and $|\DC_\al(g)| < 2.43\DC_\al(1)^{1-1/n}$. 

Next assume that $2 \leq j \leq \sqrt{n/2-1}$; in particular $n \geq 10$. Then $j^2 \leq n/2-1$,
and so 
$$1.5q^{j^2+n(j-1)}\left(1+\sqrt{16.52q^{j^2+j-1-2n}}\right) <  1.77q^{j^2+n(j-1)} \leq 1.77q^{n(j-1/2)-1} \leq 0.885q^{n(j-1/2)}.$$ 
It now follows from \eqref{gu-dual2} that 
$$\DC_\al(1) \geq  \frac{q^{nj}(1-0.885q^{-j})}{|S|/\al(1)} \geq  \frac{0.778q^{nj}}{|S|/\al(1)},~~
    |\DC_\al(g)| \leq  \frac{q^{kj}+0.885q^{n(j-1/2)}}{|S|/\al(1)}$$
and so $|\DC_\al(g)| < 2.43\DC_\al(1)^{\max(1-1/2j,k/n)}$. 

We have shown that, for any $\al \in \Irr(S)$, $\DC_\al$ is a $G$-character that involves only characters of true level $j$. As
$\tau|_G = \sum_{\al \in \Irr(S)}\al(1)(\DC_\al + D'_\al)$, it follows that the total sum $J$ of all multiplicities of irreducible constituents 
of true level $j$ in $\tau|_G$ is at least $\sum_{\al \in \Irr(S)}\al(1)$. On the other hand, $J = \sum_{\al \in \Irr(S)}\al(1)$ by
Theorem \ref{main1-gu}(i), and furthermore each character of true level $j$ enters some $\DC_\al$, and 
the total number of such characters is $k(S)$. We conclude that the characters 
$\DC_\al$ are all irreducible, pairwise distinct, and account for all characters of true level $j$ of $G$.

Suppose now $\chi \in \Irr(G)$ have $\cl(\chi) = j$ and $2 \leq j \leq \sqrt{n-3/4}-1/2$ as above. 
Multiplying $\chi$ by a suitable linear character, we may assume that $\cl^*(\chi) = j$. By what we have just shown, $\chi = \DC_\al$ for some 
$\al \in \Irr(S)$. We have therefore proved Theorem \ref{gu-dual}(ii), (iii) for $j \geq 2$. 

\smallskip
Since the case $j=0$ is obvious, it remains to prove Theorem \ref{gu-dual}(ii), (iii) for $j=1$.
In this case, statement (ii) is well known, and $D_\al = \DC_\al$;
furthermore, $\chi$ is a Weil character (see Example \ref{gu-ex}).
First suppose that $j=1 \leq \sqrt{n-3/4}-1/2$, and so $n \geq 4$.
It is easy to check that
$$\chi(1) \geq (q^n-q)/(q+1),~~|\chi(g)| \leq (q^{n-1}+q)/(q+1)$$
and so again $|\chi(g)| < 2.43\chi(1)^{1-1/n}$. Assume now that $j=1 \leq \sqrt{(n-2)/2}$, i.e. $n \geq 6$. 
If $k \leq (n-1)/2$, then
$$\chi(1) > (q^{n}-q)/(q+1),~~|\chi(g)| \leq q^k < 2.43\chi(1)^{1/2}.$$
If $k \geq n/2$, then
$$\chi(1) \geq (q^n-q)/(q+1),~~|\chi(g)| < (2q^k+q)/(q+1) < 2.43\chi(1)^{k/n}.$$
completing the proof of Theorem \ref{gu-dual}(iii) for $j=1$.

\smallskip
Theorem \ref{gu-dual}(iv) can now be proved by exactly the same argument as we had for Theorem \ref{gl-dual}(v).
We have completed the proof of Theorem \ref{gu-dual}.
\hfill $\Box$

\begin{corol}\label{slu-dual}
Let $G = \GL^\eps_n(q) \geq S = \SL^\eps_n(q)$ with $\eps = \pm$, and let $0 \leq j < n/2$. Then the following statements hold.

\begin{enumerate}[\rm(i)]
\item If $\chi \in \Irr(G)$ has $\cl(\chi) = j$, then $\varphi:=\chi|_S$ is irreducible. Furthermore, 
$$\Irr(G|\varphi) = \{\chi\lam \mid \lam \in \Irr(G/S)\},$$ 
and it contains a unique character of true level $j$.
\item If $\Theta$ is the bijection between $\{ \chi \in \Irr(G) \mid \cl^*(\chi) = j\}$ and 
$\Irr(\GL^\eps_j(q))$ defined in Theorem \ref{gl-dual}(iii), respectively Theorem \ref{gu-dual}(i), 
then $\Lambda:\al \mapsto \Theta^{-1}(\al)|_S$ is a bijection between $\Irr(\GL^\eps_j(q))$ and 
$\{ \varphi \in \Irr(S) \mid \cl(\varphi) = j\}$.
\end{enumerate}
\end{corol}

\begin{proof}
(i) Replacing $\chi$ by $\chi\lam$ for a suitable $\lam \in \Irr(G/S)$ we may assume that $\cl^*(\chi) = j$. 
By Theorems \ref{main1-gl} and \ref{main1-gu}, the first part of the partition $\gam_1$ is $n-j$. It follows that for any 
$\delta \in \mu_{q-\eps} \smallsetminus \{1\}$, 
the first part of $\gam_\delta$ is $\leq j$, whence $\cl^*(\chi\beta) \geq n-j > j$ and so $\chi\beta \neq \chi$ for any 
$1_G \neq \beta \in \Irr(G/S)$. As $G/S$ is cyclic, the statements now follow from \cite[Lemma 3.2]{KT1} and 
Gallagher's theorem \cite[(6.17)]{Is2}.

\smallskip
(ii) For $\al \in \Irr(\GL^\eps_j(q))$, let $\chi:=\Theta^{-1}(\al)$ and $\varphi := \chi|_S = \Lambda(\al)$. Then $\cl^*(\chi) = j < n/2$, and 
the same arguments as in (i) show that $\cl(\chi) = j$. Now $\varphi \in \Irr(S)$ by (i) and 
$\cl(\varphi) = \cl(\chi) = j$ by Definition \ref{def:slu}. Suppose now that $\varphi = \Lambda(\al')$ for some $\al' \in \Irr(\GL^\eps_j(q))$. Then
$(\chi')|_S = \varphi = \chi|_S$ for $\chi' := \Theta^{-1}(\al')$. By (i), $\chi' = \chi\beta$ for some $\beta \in \Irr(G/S)$. Since 
$\cl^*(\chi') = j = \cl^*(\chi)$, the arguments in (i) show that $\beta = 1_G$ and $\chi' = \chi$, whence $\al' = \al$. Thus $\Lambda$ is
injective. Finally, suppose $\varphi_1 \in \Irr(S)$ has $\cl(\varphi_1) = j$. Then $\cl(\chi_1) = j$ for some $\chi_1 \in \Irr(G|\varphi_1)$ 
again by Definition \ref{def:slu}. Replacing $\chi_1$ by $\chi_1\delta$ for a suitable $\delta \in \Irr(G/S)$ we may assume that
$\cl^*(\chi_1)=j$. Now $\varphi_1 = (\chi_1)|_S$ by (i) and so $\varphi_1 = \Lambda(\Theta(\chi_1))$, proving the surjectivity of $\Lambda$.    
\end{proof}

\subsection{Some final remarks}

First we give an example to show that the exponent $1-1/n$ in Theorems \ref{gl-dual} and \ref{gu-dual} is optimal 
(say when $q$ is bounded and $n \to \infty$).
 
\begin{examp}\label{glu-ex}
{\em Let $G = \GL(A) = \GL^\eps_n(q)$ for $\eps = \pm$ and let $\chi = \chi^{(n-1,1)}$ be the unipotent Weil character of level $1$.
Then $\chi(1) \approx q^n/(q-\eps)$.

\begin{enumerate}[\rm(i)]
\item $\chi(g) \approx q^{n-1}/(q-\eps)$ if $g$ is a transvection.

\item For any $n/2 < k \leq n-1$, we can find $g \in \SL^\eps_n(q)$ such that the $g$-fixed point subspace on $A$ 
has dimension $k$ (and so $\delta(g) = d_A(g) = k$). Then 
$\chi(g) \approx q^k/(q-\eps)$.
\end{enumerate}
}
\end{examp}

\begin{remar}\label{canonical}
{\em Let $G = \GL^\eps_n(q)$ for $\eps = \pm$ and $S = \SL^\eps_n(q)$. One can define a certain subgroup
$D$ of outer automorphisms of $G$ (in a compatible way for all $n$), such that the action of $G \rtimes D$ on $S$ induces $\Aut(S)$, 
see \cite[(5.2)]{GKNT}. In particular, this gives rise to an action of $D$ on irreducible characters of $G$, $S$, and $\GL^\eps_j(q)$
with $1 \leq j \leq n$. 
Furthermore, $\Gamma := \Gal(\overline{\QQ}/\QQ)$ also acts naturally on those characters. In the case of $G$ and 
$\GL^\eps_j(q)$, these actions are well understood, see e.g. \cite[\S5]{GKNT}. Furthermore, $\tau_n$, respectively $\zeta_n$, 
is $D$-invariant and $\Gamma$-invariant; in particular, $\cl^*(\chi)$ is preserved under the action of $D$ and $\Gamma$. It is 
straightforward now to check that the bijections in Corollaries \ref{gu-dual-bij} and \ref{gl-dual-bij} are $D$-equivariant and 
$\Gamma$-equivariant. Finally, if $\varphi \in \Irr(S)$ has level $j < n/2$, then by Corollary \ref{slu-dual} it lies under a unique 
$\chi \in \Irr(G)$ of true level $j$. Combined with the $D$-equivariance and $\Gamma$-equivariance in the $\GL^\eps$-case, this implies 
that the bijection in Corollary \ref{slu-dual}(ii) is $D$-equivariant and $\Gamma$-equivariant.
}
\end{remar}

Theorems \ref{gl-dual} and \ref{gu-dual} exhibit certain irreducible constituents of $\tau|_{G \times S}$, namely
$\DC_\al \otimes \al$ with $\al \in \Irr(S)$, where $G = \GL^\eps_n(q)$ and $S = \GL^\eps_j(q)$. One may be interested in 
the {\it total} number of irreducible constituents of $\tau|_{G \times S}$, or at least 
$N_{n,j}:= [\tau|_{G \times S},\tau|_{G \times S}]_{G \times S}$.
We will now provide some upper and lower bounds on the latter invariant.

\medskip
First we consider the linear case, and let $S = \GL_j(q) = \GL(U)$, $G = \GL_n(q)=\GL(W)$ with 
$1 \leq j \leq n$, $U = \FF_q^j$, $W = \FF_q^n$.   Set $V=U \otimes W \cong \FF_q^{nj}$. As $\tau$ is the permutation 
character of $\GL(V)$ on the point set of $V$, $N_{n,j}$ is the number of orbits of $S \times G$ acting on $V \oplus V$. 
Note that $V$ is $\Aut(S \times G)$-equivalent to $\Hom(U, W)$ with the natural $S\times G$-action.   Thus,
we are counting orbits on ordered pairs of linear transformations from $U$ to $W$.   These
were classified by Kronecker (see \cite{Ga} and for an elementary treatment see \cite{CP}).  

Now, given a pair $(A,B)$ with $A,B \in \Hom(U, W)$ (or equivalently the pencil $A + xB$),  we can replace $W$ by any subspace of 
$W$ that contains $\Im(A) + \Im(B)$. Hence, the number of orbits $N_{n,j}$ for $n=2j$ is the same as $N_{n,j}$ for any $n \geq  2j$.     

What Kronecker showed is that we can decompose any pencil  $A + xB$ as a direct sum.  More precisely, we can write
$U = \oplus_i U_i$ and $W = \oplus_i W_i$ such that $A(U_i), B(U_i) \subseteq W_i$, and one of the following holds:

\begin{enumerate}[\rm(i)]
\item   $d=\dim U_i=\dim W_i$ and 
\begin{enumerate}  
\item $(A|_{U_i},B|_{U_i})=(I_d,X)$, where $X$ invertible; or
\item $(A|_{U_i},B|_{U_i})=(I_d, Y)$, where $Y$ is nilpotent; or 
\item  $(A|_{U_i},B|_{U_i})=(Y, I_d)$ where $Y$ is nilpotent;
\end{enumerate}
\item  $\dim U_i + 1=   \dim W_i \geq 2$ and there is a unique $(A|_{U_i},B|_{U_i})$ depending only on $\dim U_i$;
\item  $\dim U_i = 1 + \dim W_i \geq 2$ and there is a unique $(A|_{U_i},B|_{U_i})$ depending only on $\dim U_i$;
\item  $(A|_{U_i},B|_{U_i})=(0,0)$.
\end{enumerate}

Here, in (a), (b), respectively (c), $A$, respectively $B$, is represented by the identity matrix $I_d$ in suitable bases of $U_i$ and $W_i$.
Moreover, this decomposition is unique up to the conjugacy of $X$ and $Y$ and the dimensions
of the pieces. We can also combine pieces of the same kind (a), (b), (c), or (iv), so that each of the these four types occurs for at most one index
$i$. 

Now assume that the term of type (a) in the decomposition occurs for  $r = \dim U_i$ (with $0 \leq r \leq j$).
The number of orbits on the first part of the decomposition is just $k(\GL_r(q))$, where $k(H) = |\Irr(H)|$ as usual.    
For the remainder of the decomposition, we write $j-r = a+ b + c + d + e$ as the sum of the totals of $\dim U_i$ for  each of 
the remaining five types. Let $p(m)$ denote the number of partitions of $m$, and let $p'(m)$ be the number of partitions
of $m$ with no parts of size $1$. Then the number of nilpotent classes of $a\times a$-matrices is $p(a)$. 
Next, the contribution of type (ii), respectively (iii), to $N_{n,j}$ is $p(c)$, respectively $p'(d)$. 
It follows for $n \geq 2j$ that $N_{n,j} = F(j,q)$, where
\begin{equation}\label{orbit-gl1}
  F(j,q) := \sum_{r=0}^j   f_{j-r} k(\GL_r(q)),
\end{equation}where 
\begin{equation}\label{orbit-gl2}
  f_m:  = \sum_{a,b,c,d \in \ZZ_{\geq 0},~ a + b + c + d \leq m} p(a)p(b)p(c)p'(d).
\end{equation}  
Note that we have in fact shown that $N_{n,j} \leq F(j,q)$ for any $1\leq j \leq n$.
If $j \geq  n/2$, to find $N_{n,j}$ precisely, the only extra condition required is that we have to see that $\sum_i \dim W_i$
in the decomposition is at most $n$.   One can easily write down the exact formula.
We just note that by ignoring the pieces of type (ii) (where $\dim W_i > \dim U_i$), we
obtain the following lower bound:
$$N_{n,j} \geq \sum_{r=0}^j   h_{j-r} k(\GL_r(q)) \geq  \sum_{r=0}^j k(\GL_r(q)),$$ 
where 
\begin{equation}\label{orbit-gl3}
  h_m:  = \sum_{a,b,d \in \ZZ_{\geq 0},~ a + b + d \leq m}p(a)p(b)p'(d).
\end{equation}  

Next we consider the unitary case and let $S = \GU_j(q) = \GU(U)$, $G = \GU_n(q)=\GU(W)$ with 
$1 \leq j \leq n$, $U = \FF_{q^2}^j$, $W = \FF_{q^2}^n$.   Set $V=U \otimes W \cong \FF_{q^2}^{nj}$. As $\tau^2$ is now 
the permutation character of $\GU(V)$ on the point set of $V$, $N_{n,j}$ is the number of orbits of $S \times G$ acting on $V$. 
We again replace $V$ by $\Hom(U,W)$, and note a couple of easy facts.

We may assume that $n \leq 2j$.   For, we can replace $W$ by any non-degenerate subspace containing the
image of  $T \in \Hom(U,W)$.    
Suppose that $\Ker(T) = 0$ and $\Im(T)$ is non-degenerate. Then $T^*T$, with $T^* = \tw t T^{(q)} = (x_{sr}^q)$ for 
$T=(x_{rs})$, is an invertible Hermitian operator on $U$ and the $S$-conjugacy 
class of the latter is an invariant for the $S \times G$-orbit of $T$. Conversely, if $M$ is any invertible Hermitian 
$j \times j$-matrix over $\FF_{q^2}$, then the corresponding Hermitian form has Gram matrix $I_j$ in a suitable basis of
$U$ and so $M = T^*T$ for a suitable injective $T$ with non-degenerate image. Also note that $M = \tw tM^{(q)}$ is 
$\GL_j(\bar\FF_q)$-conjugate to $M^{(q)}$ and so, by the Lang-Steinberg theorem, $M$ is $\GL_j(\bar\FF_q)$-conjugate to
some $M' \in \GL_j(q)$. By \cite[Theorem 1]{TaZ}, $\tw t M' = AMA^{-1}$ for some  symmetric $A \in \GL_j(q)$, and so
$M'$ is self-adjoint with respect to the Hermitian form with Gram matrix $A$. Finally, two elements of 
$\GL_j(q)$ are $\GL_j(q)$-conjugate precisely when they are $\GL_j(\bar\FF_q)$-conjugate, again by the Lang-Steinberg theorem.
We have therefore shown that the number of  $\GU_j(q)$-conjugacy classes of invertible 
Hermitian $j \times j$-matrices over $\FF_{q^2}$ is at least $k(\GL_j(q))$ (in fact equality holds, see \cite[Lemma 3.1]{FG2}).

We  can do the same thing for $T$ such that both $\Ker(T)$ and $\Im(T)$ are non-degenerate.  
Thus we obtain the lower bound $N_{n,j} \geq \sum_{r=0}^j  k(\GL_r(q))$. 

For $1 \leq j \leq n/2$ we can prove another lower bound for $N_{n,j}$. Note that $\tau|_G = (\zeta_n)^j$ contains $(\zeta_n)^{j-2}$ since 
$(\zeta_n)^2$ is the permutation character of $G$ on the point set of $W$. It follows that $\tau|_G$ contains all 
irreducible characters of true level $j-2i$, $0 \leq i \leq j/2$. Hence Theorem \ref{gu-dual}(i) implies the lower bound
$$N_{n,j} \geq \sum_{0 \leq i \leq j/2}k(\GU_{j-2i}(q)).$$

We summarize our results in the following statement.

\begin{propo}
In the notation of Theorems \ref{gl-dual} and \ref{gu-dual}, let $N_{n,j} := [\tau|_{G \times S},\tau|_{G \times S}]_{G \times S}$. Then 
$$N_{n,j} \geq \sum^j_{r=0}k(\GL_r(q)).$$ 
In fact, in the linear case, i.e. when $(G,S) = (\GL_n(q),\GL_j(q))$, we have  
$$\sum^j_{r=0}h_{j-r}k(\GL_r(q)) \leq N_{n,j} \leq \sum^j_{r=0}f_{j-r}k(\GL_r(q))$$
with $f_m$ and $h_m$ as defined in \eqref{orbit-gl2}, \eqref{orbit-gl3}. In the unitary case, if $1 \leq j \leq n/2$ then
$$N_{n,j} \geq \sum_{0 \leq i \leq j/2}k(\GU_{j-2i}(q)).$$
\end{propo}

In fact, using the Lang-Steinberg theorem one can show that the function $F(j,q)$ defined in \eqref{orbit-gl1} also gives an upper bound 
for $N_{n,j}$ in the unitary case (for a detailed argument see \cite{Gu}).

\end{document}